\documentclass{article}
\usepackage{amsmath, amsthm, amscd, amsfonts, amssymb, graphicx, color,mathrsfs,mathtools,enumerate}
\usepackage[mathcal]{euscript}
\usepackage{yfonts}
\usepackage{lineno}
\usepackage{authblk}
\usepackage{dsfont}
\usepackage{bm}
\usepackage{bbm}
\usepackage[bookmarksnumbered, colorlinks, plainpages]{hyperref}
\hypersetup{colorlinks=true,linkcolor=red, anchorcolor=green, citecolor=cyan, urlcolor=red, filecolor=magenta, pdftoolbar=true} 

\modulolinenumbers[5]
\usepackage{xcolor}

\newtheorem{thm}{Theorem}[section]
\newtheorem{lemma}[thm]{Lemma}
\newtheorem{proposition}[thm]{Proposition}

\newtheorem{remark}[thm]{Remark}
\newtheorem{defn}[thm]{Definition}
\newtheorem{hypothesis}[thm]{Hypothesis}

\newtheorem{example}[thm]{Example}

\theoremstyle{definition}

\theoremstyle{remark}
\newtheorem{rmk}[thm]{Remark}

\numberwithin{equation}{section}

\newcommand{\N}{\mathbb N}
\newcommand{\R}{\mathbb R}

\newcommand{\norm}[1]{\left\lVert#1\right\rVert}
\newcommand{\abs}[1]{\left\lvert#1\right\rvert}
\newcommand{\ps}[2]{\langle#1,#2\rangle}
\newcommand{\dscal}[3]{{_{#3}\left\langle #1,#2\right\rangle}_{{#3}^*}}

\allowdisplaybreaks

\title{\textbf{Stochastic and deterministic non-autonomous reaction-diffusion equations}}
\author[$\star$]{Davide A. Bignamini}
\author[$\dagger$]{Paolo De Fazio}
\date{}

\affil[$\star$]{\small{corresponding author, Dipartimento di Scienze e Alta Tecnologia, Università degli studi dell'Insubria, email-address: da.bignamini@uninsubria.it}}
\affil[$\dagger$]{\small{Dipartimento di Scienze e Alta Tecnologia, Università degli studi dell'Insubria, email-address: paolo.defazio@uninsubria.it}}

\providecommand{\keywords}[1]{{\textit{Keywords}:} #1}
\providecommand{\subjclass}[1]{{\textit{2020 Mathematics Subject Classification}:} #1}

\allowdisplaybreaks

\begin{document}

\maketitle


\begin{abstract}
\noindent
In this paper we prove the well-posedness of non-autonomous deterministic and stochastic reaction-diffusion equations with a polynomial reaction term. In particular, we exploit a well-posedness result for a non-autonomous deterministic PDE to prove the well-posedness of stochastic reaction-diffusion equations driven by both additive and multiplicative noise, with a non-autonomous diffusion coefficient that is Lipschitz continuous in space.\\\phantom{a}\\
\\
\end{abstract}


\providecommand{\vecchio}[1]{{ {}#1}}
\providecommand{\nuovo}[1]{{  {}#1}}

\section{Introduction}
This paper investigates abstract non-autonomous deterministic and stochastic reaction-diffusion equations with a polynomial reaction term. Reaction-diffusion systems model a wide range of phenomena in various fields such as biology, ecology, physics, neuroscience, and nanobioscience. Deterministic models are often used to describe processes such as predator-prey dynamics, particle propagation, and birth-death mechanisms (see, for instance, \cite{ceb_par_rui_2024,ceb_par_rui_2025,hor_2007,men_pin_che_aky_2014,nak_moo_wei_fan_vas_shu_2012,nak_yam_1996}). The stochastic counterparts are particularly suited for modeling intrinsic randomness in many phenomena, such as metabolic processes, population dynamics, and organism-level movements (see, for example, \cite{dos_suc_hol_kat_min_2013,kaw_sai_2006,mas_1988}).
In most existing literature, the diffusion part (typically a second-order differential operator) is assumed to have time-independent coefficients. However, in many realistic scenarios, it is natural to consider models where the diffusion explicitly depends on time (see for instance \cite{van_rob_2020}). The main goal of this work is to establish well-posedness results for a class of abstract, fully non-autonomous reaction-diffusion equations, thus covering a wide range of time-dependent models.\\

 Let $T>0$. We set $\Delta=\{(s,t)\in\R^2\ |\ 0\leq s < t\leq T\}$ and $\overline{\Delta}=\{(s,t)\in\R^2\ |\ 0\leq s \leq t\leq T\}$. Let $\{A(t)\}_{t\in [0,T]}$ be a family of closed linear operators on a real separable Banach space $E$ with possibly time-dependent and not necessarily dense domains, such that $A(t)$ is the generator of an analytic semigroup for all $t\in[0,T]$. Given   $s\in [0,T)$, we consider the following abstract non-autonomous parabolic problem
\begin{align}\label{NAP}
\begin{cases}
u'(t)=A(t)u(t)+f(t),\qquad t\in(s,T],\\
u(s)=x\in E,
\end{cases}
\end{align}
where $f:[0,T]\rightarrow E$ is a smooth enough function. 

The well-posedness of the problem \eqref{NAP} has been widely studied in the literature; see e.g. \cite{acq_1988,acq_ter_1987,paz_1983,sob_1961,tan_1979,yag_1976} and the bibliography therein. In these papers the authors defined the evolution operator $\{U(t,s)\}_{(s,t)\in\overline{\Delta}}$ associated to \eqref{NAP}; see Subsection \ref{sect: absevolproblem} for a summary about the main features of $\{U(t,s)\}_{(s,t)\in\overline{\Delta}}$ with particular attention to non-autonomous parabolic problems. This paper aims to investigate deterministic and stochastic perturbations of equation \eqref{NAP}.

In Section \ref{sect: WD} we study a semilinear version of \eqref{NAP} given by
\begin{align}\label{NAPP}
\begin{cases}
u'(t)=A(t)u(t)+F(t,u(t)),\qquad t\in(s,T],\\
u(s)=x\in E.
\end{cases}
\end{align}
Typically such equations are settled in a real separable Hilbert space $H$. However in significant examples, such as Example \ref{ex:intro} below, most nonlinearities are not well defined (or not regular) from $[0,T]\times H$ to $H$. Therefore we assume that $E$ is continuously embedded in $H$ and $F: [0,T]\times E\rightarrow H$ is Borel measurable. Under a suitable monotonicity assumption on $F$ and using the properties of $\{U(t,s)\}_{(s,t)\in\overline{\Delta}}$, we prove existence, uniqueness, and a priori estimates for the mild solution to \eqref{NAPP}. One of the most significant examples covered by the abstract Theorem \ref{proyos} is given by the following reaction-diffusion type equations.
\begin{example}[Non-autonomous reaction-diffusion equations]\label{ex:intro}
\leavevmode
\begin{enumerate}[(i)]
    \item  $\mathcal{O}$ is a bounded open subset of $\R^d$ with $C^2$ boundary and
$d\in\N$ (see Remark \ref{RegO} for special cases in which a Lipschitz
boundary suffices).

\item We choose $H=L^2(\mathcal{O})$ and $E=C(\overline{\mathcal{O}})$.

    \item For every $t\in[0,T]$ the operator $A(t)$ is a realization (with suitable boundary conditions) in $L^2(\mathcal{O})$ of the following elliptic second order non-autonomous differential operator
\begin{equation*}
\mathcal{A}(t)\varphi(\xi)=\sum_{i,j=1}^d D_i\left(a_{ij}(t,\xi)D_{j}\varphi(\xi)\right)+a_0(t,\xi)\varphi(\xi),\quad t\in [0,T],\ \xi\in \mathcal{O},
\end{equation*}
and the boundary operator $\mathcal{C}(t)$ given by
\begin{align*}
\mathcal{C}(t)u=
\begin{cases}
\begin{aligned}
&u \ \ \quad \quad\quad \quad \quad\quad\quad \quad \quad \quad\quad \quad \quad \quad \mbox{(Dirichlet)},  \\[1ex]
&\displaystyle{\sum_{i,j=1}^da_{ij}(t,\xi)D_i u\,\nu_j}\quad \quad\quad\quad \quad \quad\, \ \ \mbox{(Neumann)},
\end{aligned}
\end{cases}
\end{align*}
where  $\nu=(\nu_1,...,\nu_d)$ is the unit outer normal vector at the boundary of $\mathcal O$, the coefficients $a_{ij}$ and $a_0$ are smooth enough functions.
    \item The function $F$ is a Nemytskii operator defined by
\[
F(t,x)(\xi)=-C_{2m+1}(t,\xi)x(\xi)^{2m+1}+\sum_{k=0}^{2m}C_{k}(t,\xi)x(\xi)^{k},
\]
where $m\in\N$, $C_0,...,C_{2m+1}:[0,T]\times\overline{\mathcal{O}}\longrightarrow\R$ are continuous and bounded functions and there exists a constant $c>0$ such that $\displaystyle{\inf_{(t,\xi)\in[0,T]\times\overline{\mathcal{O}}}C_{2m+1}(t,\xi)>c}$.
\end{enumerate}
See Theorem \ref{esempio_teorema} and Remark \ref{RegO} for more details about the assumptions on $\mathcal{O}$, $A(t)$ and $F$.
\end{example}

 We are not aware of any abstract well-posedness result for mild solutions to \eqref{NAPP} which, like Theorem \ref{proyos}, holds in the genuinely non-autonomous case, i.e. with time-dependent and possibly non-dense domains $D(A(t))$, and with a nonlinearity $F$ which is merely dissipative with polynomial growth, as in Example \ref{ex:intro}. For concrete equations of this type with smooth coefficients and data, classical solutions can also be obtained by PDE methods (see, e.g., \cite[Chapter V]{lad_sol_ura_1968}); here, instead, we need an abstract mild formulation in the space $E$, since this is the one used below to treat the stochastic problems, for which classical solutions are not available.\\

    In Section \ref{sect:WS} we study the following stochastic perturbation of \eqref{NAPP}
\begin{equation}\label{stocpbintro}
\begin{cases}
dX(t)=[A(t)X(t)+F(t,X(t))]dt+B(t)dW(t), & t\in(s,T],\\
X(s)=x\in H,
\end{cases}
\end{equation}
where $\{W(t)\}_{t\in[0,T]}$ is an $H$-cylindrical Wiener process and $\{B(t)\}_{t\in[0,T]}$
is a family of linear bounded operators on $H$. We establish the existence and uniqueness of
a global generalized mild solution to \eqref{stocpbintro} (see Theorems \ref{solMild} and
\ref{limmild}) exploiting the well-posedness results for \eqref{NAPP} (see Theorem
\ref{proyos}). For the analogous results in the autonomous case (namely $A(t)\equiv A$ and
$B(t)\equiv B$) we refer to \cite{big_2021}, \cite[Chapters 6 and 7]{cer_2001}, and
\cite[Chapter 7]{dap_zab_2014}. In particular, in Theorems \ref{solMild} and \ref{limmild} we
provide some useful pathwise estimates for the solution. Moreover, in Section
\ref{sect:mult} we show that the deterministic results of Section \ref{sect: WD} can also be
exploited to study the following class of stochastic problems on $E=C(\overline{\mathcal{O}})$ driven by a multiplicative noise
\begin{equation}\label{stocpbintroM}
\begin{cases}
dX(t)=[A(t)X(t)+F(t,X(t))]dt+B(t,X(t))dW(t), & t\in(s,T],\\
X(s)=x\in C(\overline{\mathcal{O}}),
\end{cases}
\end{equation}
where $B(t,\cdot)$ is Lipschitz continuous uniformly with respect to $t$.

The literature on abstract non-autonomous stochastic equations of the type \eqref{stocpbintro} and \eqref{stocpbintroM} is still rather limited compared with the autonomous case. For the case $F\equiv 0$ we refer to \cite{add_def_2025,big_def_2024,cer_lun_2025,def_2022,kna_2011,ver_zim_2008}, while for the semilinear case we refer to \cite{BlessingSeitzSonnerTang2025,fan_2015,man_zau_1999,pro_ver_2014,san_vui_2002,san_vui_2003,san_vui_2009,seid_1993,ver_2010}. We stress that the main novelty of our work is an abstract global well-posedness result in the mild sense for \eqref{stocpbintro} and \eqref{stocpbintroM}, in which $F$ is only required to be dissipative up to a linear term and to satisfy local Lipschitz-type conditions, without any linear growth condition (see Hypotheses \vecchio{\ref{EU2}}\nuovo{ \ref{Hyp1}, \ref{Hyp2}} and \ref{Hyp-F-mult} for more details). In particular, unlike the semilinear results quoted above, the abstract result of the present paper yields global well-posedness for the stochastic version of Example \ref{ex:intro}, both with additive and multiplicative noise. The technique that we use in both cases is different from those of the above-mentioned papers; it is a non-autonomous version of the approach of \cite{cer_2001,cer_2003} and \cite[Chapter 7]{dap_zab_2014}, based on the well-posedness result for the deterministic equation \eqref{NAPP}. Finally, we underline that we study the additive case (Section \ref{sect:WS}) and the multiplicative case (Section \ref{sect:mult}) separately, since in the additive case we can work under more general assumptions.

We now compare in more detail the results of the present paper with \cite{BlessingSeitzSonnerTang2025,fan_2015,man_zau_1999,pro_ver_2014,san_vui_2002,san_vui_2003,san_vui_2009,seid_1993,ver_2010}.\\
In \cite{man_zau_1999} the existence of solutions is proved by comparison methods, but uniqueness is proved only for bounded nonlinearities.\\
The papers \cite{san_vui_2002,san_vui_2003} are among the first well-posedness contributions on non-autonomous
semilinear SPDEs with time-dependent boundary conditions, and they were later improved in
\cite{ver_2010} with other techniques; a fractional-noise counterpart of these results is
studied in \cite{san_vui_2009}. In \cite{san_vui_2002,san_vui_2003,san_vui_2009} the drift-nonlinearity is assumed to be
Lipschitz continuous; we also refer to Remark \ref{rmk:conf-sanz} for a comparison between the noise of the present paper and the one in  \cite{san_vui_2002,san_vui_2003}.\\
Also, in \cite{fan_2015}  $F$ is a continuous map from $H$ into $H$ satisfying growth and continuity conditions, which include the linear growth and global Lipschitz conditions as a special case (see \cite[Remark 3.1]{fan_2015}).\\
In \cite{pro_ver_2014,ver_2010} a more general framework is considered in several respects. Indeed,
$\{A(t)\}_{t\geq 0}$, $F$ and $B$ are allowed to be random, and the
equation is settled in a UMD space $E$ with type 2. However, the Lipschitz-type
assumption (H2) in \cite[Theorem 1.1]{ver_2010}, as well as (HF) in \cite{pro_ver_2014}, does not allow $F$ to have polynomial
growth; in \cite[Section 7]{ver_2010}, which deals with the local Lipschitz case, only local
mild solutions on a maximal interval $[0,\tau)$ are obtained (see also \cite[Theorem 1.5]{seid_1993} in Hilbert spaces with time-independent domains), whereas global solutions on
$[0,T]$ are considered only when $F$ satisfies an additional linear growth condition (see
\cite[Theorem 7.2 and Example 8.9]{ver_2010}), so that well-posedness for Example \ref{ex:intro} is not
proved. Moreover we stress that in \cite{ver_2010} the author cannot work directly on $E=C(\overline{\mathcal{O}})$,
since $C(\overline{\mathcal{O}})$ is not UMD; indeed in \cite[Example 8.9]{ver_2010}, the local solution is constructed in Besov spaces continuously embedded in $C(\overline{\mathcal{O}})$, for initial data in $B^{a}_{p,p}(\mathcal{O})$ with $p>2d$ and $a\in(d/p,\frac12)$, whereas in the present paper initial data in $C(\overline{\mathcal{O}})$ are allowed.\\
In \cite{BlessingSeitzSonnerTang2025} global pathwise mild solutions are obtained for stochastic reaction-diffusion equations with random non-autonomous generators, driven by an additive noise (see also \cite{KNS21} for the globally Lipschitz case). However, unlike here, the generators have time-independent domains, the noise is of trace class with an autonomous diffusion coefficient (namely, in \eqref{stocpbintro}, $B(t)\equiv B$ is a Hilbert--Schmidt operator on $H$) and, above all, $F$ is subject to a dimension-dependent growth restriction (which, for example, excludes the cubic nonlinearities of Example \ref{ex:intro} for $d\geq 3$). Moreover, although the noise is additive, initial data in $L^2(\mathcal{O})$ are not allowed, whereas Theorem \ref{limmild} provides generalized mild solutions for every initial datum in $L^2(\mathcal{O})$.\\
In conclusion, to the best of our knowledge, the abstract result presented in this paper is the first global well-posedness result in the mild sense that applies to stochastic non-autonomous reaction-diffusion equations of Example \ref{ex:intro} both in the additive case (see Section \ref{sect:exa}) and in the multiplicative case (see Section \ref{sect:mult}).\\
We also mention some other interesting results in other settings. In \cite{cer_lun_2017} slow-fast systems are studied for operators of the form $\gamma(t)A+L(t)$, with $L(t)$ of first order, and with time-independent boundary conditions; for non-autonomous reaction-diffusion equations with fixed domains and transport noise, see \cite{agr_ver_2023,agr_ver_2024}. For the variational approach to equations with monotone
drift we refer to
\cite{kry_roz_1981,liu_roc_2010,liu_roc_2015,par_1979,pre_roc_2007,roc_sha_zha_2024}. There, time-dependent coefficients are allowed, but another type of solution is
studied rather than the mild ones considered here. Moreover, we recall that, in general, in the variational approach the diffusion coefficient is required to be Hilbert--Schmidt.\\

 We continue the introduction by highlighting some significant differences between the
study of the autonomous and non-autonomous cases.
 
    Let $s\in[0,T)$. We stress that to deduce the well-posedness of \eqref{stocpbintro} from
the well-posedness of \eqref{NAPP}, we need a result of regularity for the stochastic
convolution process $\{Z_s(t)\}_{t\in[s,T]}$ defined by
\[
Z_s(t):=\int_s^t U(t,r)B(r)dW(r),\qquad t\in[s,T].
\]
Many authors studied the properties of the process $\{Z_s(t)\}_{t\in[s,T]}$, see for
instance \cite{dap_ian_tub_82,fan_2015,pro_ver_2014,seid_1993,seid_2003,tub_82,van_ver_2020,ver_zim_2008,ver_2010}.
In particular, in \cite[Section 4]{ver_2010} the factorization method of Da Prato,
Kwapie\'n and Zabczyk is extended to the non-autonomous setting (see also \cite{seid_1993} for time-independent domains) and explicit conditions on
the noise are given in some cases, see for instance \cite[Examples 8.2 and 8.6]{ver_2010} (see also \cite[hypothesis (C$^d$)]{san_vui_2003}).
However, in \cite[Examples 8.2 and 8.6]{ver_2010} and \cite[hypothesis (C$^d$)]{san_vui_2003} the driving process is an
$L^2(\mathcal{O})$-valued Brownian motion, whose covariance operator is fixed
and, in particular, does not depend on time. Here, instead,
$\{W(t)\}_{t\in[0,T]}$ is an $L^2(\mathcal{O})$-cylindrical Wiener process, and
we provide explicit conditions on $\{A(t)\}$ and on 
$\{B(t)\}_{t\in[0,T]}$ ensuring the regularity in $C(\overline{\mathcal{O}})$,
similarly to what is done in the autonomous case; see, for instance,
\cite[Section 2.2.2]{dap_2004}, \cite[Sections 4.1 and 4.2]{orr_fuh_16} and \cite{brz_1995,brz_1997,brz_pes_1999}.
 In Proposition
\ref{prop:reg-convu} we provide conditions on the family
$\{B(t)\}_{t\in[0,T]}$, expressed only in terms of the boundedness of $B(t)$ from
$L^2(\mathcal{O})$ to $L^q(\mathcal{O})$ with $q\geq 2$ and $q>d$, uniformly in $t\in[0,T]$, which
guarantee that $\{Z_s(t)\}_{t\in[s,T]}$ has paths in $C([s,T];C(\overline{\mathcal{O}}))$.
Our proof relies on estimates for the kernel of the evolution operator
$\{U(t,s)\}_{(s,t)\in\Delta}$, in the same spirit as the Green function estimates
exploited in \cite{san_vui_2003}.
We emphasize that in the autonomous case mild solutions to stochastic
dissipative systems of the type \eqref{stocpbintro} have been extensively studied in the
literature using semigroup theory; see, for instance
\cite{big_2021,big_fer_2024,cer_2001,cer_dap_2012,cer_dap_fla_2013,dap_zab_2014,hai_2002,hai_jon_2011}
and the references therein. In the non-autonomous case, semigroup theory is replaced by the
theory of evolution operators, which is significantly less developed. This leads to
substantial complications in the analysis of \eqref{NAPP}, \eqref{stocpbintro} and of the
process $\{Z_s(t)\}_{t\in[s,T]}$.

\section{Preliminaries}
In this section, we state some basic definitions and results that will
be essential in the sequel

\subsection{Notations}
Let $(V,\norm{\,\cdot\,}_V)$ and $(W,\norm{\,\cdot\,}_W)$ be two Banach spaces. We denote by $\mathscr{L}(V;W)$ the space of linear continuous operators from $V$ to $W$. If $W=\R$ we set $V^*=\mathscr{L}(V;\R)$. We denote by $I$ the identity operator on $V$ and by $\mathcal{B}(V)$ the Borel $\sigma$-field of $V$. 

$B(V;W)$ is the space of Borel measurable functions from $V$ to $W$ and $B_b(V;W)$ is the subset of $B(V;W)$ consisting of bounded functions. For every interval $I\subseteq\R$ we denote by $C_b(I;V)$ the set of continuous bounded functions from $I$ to $V$ and, for every $k\in\N$, we denote by $C^k_b(I;V)$ the set of bounded and $k$-times differentiable functions with bounded and continuous derivatives up to order $k$. If $I$ is compact, we simply write $C(I;V)$ and $C^k(I;V)$. 

Let $(E,\norm{\,\cdot\,}_E)$ be a Banach space continuously embedded in $V$ and let $G: {\rm D}(G)\subseteq V{\longrightarrow} V$ be a function. We call part of $G$ in $E$ the function $G_{E}:{\rm D}(G_{E})\subseteq E{\longrightarrow} E$ defined by
\begin{align*}
{\rm D}(G_{E})&:=\{x\in {\rm D}(G)\cap E\; \mbox{s.t.}\; G(x)\in E \},\\& G_{E}(x):=G(x),\; \mbox{for all}\;x\in {\rm D}(G_{E}).
\end{align*}
For $B:{\rm D}(B)\subseteq V\rightarrow V$ we denote by $\rho(B)$ the resolvent set of $B$ and by $R(\lambda,B):=(\lambda I-B)^{-1}$ the resolvent of $B$ for $\lambda\in\rho(B)$.

Let $(\Omega,\mathcal{F},\mathbb{P})$ be a complete probability space. For a random variable $\xi:(\Omega,\mathcal{F},\mathbb{P}){\longrightarrow} (V,\mathcal{B}(V))$, we denote by $\mathbb{E}[\xi]$ the expectation of $\xi$ with respect to $\mathbb{P}$.

Let $(\Omega,\mathcal{F},\{\mathcal{F}_t\}_{t\geq 0},\mathbb{P})$ be a complete filtered probability space. For a random variable  \\ $\xi:(\Omega,\mathcal{F},\mathbb{P}){\longrightarrow} (V,\mathcal{B}(V))$, we denote by $\mathbb{E}[\xi]$ the expectation of $\xi$ with respect to $\mathbb{P}$.\\
Let $I\subseteq\R$ and let $\{Z(t)\}_{t\in I}$ be a $V$-valued stochastic process. We say that $\{Z(t)\}_{t\in I}$ is continuous if the map $Z(\cdot):I {\longrightarrow} V$ is continuous $\mathbb{P}$-almost surely.\\
For all $p\geq 1$ we denote by $\mathcal{C}^p(I;V)$ the space of $V$-valued continuous stochastic processes $\{Z(t)\}_{t\in I}$, adapted to $\{\mathcal{F}_t\}_{t\geq 0}$, such that
\[
\norm{Z(\cdot)}_p:=\left(\mathbb{E}\left[\sup_{t\in I}\norm{Z(t)}_V^p\right]\right)^{1/p}<+\infty.
\]
We notice that $(\mathcal{C}^p(I;V),\norm{\,\cdot\,}_p)$ is a complete Banach space, for every $p\geq 1$. Indeed, { if a sequence is a Cauchy sequence with respect to $\norm{\,\cdot\,}_p$, then a subsequence converges $\mathbb{P}$-a.s. uniformly on $I$ to a continuous process, which is adapted} since the filtration is complete{ , and by the Fatou lemma the whole sequence converges to it with respect to $\norm{\,\cdot\,}_p$}.

\subsection{Left differentiable and dissipative mappings} 
We refer to \cite[Appendix A]{cer_2001}, \cite[Appendix D]{dap_zab_2014} for all the results in this section.

Let $(V,\norm{\,\cdot\,}_V)$ be a separable Banach space. For every $x\in V$, we define the sub-differential $\partial \norm{x}_V$ of $\norm{\,\cdot\,}_V$ at $x\in V$ as
\[
\partial \norm{x}:=\{ x^*\in V^*\; |\; \norm{x+y}_V\geq\norm{x}_V+x^*(y),\;\forall\; y\in V\}.
\]
Moreover $\partial \norm{x}_V$ is closed and convex and for every $x\neq 0$ we have
\[
\partial \norm{x}=\{ x^*\in V^*\; |\,~_{V}\ps{x}{x^*}_{V^\star}=\norm{x}_V,\; \norm{x^*}_{V^*}=1\}.
\]

Let $t_0,t_1 \in\R$ and let $f:[t_0,t_1]{\longrightarrow}V$ be a continuous function. We say that $f$ is right-differentiable at $t\in [t_0,t_1)$ if there exists $L^+\in\R$ such that
\begin{equation}\label{Llimit+}
\dfrac{d^+f(t)}{dt}:=\lim_{\varepsilon{\rightarrow} 0^+}\dfrac{f(t+\varepsilon)-f(t)}{\varepsilon}=L^+.
\end{equation}
We say that $f$ is right-differentiable if $f$ is right-differentiable at $t$ for all $t\in [t_0,t_1)$.
Moreover $f$ is left-differentiable if for every $t\in (t_0,t_1]$ there exists $L^-\in\R$ such that
\begin{equation}\label{Llimit-}
\dfrac{d^-f(t)}{dt}:=\lim_{\varepsilon{\rightarrow} 0^-}\dfrac{f(t+\varepsilon)-f(t)}{\varepsilon}=L^-.
\end{equation}
We say that $f$ is left-differentiable, if $f$ is left-differentiable at $t$ for all $t\in (t_0,t_1]$.
If $L^+=L^-$ at some $t\in(t_0,t_1)$ then we say that $f$ is differentiable at $t$ and if  $L^+=L^-$ for all $t\in(t_0,t_1)$
we say that $f$ is differentiable in $(t_0,t_1)$.

Let $u:[t_0,t_1]{\longrightarrow}V$ be a differentiable function in $r$. By \cite[Proposition A.1.3]{cer_2001}, the function $\gamma:[t_0,t_1]{\longrightarrow} [0,+\infty)$ defined by $\gamma(t)=\norm{u(t)}_V$ is left-differentiable and right-differentiable in $r$.
Further, we have 
\begin{align}
\label{Ldiff+}
&\dfrac{d^+\gamma}{dt}(r)=\max\{~_{V}\ps{u'(r)}{x^*}_{V^\star}\;:\; x^*\in\partial\norm{u(r)}_V \},\ \ r\in[t_0,t_1),\\
\label{Ldiff-}
&\dfrac{d^-\gamma}{dt}(r)=\min\{~_{V}\ps{u'(r)}{x^*}_{V^\star}\;:\; x^*\in\partial\norm{u(r)}_V \},\ \ r\in(t_0,t_1].
\end{align}

Moreover, let $b\in\R$ and let $g:[t_0,t_1]{\longrightarrow} [0,+\infty)$ be a continuous function. If
\[
\dfrac{d^-\gamma}{dt}(t)\leq b\gamma(t)+g(t), 
\]
then, by \cite[Proposition A.4]{big_fer_2024}, for every $t\in [t_0,t_1]$, we have
\begin{equation}\label{varofcost}
\gamma(t)\leq e^{b(t-t_0)}\gamma(t_0)+\int^t_{t_0}e^{b(t-s)}g(s)ds,\quad t\in [t_0,t_1].
\end{equation} 

\begin{defn}\label{dissip}
A map $f:{\rm D}(f)\subseteq V{\longrightarrow}V$ is said to be dissipative if, for every $\alpha>0$ and $x,y\in {\rm D}(f)$, we have
\begin{equation}\label{disban1}
\norm{x-y-\alpha(f(x)-f(y))}_{V}\geq \norm{x-y}_{V}.
\end{equation}
We say that $f$ is m-dissipative if the range of $\lambda I-f$ is the whole space $V$ for some $\lambda>0$ (and so for all $\lambda>0$).
\end{defn}

\begin{remark} If in Definition \ref{dissip} $f=A$ for some linear operator $A$, then \eqref{disban1} reads as 
\[
\norm{(\lambda I-A)x}_{V}\geq \lambda \norm{x}_{V}, \quad \forall \lambda>0,\; x\in{\rm D}(A).
\]

\end{remark}
Using the notion of sub-differential we have the following useful characterization of dissipative mappings.

\begin{proposition}
Let $f:{\rm D}(f)\subseteq V\longrightarrow V$. $f$ is dissipative if and only if, for every $x,y\in {\rm D}(f)$ there exists $z^*\in\partial\norm{x-y}_V$ such that
\begin{equation}\label{disban}
\dscal{f(x)-f(y)}{z^*}{V}\leq 0.
\end{equation}
If $V$ is a Hilbert space endowed with the scalar product $\ps{\cdot}{\cdot}_V$, \eqref{disban} reads as
\begin{equation*}
\ps{f(x)-f(y)}{x-y}_{V}\leq 0.
\end{equation*}
\end{proposition}

\begin{remark}\label{DisEN}
Let $B:{\rm D}(B)\subseteq V\rightarrow V$ be a linear operator. By \cite[Proposition 3.14]{eng_nag_2006} if $B$ is m-dissipative then $(0,+\infty)\subseteq \rho(B)$.
\end{remark}

\subsection{The Yosida approximations for dissipative functions}\label{sec:Yosida}

In this section we introduce a useful smoothing sequence for dissipative functions (see \cite[Appendix A]{cer_2001}, \cite[Appendix D]{dap_zab_2014}).

Let $(V, \norm{\,\cdot\,}_V)$ be a separable Banach space and let $F: [0,T]\times {\rm D}_{F}\subseteq [0,T]\times V\longrightarrow V$ be such that $F(t,\cdot)$ is a possibly non linear function. We assume that there exists $\zeta\in\R$ such that 
\begin{equation}\label{Umdiss}
F(t,\cdot)-\zeta I\;\; {\mbox{is m-dissipative for all}}\;t\in[0,T].
\end{equation}
By the m-dissipativity, for every $k\in\N$ the map $$y\longrightarrow  y-\frac{1}{k}(F(t,y)-\zeta  y)$$ is bijective for all $t\in [0,T]$. Hence for every $t\in[0,T]$ and for every $x\in V$, there exists a unique $J_k(t,x)\in {\rm D}_{F}$ such that
\begin{align}\label{eq_YO1}
J_k(t,x)-\frac{1}{k} [F(t,J_k(t,x))-\zeta J_k(t,x)]=x.
\end{align}
We define $F_k:[0,T]\times V\longrightarrow V$ as
\begin{align}\label{defyod}
F_k(t,x):=F(t,J_k(t,x)),\qquad x\in V,\ t\in[0,T],\ k\in\N.
\end{align}

\begin{defn}
Let $F: [0,T]\times {\rm D}_{F}\subseteq [0,T]\times V\longrightarrow V$ and let $\zeta\in\R$ be such that $F(t,\cdot)-\zeta I$ is m-dissipative for all $t\in[0,T]$. We denote by $\{F_k\}_{k\in \N}$ the family of the Yosida approximations of $F$ defined in \eqref{defyod}.  
\end{defn}

The proof of the following lemma is standard in the case where $F$ is independent of $t$. Although here $F$ also depends on time, the constant $\zeta$ in \eqref{Umdiss} is uniform with respect to $t$. Due to this reason, the calculations are quite similar to the autonomous case. However, we have decided to include the proof for the reader's convenience. If $\zeta=0$ we set $\displaystyle{\frac{1}{\abs{\zeta}}=+\infty}$.

\begin{lemma}\label{Lemma_YO}
Let $F: [0,T]\times {\rm D}_{F}\subseteq [0,T]\times V\longrightarrow V$ and let $\zeta\in\R$ be such that $F(t,\cdot)-\zeta I$ is m-dissipative for all $t\in[0,T]$. Let $J_k:[0,T]\times V\longrightarrow{\rm D}_{F}$ be the mapping satisfying \eqref{eq_YO1}. Then for all $t\in [0,T]$
\begin{align}
\lim_{k{\rightarrow} +\infty}\norm{J_{k}(t,x)-x}_{V}=0,\quad x\in {\rm D}_{F}.\label{conveyos}
\end{align}
For every  $k>\abs{\zeta}$and $t\in[0,T]$ the function $F_k(t,\cdot)-\zeta I$ is dissipative and we have
\begin{align}
\norm{J_k(t,x)-x}_{V}&\leq \frac{1}{k}\left(\norm{F(t,x)}_V+\abs{\zeta}\norm{x}_V\right),\qquad\qquad\qquad \ \ x\in {\rm D}_{F},\label{cji}\\
\|F_{k}(t,x)\|_{V}&\leq 3(\norm{F(t,x)}_{V}+ \abs{\zeta}\norm{x}_{V}),\phantom{aaaaaaaaaaaaaa}x\in {\rm D}_{F},\label{vy1}\\
\norm{F_k(t,x)-F_k(t,y)}_V&\leq (2k+\abs{\zeta})\norm{x-y}_V, \phantom{aaaaaaaaaaaaaaaaaa}\,x,y\in{V},\label{lipdeltaV}\\
\norm{J_k(t,x)-J_k(t,y)}_{V}& \leq \norm{x-y}_V,\phantom{aaaaaaaaaaaaaaaaaaaaaaaaaa}x,y\in{V}\label{lipJ}.
\end{align}
For every $x,y\in {\rm D}_{F}$, $k,h>\abs{\zeta}$ and $t\in[0,T]$ we have
\begin{equation}\label{superyE}
\norm{J_k(t,x)-J_h(t,y)}_{V}\leq\norm{x-y}_{V}+ \left(\frac{1}{k}+\frac{1}{h}\right)(\norm{F(t,y)}_{V}+\norm{F(t,x)}_{V}+\abs{\zeta}\norm{y}_{V}+\abs{\zeta}\norm{x}_{V}).
\end{equation} 
If in addition ${V}$ is a Hilbert space, then  there exists $C_\zeta>0$ such that for every $x,y\in {\rm D}_{F}$ and $k,h>\abs{\zeta}$ we have
{\small\begin{equation}\label{supery}
\ps{F_{k}(t,x)-F_h(t,y)}{x-y}_{V}\leq \abs{\zeta}\norm{x-y}_{V}^2+C_\zeta\left(\frac{1}{k}+\frac{1}{h}\right)(\norm{F(t,x)}_{V}^2+\norm{F(t,y)}_{V}^2+\norm{x}_{V}^2+\norm{y}_{V}^2).
\end{equation}}
\end{lemma}

\begin{proof}
Let $k>\abs{\zeta}$. We set
\begin{align}
G(t,x)&:=F(t,x)-\zeta  x, \qquad t\in[0,T],\ x\in{\rm D}_{F},\label{defnG}\\
G_{k}(t,x)&:=G(t,J_k(t,x)), \qquad t\in[0,T],\ x\in V,
\end{align}
and by \eqref{eq_YO1} we have
\begin{align}\label{eq_YO}
J_k(t,x)-\frac{1}{k} G_k(t,x)=x.
\end{align}
We prove first that $G_{k}(t,\cdot)$ is dissipative on ${V}$ for every $t\in[0,T]$ and
\begin{align}
\norm{J_k(t,x)-x}_{V}&\leq \frac{1}{k}\norm{G(t,x)}_{V}, \qquad x\in {\rm D}_{F},\ t\in[0,T],\label{G1}\\
\|G_{k}(t,x)\|_{V}&\leq \|G(t,x)\|_{V},\qquad\ \  \, x\in {\rm D}_{F},\ t\in[0,T],\label{G2}\\
\norm{G_k(t,x)-G_k(t,y)}_V&\leq 2k\norm{x-y}_V,\qquad \ x,y\in{V},\ t\in[0,T],\label{G3}\\
\norm{J_k(t,x)-J_k(t,y)}_{V} &\leq \norm{x-y}_V,\qquad \ \ \ \ x,y\in{V},\ t\in[0,T]\label{G4}.
\end{align}
For every $t\in[0,T]$ and $x\in{\rm D}_{F}$, by dissipativity of $G(t,\cdot)$ (i.e. \eqref{disban1} with $\alpha=k^{-1}$) we have
\begin{align*}
\norm{x-J_k(t,x)}_{V}\leq \norm{x-J_k(t,x)-\frac{1}{k}\left[G(t,x)-G_k(t,x)\right]}_{V}.
\end{align*}
So by the definition of $G_k$ and \eqref{eq_YO} we obtain \eqref{G1} and \eqref{G2}. 
For every $t\in[0,T]$ and $x,y\in {V}$, by dissipativity of $G(t,\cdot)$ we have
\begin{align}\label{II1}
\norm{J_k(t,x)-J_k(t,y)}_{V}\leq \norm{J_k(t,x)-J_k(t,y)-\frac{1}{k}[G_k(t,x)-G_k(t,y)]}_{V}.
\end{align}
By the definition of $G_k$, \eqref{eq_YO} and \eqref{II1} we obtain \eqref{G4} and
\begin{align*}
\frac{1}{k}\norm{G_k(t,x)-G_k(t,y)}-\norm{x-y}_{V}\leq\norm{\frac{1}{k}(G_k(t,x)-G_k(t,y))+x-y}_{V}\leq \norm{x-y}_{V},
\end{align*}
that yields \eqref{G3}. For every $t\in[0,T]$ and $x,y\in {V}$ and $\alpha>0$, by \eqref{eq_YO} and \eqref{G4} we have
\begin{align*}
\norm{x-y-\alpha(G_k(t,x)-G_k(t,y))}_{V}=\norm{\left(1+\alpha k\right)(x-y)-\alpha k(J_k(t,x)-J_k(t,y))}_{V}\geq \norm{x-y}_{V}
\end{align*}
so $G_k(t,\cdot)$ is dissipative. Since ${\rm D}(F)={\rm D}(G)=[0,T]\times{\rm D}_{F}$, then \eqref{G1}, \eqref{G2}, \eqref{G3} and \eqref{G4} yield \eqref{cji}, \eqref{vy1}, \eqref{lipdeltaV} and \eqref{lipJ}.
We prove only \eqref{vy1}. For every $t\in [0,T]$ and $x\in {\rm D}_{F}$, by \eqref{G1} and \eqref{G2} we have
\begin{align*}
\norm{F_k(t,x)}_V&=\norm{G_{k}(t,x)+\zeta J_k(t,x)}_V=\norm{G_{k}(t,x)+\zeta [J_k(t,x)-x]+\zeta x}_V\\
&\leq\norm{G_k(t,x)}_V+\abs{\zeta} \frac{1}{k} \norm{G(t,x)}_V+\abs{\zeta}\norm{x}_V\\
&\leq 2\norm{G(t,x)}_V+\abs{\zeta}\norm{x}_V\\
&\leq3\left(\norm{F(t,x)}_V+\abs{\zeta}\norm{x}_V\right),
\end{align*}
where in the last inequality we used $k>\abs{\zeta}$ and \eqref{defnG}.
 Now we show that $F_k(t,\cdot)-\zeta  I$ is dissipative in ${V}$. Let $\alpha>0$, $t\in[0,T]$ and $x,y\in {V}$. By \eqref{defnG} and \eqref{eq_YO} we have
\begin{align*}
F_k(t,x)-\zeta x=G_k(t,x)+\zeta[J_k(t,x)-x]=\left(1+\frac{\zeta}{k}\right)G_{k}(t,x),
\end{align*}
hence by the dissipativity of $G_{k}(t,\cdot)$ we obtain
\begin{align*}
&\norm{x-y-\alpha[F_k(t,x)-\zeta  x-F_k(t,y)+\zeta  y]}_{V}\\
=&\norm{x-y-\alpha\left(1+\frac{\zeta}{k}\right)[G_{k}(t,x)- G_{k}(t,y)]}_{V}\geq \norm{x-y}_{V},
\end{align*}
where we used that $k>\abs{\zeta}$ implies $\displaystyle{\alpha\left(1+\frac{\zeta}{k}\right)>0}$. Hence $F_{k}(t,\cdot)-\zeta I$ is dissipative on ${V}$. 
Let $h>\abs{\zeta}$. For every $t\in[0,T]$ and $x,y\in{\rm D}_{F}$, by \eqref{defnG}, \eqref{G1} and \eqref{G3} we have 
\begin{align*}
\norm{J_k(t,x)-J_h(t,y)}_V&=\norm{x+\frac{1}{k} G_{k}(t,x)-y-\frac{1}{h}G_h(t,y)}_V\\
&\leq \norm{x-y}_V+\frac{1}{k}\norm{G(t,x)}_V+\frac{1}{h}\norm{G(t,y)}_V\\
&\leq\norm{x-y}_V+\left(\frac{1}{k}+\frac{1}{h}\right)\left(\norm{F(t,x)}_V+\norm{F(t,y)}_V+\abs{\zeta}\norm{x}_V+\abs{\zeta}\norm{y}_V\right),
\end{align*}
that yields \eqref{superyE}. Let $h>\abs{\zeta}$. For every $t\in[0,T]$ and $x,y\in{\rm D}_{F}$, by \eqref{defnG} we have
\begin{align*}
\ps{F_{k}(t,x)-F_h(t,y)}{x-y}_{V}=\ps{G_{k}(t,x)-G_h(t,y)}{x-y}_{V}+\zeta \ps{J_{k}(t,x)-J_h(t,y)}{x-y}_{V}.
\end{align*}
By \eqref{eq_YO} and the dissipativity of $G(t,\cdot)$ we have
\begin{align*}
\ps{G_{k}(t,x)-G_h(t,y)}{x-y}_{V}&=\ps{G_{k}(t,x)-G_h(t,y)}{J_k(t,x)-J_h(t,y)}_{V}\\
&\mathbin{ -}\ps{G_{k}(t,x)-G_h(t,y)}{\frac{1}{k}G_k(t,x)-\frac{1}{h}G_h(t,y)}_{V}\\
&\leq \left(\frac{1}{k}+\frac{1}{h}\right)(\norm{G(t,x)}_{V}+\norm{G(t,y)}_{V})^2,
\end{align*}
while by { \eqref{eq_YO} and \eqref{G2} we have
\[
\norm{J_{k}(t,x)-J_h(t,y)}_{V}\leq\norm{x-y}_{V}+\left(\frac{1}{k}+\frac{1}{h}\right)(\norm{G(t,x)}_{V}+\norm{G(t,y)}_{V}),
\]
so that}
\[
\zeta\ps{J_{k}(t,x)-J_h(t,y)}{x-y}_{V}\leq\abs{\zeta}\norm{x-y}_{V}^2+{ \abs{\zeta}\left(\frac{1}{k}+\frac{1}{h}\right)(\norm{G(t,x)}_{V}+\norm{G(t,y)}_{V})\norm{x-y}_{V}}.
\]
Hence
\begin{align*}
\ps{F_{k}(t,x)-F_h(t,y)}{x-y}_{V}\leq{}& \abs{\zeta} \norm{x-y}_{V}^2+\left(\frac{1}{k}+\frac{1}{h}\right)(\norm{G(t,x)}_{V}+\norm{G(t,y)}_{V})^2\\
&+{ \abs{\zeta}\left(\frac{1}{k}+\frac{1}{h}\right)(\norm{G(t,x)}_{V}+\norm{G(t,y)}_{V})\norm{x-y}_{V}},
\end{align*}
so by the Young inequality{ , $\norm{x-y}_V\leq\norm{x}_V+\norm{y}_V$} and \eqref{defnG} we obtain \eqref{supery}.
\end{proof}

\begin{remark}
Let 
\[
F(t,x)=B(t)x,\qquad t\in[0,T],\, x\in {\rm D}(B(t)),
\]
where $B(t):{\rm D}(B(t))\subseteq V\rightarrow V$ is a linear m-dissipative operator. Then by Remark \ref{DisEN} $(0,+\infty)\subseteq\rho(B)$ and \eqref{defyod} reads as
\begin{equation}\label{L-yosida}
F_k(t,x)=B_k(t)x=kB(t)R\left(k,B(t)\right)x,\qquad t\in[0,T],\, x\in V,\, k\in\N.
\end{equation}
\end{remark}

\subsection{Abstract evolution problems}\label{sect: absevolproblem}

In this section we summarize results on abstract non-autonomous problems.

Let $(V,\norm{\,\cdot\,}_V)$ be a real Banach space and let $T>0$.  Given $s\in [0,T)$, we consider the abstract evolution problem

\begin{align} \label{probastratto}
\begin{cases}
u'(t)=A(t)u(t)+F(t,u(t))\ \ t\in(s,T], \\
u(s)=x\in {\rm D}_F,
\end{cases}
\end{align}
where $\{A(t)\}_{t\in [0,T]}$ is a family of linear closed operators $A(t):{\rm D}(A(t))\subseteq V\rightarrow V$ and $F:[0,T]\times{\rm D}_F\subseteq [0,T]\times V\longrightarrow V$ is a continuous function. We recall that ${\rm D}(A(t))$ is a Banach space endowed with the graph norm of $A(t)$, for every $t\in [0,T]$.

We set 
\begin{align*}
C([a,b];{\rm D}(A(t)))&=\{g\in C([a,b];V)\ |\ t\longmapsto A(t)g(t)\in C([a,b];V)\},\\ 
C^k([a,b];{\rm D}(A(t)))&=\{g\in C^k([a,b];V)\ |\ t\longmapsto A(t)g(t)\in C^k([a,b];V)\},\ \ k\in\N.
\end{align*}

\begin{defn} $ $
\begin{enumerate}[{\rm (i)}]
\item A strict solution of \eqref{probastratto} is a function $u\in C^1([s,T];V)\cap C([s,T]; {\rm D}(A(t)))$ such that $u(s)=x\in D(A(s))\cap {\rm D}_F$, $u(t)\in {\rm D}(A(t))\cap {\rm D}_F$ and $u'(t)=A(t)u(t)+F(t,u(t))$ for all $t\in [s,T]$.
\item A classical solution of \eqref{probastratto} is a function $u\in C([s,T];V)\cap C^1((s,T];V)\cap C((s,T]; {\rm D}(A(t)))$ such that $u(s)=x\in V\cap {\rm D}_F$, $u(t)\in {\rm D}(A(t))\cap {\rm D}_F$ and $u'(t)=A(t)u(t)+F(t,u(t))$ for all $t\in (s,T)$.
\item A strong solution of \eqref{probastratto} is a function $u\in C([s,T],V)$ such that there exists $\{u_n\}_{n\in\N}\subseteq C^1([s,T];V)\cap C([s,T]; {\rm D}(A(t)))$, $u(t)\in{\rm D}_F$ for all $t\in [s,T]$ and
\begin{align*}
&u_n\longrightarrow u,\ \ \mbox{in}\  C([s,T],V),\\
&u_n(s)\longrightarrow x,\ \ \mbox{in}\ V,\\
&u'_n(\cdot)-A(\cdot)u_n(\cdot)\longrightarrow F(\cdot,u(\cdot)),\ \ \mbox{in}\  C([s,T],V).
\end{align*}
\end{enumerate}
\end{defn}

In the autonomous case, namely $A(t) \equiv A$ for all $t \in [s,T]$, semigroup theory is one of the most important tools in the literature for studying evolutionary PDEs such as \eqref{probastratto}. Instead, in the non-autonomous case, we can exploit the theory of evolution operators.

\begin{defn} 
$\{U(t,s)\}_{(s,t)\in\overline{\Delta}}\subseteq\mathscr L(V)$ is a (forward) evolution operator if
\begin{enumerate}[{\rm (i)}]
\item $U(t,t)=I$ for any $t\in [0,T]$,
\item $U(t,s)=U(t,r)U(r,s)$ for $s,r,t\in [0,T]$ with $s\leq r\leq t$.
\end{enumerate}
Moreover we say that
\begin{itemize}
\item $\{U(t,s)\}_{(s,t)\in\overline{\Delta}}$ is  strongly continuous if for every $x\in V$ the map
\begin{equation}
(s,t)\in\overline{\Delta}\longmapsto U(t,s)x\in V,
\end{equation}
is continuous;
\item $\{U(t,s)\}_{(s,t)\in\overline{\Delta}}$ is  uniformly continuous if for every $x\in V$ the map
\begin{equation}
(s,t)\in\overline{\Delta}\longmapsto U(t,s)\in\mathscr L(V),
\end{equation}
is continuous.
\end{itemize}
\end{defn}
From here on, whenever there is no risk of ambiguity, we shall simply denote $U(t,s)$ in place of $\{U(t,s)\}_{(s,t)\in\overline{\Delta}}$. Many authors gave suitable conditions on the operators $A(t)$ to prove the existence of an evolution operator $U(t,s)$ such that for suitable $x\in V$ (typically $x\in D(A(s))$ or $x\in \overline{D(A(s))}$ and $A(s)x\in \overline{D(A(s))}$) $u(t):=U(t,s) x$ is a classical or strict solution of 

\begin{align} \label{probastrattolin}
\begin{cases}
u'(t)=A(t)u(t)\ \ t\in (s,T], \\
u(s)=x\in V.
\end{cases}
\end{align} 

In such cases, we say that $\{U(t,s)\}_{(s,t)\in\overline{\Delta}}$ is associated to $\{A(t)\}_{t\in [0,T]}$. We refer to \cite{acq_1988, acq_ter_1987, tan_1979, yag_1976, sob_1961,paz_1983} for many results in this direction. However, the theory of evolution operators is not consistent and complete as the semigroup theory; for instance, there are no analogous versions of the Hille-Yosida theorem and there are no connections between the spectra of the operators $\{A(t)\}_{t\in [0,T]}$ and the asymptotic behavior of $\{U(t,s)\}_{(s,t)\in\overline{\Delta}}$.

\begin{defn}
A mild solution of \eqref{probastratto} is a function $u:[s,T]\longrightarrow V$ such that $u(t)\in{\rm D}_F$  for all $t\in [s,T]$, $F(\cdot,u(\cdot))\in L^1((s,T);V)$ and 
\begin{align}
u(t)=U(t,s)x+\int_s^tU(t,r)F(r,u(r))dr,\qquad t\in [s,T]
\end{align}
\end{defn}

In the next paragraphs, we list some results that will be useful later on.

\subsubsection{The continuous case} 
\begin{hypothesis}\label{cont_case}
 We assume that for every $t\in [0,T]$ the operator $A(t)$ is bounded and that the function
 \[
 t\in [0,T]\longmapsto A(t)\in\mathscr{L}(V)
 \]
 is continuous.
\end{hypothesis}
We have the following result.
\begin{thm}\label{cont_case_thm}
Assume that Hypothesis \ref{cont_case} holds. Then for all $s\in [0,T)$ and $x\in V$ the equation \eqref{probastrattolin} has a unique strict solution $u_{s,x}$. Setting $U(t,s)x:=u_{s,x}(t)$ for all $(s,t)\in\overline{\Delta}$ and $x\in V$, then $\{U(t,s)\}_{(s,t)\in\overline{\Delta}}$ is a uniformly continuous evolution operator on $V$ and for all $(s,t)\in\overline{\Delta}$ we have
\begin{align}
&\norm{U(t,s)}_{\mathscr L(V)}\leq  e^{\int_s^t\norm{A(r)}_{\mathscr L(V)}dr},\\
&\frac{\partial}{\partial t}U(t,s)=A(t)U(t,s),\\
&\frac{\partial}{\partial s}U(t,s)=-U(t,s)A(s).
\end{align}
\end{thm}
\begin{proof}
See \cite[Theorem 5.1, Theorem 5.2]{paz_1983}.
\end{proof}

\subsubsection{Abstract parabolic problem}
We state here the hypotheses of Acquistapace-Terreni, who constructed an evolution operator for \eqref{probastrattolin} for abstract parabolic equations in a very general setting. In such a case, each $A(t)$ is the generator of an analytic semigroup. We refer to \cite{acq_1988,acq_ter_1987,sch_2004}.

Throughout this subsection, when $V$ is a real Banach space we understand $\rho(A(t))$, $R(\lambda,A(t))$ and the sector $S_{\theta_0}$ to refer to the complexification of $V$ and of $A(t)$, and $\norm{\,\cdot\,}_{\mathscr{L}(V)}$ to denote the corresponding norm on the complexified space.

\begin{hypothesis}\label{ATipotesi}
Let $T>0$. For every $t\in [0,T]$, $A(t): {\rm D}(A(t))\subseteq V\longrightarrow V$ is a linear operator such that:
\begin{itemize}
\item{\rm\textbf{AT1}} there exist $M> 0$ and $\displaystyle{\theta_0\in \left(\frac{\pi}{2},\pi\right)}$ such that $$\rho(A(t))\supseteq S_{\theta_0}:=\left\{\lambda\in\mathbb{C}:\ \abs{{\rm arg}\,\lambda}\leq \theta_0\right\}\cup\{0\},\qquad t\in [0,T],$$ and $$\norm{R(\lambda, A(t))}_{\mathscr{L}(V)}\leq\frac{M}{1+\abs{\lambda}},\ \ \lambda\in S_{\theta_0},\ t\in [0,T];$$
\item{\rm\textbf{AT2}} there exist $L> 0$, $\mu,\nu\in(0,1]$ such that $\mu+\nu>1$ and
$$\abs{\lambda}^{\nu}\norm{A(t)R(\lambda, A(t))(A(t)^{-1}-A(s)^{-1})}_{\mathscr{L}(V)}\leq L\abs{t-s}^\mu,\ \ \lambda\in S_{\theta_0}\setminus\{0\},\ (s,t)\in\overline{\Delta}.$$
\end{itemize}
\end{hypothesis}

\begin{remark}
We remark that Hypotheis \ref{ATipotesi} can be rephrased without assuming $0\in\rho(A(t))$ for all $t\in[0,T]$; see e.g. \cite[pag. 5]{sch_2004}.
\end{remark}

\begin{thm}\label{thm23paolo} 
 Assume that $\{A(t)\}_{t\in [0,T]}$ satisfies Hypothesis \ref{ATipotesi}. Then there exists a unique (forward) evolution operator $\{U(t,s)\}_{(s,t)\in\overline{\Delta}}$ such that 
\begin{enumerate}[{\rm(i)}]
\item the mapping $(s,t)\longmapsto U(t,s)$ belongs to $B(\overline{\Delta};\mathscr{L}(V))\cap C(\Delta;\mathscr{L}(V))$ and
\begin{equation}\label{flow}
U(t,t)=I,\qquad U(t,r)U(r,s)=U(t,s),\qquad 0\leq s\leq r\leq t\leq T;
\end{equation}
\item $\displaystyle{\lim_{s\nearrow t}\norm{U(t,s) x-x}_V=0}$ if and only if $x\in\overline{{\rm D}(A(t))}$, and $\displaystyle{\lim_{t\searrow s}\norm{U(t,s) x-x}_V=0}$ if and only if $x\in\overline{{\rm D}(A(s))}$;
\item $U(t,s)$ maps $V$ into ${\rm D}(A(t))$ for every $(s,t)\in\Delta$ and there exists $C_1>0$ such that
\begin{equation}
\norm{A(t)U(t,s)}_{\mathscr{L}(V)}\leq \frac{C_1}{t-s},\qquad (s,t)\in\Delta;
\end{equation}
\item given $s\in [0,T)$, the mapping $t\longmapsto U(t,s)$ belongs to $C^1((s,T];\mathscr{L}(V))\cap C((s,T];\mathscr{L}(V,{\rm D}(A(t))))$ and 
\begin{equation}
\frac{\partial}{\partial t}U(t,s)=A(t)U(t,s);
\end{equation}
\item $U(t,s)$ maps ${\rm D}(A(s))$ in ${\rm D}(A(t))$ for all $(s,t)\in\Delta$ and there exists $C_2>0$ such that
\begin{equation}
\norm{A(t)U(t,s)A(s)^{-1}}_{\mathscr{L}(V)}\leq C_2;
\end{equation}
\item the mapping $(s,t)\longmapsto A(t)U(t,s)A(s)^{-1}$ belongs to $B(\overline{\Delta};\mathscr{L}(V))\cap C(\Delta;\mathscr{L}(V))$;
\item if $x\in\overline{{\rm D}(A(s))}$, then
\begin{equation}
V-\lim_{h\downarrow 0}h^{-1}[U(t,s+h)-U(t,s)]A(s)^{-1}x=-U(t,s)x,\qquad 0\leq s<t\leq T,
\end{equation}
and
\begin{equation}
V-\lim_{h\uparrow 0}h^{-1}[U(t,s+h)-U(t,s)]A(s+h)^{-1}x=-U(t,s)x,\qquad 0< s\leq t\leq T.
\end{equation}
\end{enumerate}
\end{thm}
\begin{proof}
See \cite[Theorem 2.3]{acq_1988}.
\end{proof}
Let $n\in\N$ and let $A_n(t)=nA(t)R(n,A(t))\in\mathscr L(V)$ be the Yosida approximation of $A(t)$ for all $t\in [0,T]$. We have the following result.

\begin{proposition}
Assume that $\{A(t)\}_{t\in [0,T]}$ satisfies Hypothesis \ref{ATipotesi}. Then the mapping $t\longmapsto A_n(t)$ is continuous and for all $n\in\N$ there exists an evolution operator $\{U_n(t,s)\}_{(s,t)\in\overline{\Delta}}$ associated to $\{A_n(t)\}_{t\in [0,T]}$ such that
\begin{align}
&U_n(t,s)\longrightarrow U(t,s),\ \ \mbox{in}\ \mathscr L(V)\ \mbox{as}\ n\to+\infty,\qquad (s,t)\in\Delta,\\
&A_n(t)U_n(t,s)\longrightarrow\frac{\partial}{\partial t} U(t,s),\ \ \mbox{in}\ \mathscr L(V)\ \mbox{as}\ n\to+\infty,\qquad (s,t)\in\Delta.
\end{align}
\end{proposition}
\begin{proof}
See \cite[Theorem 2.3]{acq_1988} and \cite[Proposition 4.5]{acq_ter_1987}.
\end{proof}

\section{Well-posedness of the deterministic problem \eqref{NAPP}}\label{sect: WD}
Let $T>0$ and let $(H,\norm{\,\cdot\,}_H,\ps{\cdot}{\cdot}_H)$ be a separable Hilbert space. 
The following assumptions define the coefficients of the abstract PDE that we will study in this section.

\begin{hypothesis}\label{EU2} $ $
\begin{enumerate}[{\rm(i)}]
\item Let $(E,\norm{\,\cdot\,}_E)$ be a separable Banach space densely and continuously embedded in $H$.
\item\label{EU2.2} Let $F:[0,T]\times {\rm D}_F\subseteq [0,T]\times H\rightarrow H$ be a function such that $E\subseteq {\rm D}_F$. For every $t\in[0,T]$ we denote by $F_E(t,\cdot)$ the part of $F(t,\cdot)$ in $E$, namely
\[
{\rm D}_{F_E}:=\left\{x\in {  E}:\ F(t,x)\in E\ \text{ for all }t\in[0,T]\right\},\qquad F_E(t,x):=F(t,x),\ \ x\in {\rm D}_{F_E}.
\]
We assume that ${\rm D}_{F_E}$ is dense in $E$ and that
\begin{enumerate}[{\rm(a)}]
    \item for every $R>0$ there exists $M_R>0$ such that
    \[
    \sup_{t\in[0,T]}\ \sup_{\norm{x}_E\leq R}\norm{F(t,x)}_H\leq M_R,
    \]
    namely $F$ maps bounded sets of $[0,T]\times E$ into bounded sets of $H$;
    \item for every $x\in {\rm D}_{F_E}$ the function $t\in[0,T]\rightarrow F_E(t,x)\in E$ is continuous;
    \item there exists $\zeta\in\R$ such that $F(t,\cdot)-\zeta I$ is $m$-dissipative in $H$ and $F_{E}(t,\cdot)-\zeta I$ is $m$-dissipative in $E$ for all $t\in[0,T]$.
\end{enumerate}

\item\label{EU2.3} Let $\{A(t)\}_{t\in [0,T]}$ be a family of closed linear operators $A(t):{\rm D}(A(t))\subseteq H\rightarrow H$ such that $A(t)$ is $m$-dissipative in $H$ and $A_E(t)$ is $m$-dissipative in $E$ for all $t\in [0,T]$.

\item\label{EU2.5} There exists a strongly continuous evolution operator $\{U(t,s)\}_{(s,t)\in\overline{\Delta}}$ on $H$ such that:
\begin{enumerate}[{\rm(a)}]
    \item $U(\cdot,\cdot)_{|E}\in C(\Delta; \mathscr{L}(E))$,
    \item  the function $t\in[0,T]\rightarrow A_n(t):=nA_E(t)R(n,A_E(t))\in\mathscr{L}(E)$ is continuous for all $n\in\N$,
     \item for all $(s,t)\in\Delta$ 
\begin{align}\label{yosidaE}
\lim_{n\rightarrow+\infty}\norm{U_n(t,s)-U(t,s)}_{\mathscr{L}(E)}=0
\end{align}
where, for all $n\in\N$, $\{U_n(t,s)\}_{(s,t)\in\overline{\Delta}}$ is the evolution operator associated to $\{A_n(t)\}_{t\in[0,T]}$ in the sense of Theorem \ref{cont_case_thm}.
\end{enumerate} 
\item For every continuous function $f:[0,T]\rightarrow E$ the function $t\in[0,T]\rightarrow F_E(t,f(t))\in E$ is measurable and
\[
\int^t_s\norm{F_E(r,f(r))}_Eds<+\infty,\qquad (s,t)\in\Delta.
\]
\end{enumerate}
\end{hypothesis}

\begin{remark}
In the sequel we often omit the subindex $E$ in $A_E(t)$ and $U(t,s)_{|E}$, writing instead $A(t)$ and $U(t,s)$.
\end{remark}

Let $s\in [0,T)$ and let $f:[s,T]{\longrightarrow} E$ be a continuous function. We will study the following PDE
\begin{align}\label{eq}
\begin{cases}
\dfrac{dy}{dt}(t)=A(t)y(t)+F(t,y(t)+f(t)),\ \  t\in(s,T],\\
y(s)=x\in E.
\end{cases}
\end{align}

\begin{remark}
We note that Hypotheses \ref{EU2}(\ref{EU2.2}-c) and \ref{EU2}\eqref{EU2.3} are equivalent to the following:
\begin{enumerate}[\rm(a)]
\item there exists $\alpha\in\R$ such that $A(t)-\alpha I$ is { $m$-}dissipative in $H$ and $A_E(t)-\alpha I$ is { $m$-}dissipative in $E$ for all $t\in[0,T]$;
\item  there exists $\eta\in\R$ such that $F(t,\cdot)-\eta I$ is $m$-dissipative in $H$ and $F_{E}(t,\cdot)-\eta I$ is $m$-dissipative in $E$ for all $t\in[0,T]$.

\end{enumerate}
In this case, in \eqref{eq} we replace $A$ and $F$ with $\widetilde{A}=A-\alpha I$ and $\widetilde{F}{ (t,z)=F(t,z)+\alpha(z-f(t))}$, respectively. In this way $\widetilde{A}$ and $\widetilde{F}$ verify Hypothesis \ref{EU2} with $\zeta=\alpha+\eta$.
\end{remark}

In this section we are interested in existence and uniqueness of a mild solution to \eqref{eq}.
\begin{defn}
For every $T>0$, $s\in [0,T)$ and $x\in E$ we call mild solution of \eqref{eq} every { function $y_{s,x}\in C([s,T];H)\cap C_b((s,T];E)$} such that, for every $t\in [s,T]$, we have
\begin{align*}
y_{s,x}(t)=U(t,s)x+\int_s^tU(t,r)F(r,y_{s,x}(r)+f(r))dr.
\end{align*}
\end{defn}
To be more precise, we will prove that for every $T>0$, $s\in [0,T)$ and $x\in E$, \eqref{eq} has a unique mild solution $y_{s,x}\in C([s,T],H)\cap C_{ b}((s,T],E)$ under one of the following two assumptions.

\begin{hypothesis}\label{Hyp1}
Let Hypothesis \ref{EU2} holds true and let $E$ be reflexive. We assume that
\begin{enumerate}[{\rm(i)}]
\item there exist $\theta\in[0,1)$ and $C_\theta>0$ such that $U(t,s)\in\mathscr{L}(H;E)$ for every $(s,t)\in\Delta$ and
\begin{align}\label{smoothing}
\norm{U(t,s)}_{\mathscr{L}(H;E)}\leq \frac{C_\theta}{(t-s)^{\theta}},\qquad (s,t)\in\Delta;
\end{align}
\item for every $R>0$ there exists $L_R>0$ such that 
\begin{equation}\label{FdualEH}
\abs{\ps{F(t,x)-F(t,y)}{h}_{H}}\leq L_R\norm{x-y}_H\norm{h}_E,\qquad t\in[0,T],\ h\in E,
\end{equation}
for every $x,y\in E$ with $\norm{x}_E\leq R$ and $\norm{y}_E\leq R$.
\end{enumerate}
\end{hypothesis}

\begin{hypothesis}\label{Hyp2}
Let Hypothesis \ref{EU2} holds true. We assume that $F(t,\cdot) $ maps $E$ into itself for all $t\in[0,T]$, 
\begin{equation}\label{F0limitata}
\sup_{t\in[0,T]}\norm{F(t,0)}_E<+\infty,
\end{equation}
and $F(t,\cdot)$ is locally Lipschitz in $E$ uniformly in $t\in[0,T]$, namely for every bounded subset $B$ of $E$ there exists $L_B>0$ such that 
\begin{equation*}
\norm{F(t,z_1)-F(t,z_2)}_E\leq L_B\norm{z_1-z_2}_E,\ \ \forall t\in[0,T],\ \forall z_1,z_2\in B.
\end{equation*}
\end{hypothesis}

\subsection{Approximating problems}
In this subsection we approximate the PDE \eqref{eq} via the Yosida approximations of $\{A(t)\}_{t\in[0,T]}$ and $F$. We still denote by $A$ and $F$ the part of $A$ and $F$ in $E$, respectively. Let $\{F_k(t,\cdot)\}_{k> \zeta}$ and $\{A_n(t)\}_{n\in\N}$ be the Yosida approximations of $F(t,\cdot)$ and $A(t)$, respectively. By Hypotheses \ref{EU2}, the assumptions of Proposition \ref{Lemma_YO} are verified with both $V=H$ and $V=E$. Before studying the approximate problem, we need a technical lemma on the evolution operator $\{U_n(t,s)\}_{(s,t)\in\overline{\Delta}}$ associated with $\{A_n(t)\}_{t\in[0,T]}$.

\begin{lemma}
Assume that Hypothesis \ref{EU2} holds. Then
\begin{align}
\max_{(s,t)\in\overline{\Delta}}\{\norm{U_n(t,s)}_{\mathscr{L}(H)},\norm{U_n(t,s)}_{\mathscr{L}(E)}\}&\leq 1,\qquad n\in\N,\label{contrazione-n}\\
\max_{(s,t)\in\overline{\Delta}}\{\norm{U(t,s)}_{\mathscr{L}(H)},\norm{U(t,s)}_{\mathscr{L}(E)}\}&\leq 1\label{contrazione}.
\end{align}
\end{lemma}

\begin{proof}
    We prove the statement with respect to the norm of $E$. Let $x\in E$ and $n\in\N$, and let $s\in[0,T]$ be fixed. By Theorem \ref{cont_case_thm} the function $t\in[s,T]\longrightarrow U_n(t,s)x\in E$ is differentiable and $\frac{d}{dt}U_n(t,s)x=A_n(t)U_n(t,s)x$. Hence by \eqref{Ldiff-} and Hypothesis \ref{EU2}\eqref{EU2.3}, for every $t\in(s,T]$ there exists $z^*\in\partial\norm{U_n(t,s)x}_E$ such that
   \begin{align*}
        \frac{d^-\norm{U_n(t,s)x}_E}{dt}\leq \dscal{A_n(t)U_n(t,s)x}{z^*}{E}\leq 0,
    \end{align*}
    where we used that $A_n(t)$ is dissipative in $E$. Since $U_n(s,s)x=x$, by \eqref{varofcost} with $b=0$ and $g\equiv 0$ we obtain
    \begin{align*}
 \norm{U_n(t,s)x}_E\leq \norm{x}_E,\qquad (s,t)\in\overline{\Delta}.
    \end{align*}
    By the arbitrariness of $x\in E$ we obtain \eqref{contrazione-n}. Finally, for $(s,t)\in\Delta$ the estimate \eqref{contrazione} follows by \eqref{yosidaE}, while for $t=s$ it is immediate since $U(s,s)=I$ by \eqref{flow}. The statement with respect to the norm of $H$ follows by similar calculations, replacing the duality product of $E$ with the inner product of $H${  , for $x\in E$, and by the density of $E$ in $H$}.
\end{proof}

In all this section we fix $T>0$, $s\in [0,T)$ and $f:[s,T]\rightarrow E$ continuous. For all $n\in\N$ and $k>\abs{\zeta}$ we introduce the approximate problem
\begin{align}\label{eqdt}
\begin{cases}
\dfrac{dy^{k,n}}{dt}(t)=A_{n}(t) y^{k,n}(t)+F_{k}(t,y^{k,n}(t)+f(t)), & t\in (s,T],\\
y^{k,n}(s)=x\in E.
\end{cases}
\end{align}

\begin{proposition}\label{soldt}
Assume that Hypothesis \ref{EU2} holds true. For every $n\in\N$, $k>\abs{\zeta}$ and $x\in E$ the problem \eqref{eqdt} has a unique strict solution $y_{s,x}^{k,n}$ in  $C([s,T],E)$. Moreover, for every $n\in\N$, $h,k>\abs{\zeta}$, $x,z\in E$ and $t\in [s,T]$, there exists $C=C(x,\zeta,s,T)>0$ such that
\begin{align}
&\norm{y_{s,x}^{k,n}(t)}_{H}\leq e^{\zeta (t-s)}\norm{x}_{H}+3\int_s^te^{\zeta (t-r)}\left(\norm{F(r,f(r))}_{H}+\abs{\zeta} \norm{f(r)}_{H}\right)dr,\label{stidXth}\\
&\norm{y_{s,x}^{k,n}(t)}_{E}\leq e^{\zeta (t-s)}\norm{x}_{E}+3\int_s^te^{\zeta (t-r)}\left(\norm{F(r,f(r))}_{E}+\abs{\zeta} \norm{f(r)}_{E}\right)dr,\label{stidEth}	\\
&\norm{y_{s,x}^{k,n}(t)-y_{s,z}^{k,n}(t)}_{H}\leq  e^{\zeta (t-s)}\norm{x-z}_{H},\label{lipdXth}\\
&\norm{y_{s,x}^{k,n}(t)-y_{s,z}^{k,n}(t)}_{E}\leq  e^{\zeta (t-s)}\norm{x-z}_{E}\label{lipdEth},\\
&\norm{y_{s,x}^{k,n}(t)-y_{s,x}^{h,n}(t)}_{H}^2\leq C\left(\frac{1}{k}+\frac{1}{h}\right),\label{lipdelta}
\end{align}
where $\zeta$ is the constant defined in Hypotheses \ref{EU2}(\ref{EU2.2}-c).
\end{proposition}
\begin{proof}
Fix  $n\in\N$, $k>\abs{\zeta}$ and $x\in E$. We consider the operator $V:C([s,T],E){\longrightarrow} C([s,T],E)$ defined by
\[
[V(y)](t):= U_{n}(t,s) x+\int^t_s U_{n}(t,r) F_{k}(r,y(r)+f(r))dr,\quad y\in C([s,T],E),\; t\in [s,T].
\]
By Hypotheses \ref{EU2}\eqref{EU2.5} and Theorem \ref{cont_case_thm} we have $V\left(C([s,T],E)\right)\subseteq C([s,T],E)$, and, by \eqref{lipdeltaV} and \eqref{contrazione-n} for every $y,z\in C([s,T],E)$ 
\[
\norm{V(y)-V(z)}_{C([s,T],E)}\leq 3k\left(T-s\right)\norm{y-z}_{C([s,T],E)},
\]
Hence for $s<T_0<T$ small enough, by the contraction mapping theorem the problem  \eqref{eqdt} has a unique mild solution $y_{s,x,T_0}^{k,n}$ belonging to $C([s,T_0],E)$. Reiterating this argument, we get that \eqref{eqdt} has a unique mild solution in every $C([T_0,2T_0-s];E),\ldots, C([h(T_0-s)-s,h(T_0-s)+T_0];E)$ and $ C([h(T_0-s)+T_0,T];E)$ where $h$ is the integer part of $(T-T_0)(T_0-s)^{-1}$. So, attaching the solutions on such intervals, we obtain the existence of a mild solution $y_{s,x}^{k,n}$ to \eqref{eqdt} defined on $[s,T]$ and belonging to $C([s,T];E)$. Since $F_{k}(r,\cdot)$ is Lipschitz continuous for every $r\in [s,t]$,  then it is standard to prove that $y_{s,x}^{k,n}$ is unique in $C([s,T];E)$. 

Now we prove that $y_{s,x}^{k,n}$ is also a strict solution to \eqref{eqdt}. Since $f:[s,T]\rightarrow E$ is continuous, by Lemma \ref{Lemma_YO} and Hypotheses \ref{EU2}(\ref{EU2.2}-b) the map $r\in [s,T]\longrightarrow F_k(r,y_{s,x}^{k,n}(r)+f(r))\in E$ is continuous. Hence, taking into account Hypotheses \ref{EU2}(\ref{EU2.5}-b), by Theorem \ref{cont_case_thm} and the fundamental theorem of calculus we have $y_{s,x}^{k,n}\in C^1([s,T];E)$ and
\[
\frac{dy_{s,x}^{k,n}}{dt}(t)=A_{n}(t) y_{s,x}^{k,n}(t)+F_{k}(t,y_{s,x}^{k,n}(t)+f(t)).
\]

We prove \eqref{stidEth}. If $t\in[s,T]$, then by \eqref{Ldiff-}, Proposition \ref{Lemma_YO} and Hypotheses \ref{EU2} there exists $y^*\in\partial \norm{y_{s,x}^{k,n}(t)}_E$ such that
\begin{align*}
\dfrac{d^-\norm{y_{s,x}^{k,n}(t)}_E}{dt}&\leq\dscal{A_{n}(t) y_{s,x}^{k,n}(t)}{y^*}{E}+\dscal{F_{k}(t,y_{s,x}^{k,n}(t)+f(t))}{y^*}{E}\\
&=\dscal{A_{n}(t) y_{s,x}^{k,n}(t)}{y^*}{E}+\dscal{F_{k}(t,y_{s,x}^{k,n}(t)+f(t))-F_{k}(t,f(t))}{y^*}{E}\\&+\dscal{F_{k}(t,f(t))}{y^*}{E}\leq \zeta\norm{y_{s,x}^{k,n}(t)}_E+\norm{F_{k}(t,f(t))}_E.
\end{align*}
 By \eqref{varofcost} we get
\[
\norm{y_{s,x}^{k,n}(t)}_E\leq e^{\zeta (t-s)}\norm{x}_E+\int_s^te^{\zeta (t-r)}\norm{F_k(r,f(r))}_E dr.
\] 
Hence \eqref{stidEth} follows by \eqref{vy1}. 

We prove now \eqref{lipdEth}. Let $t\in[s,T]$. Still by \eqref{Ldiff-}, Proposition \ref{Lemma_YO} and Hypotheses \ref{EU2} there exists $y^*\in\partial \norm{y_{s,x}^{k,n}(t)-y_{s,z}^{k,n}(t)}_E$ such that
\begin{align*}
\dfrac{d^-\norm{y_{s,x}^{k,n}(t)-y_{s,z}^{k,n}(t)}_E}{dt}&\leq\dscal{A_{n}(t)[y_{s,x}^{k,n}(t)-y_{s,z}^{k,n}(t)]}{y^*}{E}\\
&+\dscal{F_{k}(t,y_{s,x}^{k,n}(t)+f(t))-F_{k}(t,y_{s,z}^{k,n}(t)+f(t))}{y^*}{E}\\
&\leq \zeta\norm{y_{s,x}^{k,n}(t)-y_{s,z}^{k,n}(t)}_E.
\end{align*}
By \eqref{varofcost}, \eqref{lipdEth} follows. 

Since $A_n(t) $ and $F_k(t,\cdot)-\zeta I$ are even dissipative in $H$, the same procedures performed to prove \eqref{stidEth} and \eqref{lipdEth} can be exploited to prove \eqref{stidXth} and \eqref{lipdXth}, respectively, replacing the duality product of $E$ with the inner product of $H$.

Lastly, we prove \eqref{lipdelta}. Let $t\in[s,T]$ and $h>\abs{\zeta}$. By \eqref{Ldiff-}, Proposition \ref{Lemma_YO} and Hypothesis \ref{EU2} there exists $C=C(x,\zeta,s,T)>0$ such that
\begin{align*}
\dfrac{d^-}{dt}\norm{y_{s,x}^{k,n}(t)-y_{s,x}^{h,n}(t)}_{H}^2&\leq 2\ps{A_{n}(t) [y_{s,x}^{k,n}(t)-y_{s,x}^{h,n}(t)]}{y_{s,x}^{k,n}(t)-y_{s,x}^{h,n}(t)}_{H}\\
&+2\ps{F_{k}(t,y_{s,x}^{k,n}(t)+f(t))-F_{h}(t,y_{s,x}^{h,n}(t)+f(t))}{y_{s,x}^{k,n}(t)-y_{s,x}^{h,n}(t)}_{H}\\
&\leq 2\abs{\zeta}\norm{y_{s,x}^{k,n}(t)-y_{s,x}^{h,n}(t)}_{H}^2\\
&+\left(\frac{1}{k}+\frac{1}{h}\right)C\left(\norm{F(t,y_{s,x}^{k,n}(t)+f(t))}_{H}^2+\norm{F(t,y_{s,x}^{h,n}(t)+f(t))}_{H}^2\right)\\
&+\left(\frac{1}{k}+\frac{1}{h}\right)C\left(\norm{y_{s,x}^{k,n}(t)+f(t)}_{H}^2+\norm{y_{s,x}^{h,n}(t)+f(t)}_{H}^2\right),
\end{align*}
where we used \eqref{supery} with $V=H$ and the dissipativity of $A_n(t)$ in $H$. By \eqref{varofcost} we obtain
{\small
\begin{align*}
\norm{y_{s,x}^{k,n}(t)-y_{s,x}^{h,n}(t)}_{H}^2&\leq \left(\frac{1}{k}+\frac{1}{h}\right)C\int^t_se^{2\abs{\zeta}(t-r)}\left(\norm{F(r,y_{s,x}^{k,n}(r)+f(r))}_{H}^2+\norm{F(r,y_{s,x}^{h,n}(r)+f(r))}_{H}^2\right)dr\\
&+\left(\frac{1}{k}+\frac{1}{h}\right)C\int^t_se^{2\abs{\zeta}(t-r)}\left(\norm{y_{s,x}^{k,n}(r)+f(r)}_{H}^2+\norm{y_{s,x}^{h,n}(r)+f(r)}_{H}^2\right)dr.
\end{align*}
}
By \eqref{stidEth} there exists $R>0$, independent of $k$, $h$ and $n$, such that $\norm{y_{s,x}^{k,n}(r)+f(r)}_E\leq R$ for every $r\in[s,T]$, hence by Hypothesis \ref{EU2}(\ref{EU2.2}-a) we have $\norm{F(r,y_{s,x}^{k,n}(r)+f(r))}_H\leq M_R$ for every $r\in[s,T]$, and the same bounds hold with $h$ in place of $k$. Since $E$ is continuously embedded in $H$, \eqref{lipdelta} follows.
\end{proof}

For every $k>\abs{\zeta}$, we study the mild solution of the following PDE.
\begin{align}\label{eqd}
\begin{cases}
\dfrac{dy_{s,x}^k }{dt}(t)=A(t)y_{s,x}^k (t)+F_{k}(t,y_{s,x}^k +f(t)),\ \  t\in(s,T],\\
y_{s,x}^k (s)=x\in E.
\end{cases}
\end{align}

\begin{proposition}\label{sold}
Assume that Hypothesis \ref{EU2} holds true. For every $k>\abs{\zeta}$ and $x\in E$ equation \eqref{eqd} has a unique 
 mild solution $y_{s,x}^k$ in $C([s,T],H)\cap C_b((s,T],E)$ such that for every $t\in [s,T]$
\begin{equation}\label{convdt}
\lim_{n{\to}+\infty}\norm{y_{s,x}^{k,n}(t)-y_{s,x}^k(t)}_{E}=0,
\end{equation}
where, for every $n\in\N$, $y_{s,x}^{k,n}\in C([s,T],E)$ is the unique strict solution  to \eqref{eqdt} given by Proposition \eqref{soldt}. In addition, for every $k,h>\abs{\zeta}$, $x,z\in E$ and $t\in [s,T]$  we have
\begin{align}
&\norm{y_{s,x}^k(t) }_{H}\leq e^{\zeta (t-s)}\norm{x}_{H}+3\int_s^te^{\zeta (t-r)}\left(\norm{F(r,f(r))}_{H}+\abs{\zeta}\norm{f(r)}_{H}\right)dr,\label{stidX}\\
&\norm{y_{s,x}^k(t) }_{E}\leq e^{\zeta (t-s)}\norm{x}_{E}+3\int_s^te^{\zeta (t-r)}\left(\norm{F(r,f(r))}_{E}+\abs{\zeta} \norm{f(r)}_{E}\right)dr,\label{stidE}	\\
&\norm{y_{s,x}^k (t)-y_{s,z}^k (t)}_{H}\leq  e^{\zeta (t-s)}\norm{x-z}_{H},\label{lipdX}\\
&\norm{y_{s,x}^k (t)-y_{s,z}^k (t)}_{E}\leq  e^{\zeta (t-s)}\norm{x-z}_{E}\label{lipdE},\\
&\norm{y_{s,x}^k (t)-y_{s,x}^{h}(t)}_{H}^2\leq C\left(\frac{1}{k}+\frac{1}{h}\right),\label{stimaconv}
\end{align}
where $\zeta$ is the constant given by Hypotheses \ref{EU2}(\ref{EU2.2}-c) and $C=C(x,\zeta,s,T)>0$ is the constant given by \eqref{lipdelta}.
\end{proposition}
\begin{proof}
We fix $k>\abs{\zeta}$ and $x\in E$.
Since $F_k$ is Lipschitz continuous and considering \eqref{contrazione}, the existence and uniqueness of a mild solution to \eqref{eqd} can be proven by a fixed point procedure as in the proof of Proposition \ref{soldt}. So by Hypotheses \ref{EU2}\eqref{EU2.5} equation \eqref{eqd} has a unique mild solution $y_{s,x}^k $ belonging to $C([s,T];H)\cap C((s,T];E)$.

Let us prove \eqref{convdt}.  For every $n\in\N$ and  $t\in [s,T]$ we have

\begin{align*}
y_{s,x}^{k,n}(t)-y^{k}_{s,x}(t)&= \left[ U_n(t,s)-U(t,s)\right]x\phantom{aaaaaaaaaaa}\\
&+\int^t_sU_n(t,r)F_k\left(r,y_{s,x}^{k,n}(r)+f(r)\right)dr\\
&-\int^t_s U(t,r)F_k(r,y_{s,x}^k (r)+f(r))dr\\
&=\left[U_n(t,s)-U(t,s)\right]x\\
&+\int^t_s\left[U_n(t,r)-U(t,r)\right]F_k(r,y_{s,x}^{k,n}(r)+f(r))dr\\
&+\int^t_sU(t,r)\left[F_k(r,y_{s,x}^{k,n}(r)+f(r))-F_k(r,y_{s,x}^{k}(r)+f(r))\right]dr,
\end{align*}

hence by \eqref{lipdeltaV} (with $V=E$) we have
\begin{align*}
\norm{y_{s,x}^{k,n}(t)-y_{s,x}^{k}(t)}_E & \leq  \norm{U_n(t,s)-U(t,s)}_{\mathscr{L}(E)}\norm{x}_E\\
&+\int^t_s  \norm{U_n (t,r)-U(t,r)}_{\mathscr{L}(E)}\norm{F_k(r,y_{s,x}^{k,n}(r)+f(r))}_E dr\\
&+3k\int^t_s\norm{U(t,r)}_{\mathscr{L}(E)}\norm{y_{s,x}^{k,n}(r)-y_{s,x}^{k}(r))}_Edr.
\end{align*}
Noting that \eqref{stidEth} is independent of $n\in\N$, by \eqref{vy1} (with $V=E$), Hypotheses \ref{EU2}(\ref{EU2.2}-a) and \eqref{contrazione} there exists a constant $M_k=M_k(x,\zeta,s,T)$ such that
\begin{align}
\norm{y_{s,x}^{k,n}(t)-y_{s,x}^{k}(t)}_E  \leq M_k\Bigg(&\norm{U_n(t,s)-U(t,s)}_{\mathscr{L}(E)}\notag\\
&+\int^t_s  \norm{U_n (t,r)-U(t,r)}_{\mathscr{L}(E)}+\norm{y_{s,x}^{k,n}(r)-y_{s,x}^{k}(r))}_Edr\Bigg).\label{perconvergenza}
\end{align}
by the Gr\"onwall inequality there exists a constant $M_k'=M_k'(x,\zeta,s,T)$ such that
\[
\norm{y_{s,x}^{k,n}(t)-y_{s,x}^{k}(t)}_E\leq M_k'\Bigg(\norm{U_n(t,s)-U(t,s)}_{\mathscr{L}(E)}+\int^t_s  \norm{U_n (t,r)-U(t,r)}_{\mathscr{L}(E)}dr\Bigg),
\]
so by Hypotheses \ref{EU2}(\ref{EU2.5}-c), \eqref{contrazione-n}, \eqref{contrazione} and the dominated convergence theorem we obtain \eqref{convdt}.

Finally letting $n\longrightarrow+\infty$ in \eqref{stidXth}, \eqref{stidEth}, \eqref{lipdXth}, \eqref{lipdEth} and \eqref{lipdelta} we obtain \eqref{stidX}, \eqref{stidE}, \eqref{lipdX}, \eqref{lipdE} and \eqref{stimaconv}, respectively.\end{proof}

\subsection{Well-posedness of \eqref{eq}}

Now we can prove that \eqref{eq} has unique mild solution. In all this section we fix $T>0$, $s\in [0,T)$ and $f:[s,T]\rightarrow E$ continuous.

\begin{thm}\label{proyos}
Assume that either Hypothesis \ref{Hyp1} or Hypothesis \ref{Hyp2} holds true. Then for every $x\in E$ equation \eqref{eq} has a unique mild solution $y_{s,x}$ in $C([s,T],H)\cap C_b((s,T],E)$ such that
\begin{equation}\label{convdX}
\lim_{k\rightarrow+\infty}\norm{y_{s,x}^k-y_{s,x}}_{C([s,T],H)}=0,\quad \forall\; T>0,\;
\end{equation}
where, for every $k>\abs{\zeta}$, $y_{s,x}^{k}$ is the unique mild solution  to \eqref{eqd} given by Proposition \eqref{sold}. Moreover for every $x,z\in E$ and $(s,t)\in\Delta$ we have
\begin{align}
&\norm{y_{s,x}(t)}_{H}\leq e^{\zeta (t-s)}\norm{x}_{H}+3\int_s^te^{\zeta (t-r)}\left[\norm{F(r,f(r))}_{H}+\abs{\zeta} \norm{f(r)}_{H}\right]dr,\label{stiX}\\
&\norm{y_{s,x}(t)}_E\leq e^{\zeta (t-s)}\norm{x}_E+3\int_s^te^{\zeta (t-r)}\left[\norm{F(r,f(r))}_E+\abs{\zeta}\norm{f(r)}_E\right]dr,\label{stiE}	\\
&\norm{y_{s,x}(t)-y_{s,z}(t)}_{H}\leq  e^{\zeta (t-s)}\norm{x-z}_{H}\label{lipX}\\
&\norm{y_{s,x}(t)-y_{s,z}(t)}_E\leq  e^{\zeta (t-s)}\norm{x-z}_E\label{lipE},
\end{align}
where $\zeta$ is the constant defined in Hypotheses \ref{EU2}(\ref{EU2.2}-c). If in addition  Hypothesis \ref{Hyp2} holds true then for every $x\in E$
\begin{equation}\label{convE}
\lim_{k\rightarrow+\infty}\sup_{t\in [s+\varepsilon,T]}\norm{y_{s,x}^{k}(t)-y_{s,x}(t)}_{E}=0,\ \ \forall\varepsilon>0.
\end{equation} 
\end{thm}
\begin{proof}
Fix $x\in E$. By \eqref{stimaconv}, there exists $y_{s,x}\in C([s,T],H)$ such that \eqref{convdX} holds true. Moreover letting $k\rightarrow +\infty$ in \eqref{stidX} and \eqref{lipdX}, we obtain \eqref{stiX} and \eqref{lipX}. Now we prove that the limit $y_{s,x}$ is the unique mild solution to \eqref{eq}. \\

\noindent \textbf{Existence}
Here we prove that $y_{s,x}$ is a mild solution to \eqref{eq} but to do so we must distinguish the following two cases.

\begin{itemize}

\item\textbf{Case 1: Hypothesis \ref{Hyp1} holds true.}

Fix $t\in (s,T]$. We first prove that $y_{s,x}(t)\in E$. By \eqref{stidE} and the continuity of $f$ there exists $R>0$, independent of $k$, such that
\begin{equation}\label{boundR}
\norm{y_{s,x}^k(r)+f(r)}_E\leq R,\qquad r\in[s,T],\ k>\abs{\zeta}.
\end{equation}
Since $E$ is reflexive, by \eqref{boundR} there exist a subsequence $\{y_{s,x}^{k_j}(t)\}_{j\in\N}$ and $z\in E$ such that $y_{s,x}^{k_j}(t)\rightharpoonup z$ weakly in $E$ as $j\rightarrow+\infty$. Since $E$ is continuously embedded in $H$, for every $h\in H$ the functional $w\in E\longmapsto \ps{w}{h}_H\in\R$ belongs to $E^*$, so that $y_{s,x}^{k_j}(t)\rightharpoonup z$ weakly in $H$ as well. On the other hand $y_{s,x}^{k}(t)\rightarrow y_{s,x}(t)$ in $H$ by \eqref{convdX}, whence $z=y_{s,x}(t)$. Therefore $y_{s,x}(t)\in E$ and, by the weak lower semicontinuity of the norm of $E$ and \eqref{stidE}, the estimate \eqref{stiE} holds and
\begin{equation}\label{boundRlim}
\norm{y_{s,x}(t)+f(t)}_E\leq R,\qquad t\in (s,T].
\end{equation}

\noindent We now prove that
\begin{equation}\label{convFk}
\lim_{k\rightarrow +\infty}\abs{\ps{F_k(t,y_{s,x}^{k}(t)+f(t))-F(t,y_{s,x}(t)+f(t))}{h}_H}=0,\qquad h\in E.
\end{equation}
By \eqref{cji} (with $V=H$) we have
\begin{align*}
\norm{J_{k}(t,y_{s,x}^k(t)+f(t))-y_{ s,x}(t)-f(t)}_{H}&\leq \frac{1}{k}\left[\norm{F(t,y_{s,x}^k(t)+f(t))}_{H}+\abs{\zeta}\norm{y_{s,x}^k(t)+f(t)}_{H}\right]\\
&+\norm{y_{s,x}^k(t)-y_{ s,x}(t)}_{H}.
\end{align*}
Hence by \eqref{boundR}, Hypothesis \ref{EU2}(\ref{EU2.2}-a) and \eqref{convdX} we obtain
\begin{align}\label{convJ}
\lim_{k{\rightarrow} +\infty}\norm{J_{k}(t,y_{s,x}^k(t)+f(t))-y_{ s,x}(t)-f(t)}_{H}=0.
\end{align}
Let $x_0\in {\rm D}_{F_E}$. We remark that ${\rm D}_{F_E}$ is non-empty since ${\rm D}_{F_E}$ is dense in $E$ and we set $$K_0:=\sup_{r\in[0,T]}\norm{F(r,x_0)}_E.$$ We notice that $K_0$ is finite by Hypothesis \ref{EU2}(\ref{EU2.2}-b). By \eqref{cji} and \eqref{lipJ} (both with $V=E$) and \eqref{boundR}, for every $k\in\N$ such that $k>\abs{\zeta}$ we have
\begin{align}
\norm{J_k(t,y_{s,x}^k(t)+f(t))}_E&\leq \norm{J_k(t,y_{s,x}^k(t)+f(t))-J_k(t,x_0)}_E+\norm{J_k(t,x_0)-x_0}_E+\norm{x_0}_E\nonumber\\
&\leq R+2\norm{x_0}_E+K_0+\abs{\zeta}\norm{x_0}_E=:R'.\label{stimaJk}
\end{align}
Recalling that $F_k(t,\cdot)=F(t,J_k(t,\cdot))$, by \eqref{boundRlim}, \eqref{stimaJk} and Hypothesis \ref{Hyp1}(2) applied with $R$ replaced by $R'$, we obtain
\begin{align*}
&\left|\ps{F_k(t,y_{s,x}^{k}(t)+f(t))-F(t,y_{s,x}(t)+f(t))}{h}_H\right|\\
 &\leq L_{R'}\norm{J_{k}(t,y_{s,x}^k(t)+f(t))-y_{ s,x}(t)-f(t)}_{H}\norm{h}_E,\ \ h\in E.
\end{align*}
Hence \eqref{convFk} follows by \eqref{convJ}.
\noindent By Hypothesis \ref{EU2}(\ref{EU2.2}-a), \eqref{vy1} (with $V=H$), \eqref{stidE} and \eqref{boundR}, there exists $M>0$, independent of $k$ and $t$, such that
\begin{equation}\label{unifEH}
\norm{F_k(t,y_{s,x}^{k}(t)+f(t))}_H\leq M,\qquad k>\abs{\zeta},\ t\in[s,T].
\end{equation} 
By Proposition \ref{sold}, for every $k>\abs{\zeta}$, $t\in [s,T]$ and $h\in{E}$, we have 
\[
\ps{y_{s,x}^{k}(t)}{h}_{H}=\ps{U(t,s)x}{h}_{H}+\int_s^t\ps{U(t,r)F_k(r,y_{s,x}^k(r)+f(r))}{h}_Hdr,\ \ h\in E.
\]
Letting $k\rightarrow +\infty$, by \eqref{contrazione}, \eqref{convFk}, \eqref{unifEH}, and the Dominated Convergence Theorem we obtain
\[
\ps{y_{s,x}(t)}{h}_{H}=\ps{U(t,s)x}{h}_{H}+\int_s^t\ps{U(t,r)F(r,y_{ s,x}(r)+f(r))}{h}_{H} dr,\ \ h\in E,
\]
and by the density of $E$ in $H$ the function $y_{s,x}$ solves the mild form of \eqref{eq}.
\noindent Lastly we prove that $y_{s,x}\in C_b((s,T];E)$. Let $s<t_1<t_2\leq T$. By the mild form of \eqref{eq} we have
\begin{align}\label{contEsol}
y_{s,x}(t_2)-y_{s,x}(t_1)&=\left[U(t_2,s)-U(t_1,s)\right]x+\int_{t_1}^{t_2}U(t_2,r)F(r,y_{s,x}(r)+f(r))dr\nonumber\\
&+\int_s^{t_1}\left[U(t_2,r)-U(t_1,r)\right]F(r,y_{s,x}(r)+f(r))dr.
\end{align}
By \eqref{boundRlim} and Hypothesis \ref{EU2}(\ref{EU2.2}-a) we have
\begin{equation}\label{unifEHlim}
\norm{F(r,y_{s,x}(r)+f(r))}_H\leq M_R,\qquad r\in(s,T],
\end{equation}
where $M_R$ is the constant given by Hypothesis \ref{EU2}(\ref{EU2.2}-a) with $R$ as in \eqref{boundR}.

The first term of \eqref{contEsol} converges to $0$ in $E$ as $t_2\rightarrow t_1$ by Hypothesis \ref{EU2}(\ref{EU2.5}-a). By \eqref{smoothing} and \eqref{unifEHlim} we have
\begin{align*}
\norm{\int_{t_1}^{t_2}U(t_2,r)F(r,y_{s,x}(r)+f(r))dr}_E\leq C_\theta M_R(1-\theta)^{-1}(t_2-t_1)^{1-\theta},
\end{align*}which converges to $0$ since $\theta<1$. Finally, by \eqref{smoothing} and \eqref{unifEHlim} we get 
\begin{align*}
\norm{\left[U(t_2,r)-U(t_1,r)\right]F(r,y_{s,x}(r)+f(r))}_E\leq 2C_\theta M_R(t_1-r)^{-\theta},
\end{align*}which belongs to $L^1((s,t_1))$. Hence the third term of \eqref{contEsol} converges to $0$ in $E$ by Hypothesis \ref{EU2}(\ref{EU2.5}-a) {  (applied to $U(t_2,t')-U(t_1,t')$ with $t'\in(r,t_1)$, since $U(t_i,r)=U(t_i,t')U(t',r)$ for $i=1,2$ and $U(t',r)$ maps $H$ into $E$ by \eqref{smoothing})} and the Dominated Convergence Theorem. Hence $y_{s,x}\in C_b((s,T];E)$ and \eqref{lipE} follows by \eqref{lipdE} and the lower semi-continuity of the norm.

\item \textbf{Case 2: Hypothesis \ref{Hyp2} holds true.}  
We start by proving \eqref{convE}. For every $k,h>\abs{\zeta}$ and $t\in [s,T]$
we have
\begin{align*}
\norm{y_{s,x}^{k}(t)-y_{s,x}^{h}(t)}_E&=\norm{\int^t_sU(t,r)\left(F_k(r,y_{s,x}^k(r)+f(r))-F_h(r,y_{s,x}^{h}(r)+f(r))\right)dr}_E,
\end{align*}
so by \eqref{contrazione} we get 
\begin{align*}
\norm{y_{s,x}^{k}(t)-y_{s,x}^{h}(t)}_E\leq \int^t_s\norm{F(r,J_k(r,y_{s,x}^k(r)+f(r)))-F(r,J_h(r,y^{h}_{s,x}(r)+f(r)))}_Edr
\end{align*}

By Hypotheses \ref{Hyp2}, $F(t,\cdot):E\longrightarrow E$ is locally Lipschitz continuous. By \eqref{cji}  and \eqref{stidE} there exists a constant $R:=R(\zeta,T,x)>0$ such that, for every $k,h>\abs{\zeta}$ and $t\in [s,T]$ we have
\begin{equation*}
\norm{J_k(t,y_{s,x}^k(t)+f(t))}_E+\norm{J_h(t,y_{s,x}^{h}(t)+f(t))}_E\leq R.
\end{equation*}
Let $L_R>0$ be the Lipschitz constant of $F(t,\cdot)$ on $B_E(0,R)$, which by Hypotheses \ref{Hyp2} can be chosen independent of $t\in[0,T]$. By \eqref{superyE} with ${\rm D}_F$ replaced by $E$, we obtain
{\small\begin{align*}
\norm{y_{s,x}^{k}(t)-y_{s,x}^{h}(t)}_E&
\leq L_R\int^t_s\norm{J_k(r,y_{s,x}^k(r)+f(r))-J_h(r,y_{s,x}^{h}(r)+f(r))}_Edr\\
&\leq L_R\int^t_s\norm{y_{s,x}^k(r)-y_{s,x}^{h}(r)}_Edr\\
&+L_R\left(\frac{1}{k}+\frac{1}{h}\right)\int^t_s\norm{F(r,y^{k}_{s,x}(r)+f(r))}_E+\norm{F(r,y_{s,x}^{h}(r)+f(r))}_Edr\\
&+\abs{\zeta}L_R\left(\frac{1}{k}+\frac{1}{h}\right)\int^t_s\norm{y^{k}_{s,x}(r)+f(r)}_E+\norm{y^{h}_{s,x}(r)+f(r)}_Edr.
\end{align*}}
By \eqref{stidE} there exists $R'>0$, independent of $k$ and $h$, such that $\norm{y_{s,x}^{k}(r)+f(r)}_E\leq R'$ for every $r\in[s,T]$, and the same bound holds with $h$ in place of $k$. Hence by \eqref{F0limitata} and Hypotheses \ref{Hyp2} we have
\[
\norm{F(r,y_{s,x}^{k}(r)+f(r))}_E\leq \sup_{t\in[0,T]}\norm{F(t,0)}_E+L_{R'}R',\qquad r\in[s,T],
\]
and the same bound holds with $h$ in place of $k$, so that the last two integrals are bounded by a constant independent of $k$ and $h$. Hence, for every $k,h>\abs{\zeta}$ and $t\in [s,T]$, by the Gr\"onwall inequality there exists a constant $C:=C(\zeta,x,s,T)$ such that
\begin{align*}
\norm{y_{s,x}^{k}(t)-y_{s,x}^{h}(t)}_E\leq C\left(\frac{1}{k}+\frac{1}{h}\right),
\end{align*}
so there exists $\widetilde{y}_{s,x}\in C((s,T],E)$ that verifies \eqref{convE}. Since $E$ is continuously embedded in $H$ we have $\widetilde{y}_{s,x}\equiv y_{s,x}$. Moreover, by \eqref{stidE}, \eqref{lipdE} and \eqref{convE} we obtain \eqref{stiE} and \eqref{lipE}.
To prove that $y_{s,x}$ is a mild solution of \eqref{eq} we still have to prove that for every $t\in [s,T]$ 
\begin{equation}\label{mildFFE}
\lim_{k\rightarrow +\infty}\int^t_s\norm{F_k(r,y_{s,x}^k(r)+f(r))-F(r,y_{ s,x}(r)+f(r))}_Edr=0.
\end{equation}
 By \eqref{cji}, \eqref{lipJ} (with $V=E$) and \eqref{convE} we have
\begin{align*}
\lim_{k\rightarrow+\infty}\norm{J_{k}(t,y_{s,x}^{k}(t)+f(t))- y_{s,x}(t)-f(t)}_E=0.
\end{align*}
Recalling that $F(t,\cdot):E\longrightarrow E$ is locally Lipschitz continuous we have
\begin{align}\label{convFE}
\lim_{k\rightarrow+\infty}\norm{F(t,J_{k}(t,y_{s,x}^{k}(t)+f(t)))-F(t,y_{s,x}(t)+f(t))}_E=0.
\end{align}
Finally by \eqref{convFE} and the Dominated Convergence theorem we obtain \eqref{mildFFE} and so $y_{s,x}$ is a mild solution of \eqref{eq}, for every $x\in E$.
\end{itemize}

\noindent \textbf{Uniqueness}

Let $u_{s,x}$ and $v_{s,x}$ be two mild solutions to \eqref{eq} in $C([s,T];H)\cap C_{ b}((s,T];E)$. For every $n\in\N$ we define the functions $u^n_{s,x}$ and $v^n_{s,x}$ given, for every $t\in [s,T]$, by
\begin{align*}
    &u_{s,x}^n(t):=U_{n}(t,s)x+\int_s^tU_{n}(t,r)F_n(r,u_{s,x}(r)+f(r))dr\\
     &v_{s,x}^n(t):=U_{n}(t,s)x+\int_s^tU_{n}(t,r)F_n(r,v_{s,x}(r)+f(r))dr,
\end{align*}
where $F_n$ are the Yosida approximations of $F$ given by \eqref{defyod} and $\{U_n(t,s)\}_{(s,t)\in\overline{\Delta}}$ are the approximations of  $\{U(t,s)\}_{(s,t)\in\overline{\Delta}}$ given by Hypotheses \ref{EU2}(\ref{EU2.5}-c).
Since $u_{s,x},v_{s,x}\in C([s,T];H)\cap C_{ b}((s,T];E)$ and $f$ is continuous with values in $E$, there exist $R,R_H>0$ such that, for every $r\in[s,T]$,
\begin{equation}\label{limunic}
\norm{u_{s,x}(r)+f(r)}_E+\norm{v_{s,x}(r)+f(r)}_E\leq R,\qquad \norm{u_{s,x}(r)+f(r)}_H+\norm{v_{s,x}(r)+f(r)}_H\leq R_H,
\end{equation}
hence by Hypothesis \ref{EU2}(\ref{EU2.2}-a) there exists $M_R>0$ such that
\begin{equation}\label{limunicF}
\norm{F(r,u_{s,x}(r)+f(r))}_H+\norm{F(r,v_{s,x}(r)+f(r))}_H\leq 2M_R,\qquad r\in[s,T].
\end{equation}
We first prove that
\begin{equation}\label{convdiag}
\lim_{n\rightarrow+\infty}\norm{u_{s,x}^n(t)-u_{s,x}(t)}_H=0,{ \qquad t\in[s,T],}
\end{equation}
and that the same holds with $v_{s,x}$ in place of $u_{s,x}$. Since $u_{s,x}$ is a mild solution to \eqref{eq}, for every $t\in[s,T]$ we have
\begin{align}
u_{s,x}^n(t)-u_{s,x}(t)&=\left[U_n(t,s)-U(t,s)\right]x\nonumber \\
&+\int_s^t\left[U_n(t,r)-U(t,r)\right]F(r,u_{s,x}(r)+f(r))dr \nonumber\\
&+\int_s^tU_n(t,r)\left[F_n(r,u_{s,x}(r)+f(r))-F(r,u_{s,x}(r)+f(r))\right]dr. \label{unmenou}
\end{align}
By \eqref{defnG}, \eqref{eq_YO} and \eqref{G2} (with $V=H$) we have
\[
F_n(r,x)-F(r,x)=G_n(r,x)-G(r,x)+\zeta[J_n(r,x)-x],
\]
so by \eqref{G1}, \eqref{G2} and \eqref{defnG} we obtain
\[
\norm{F_n(r,u_{s,x}(r)+f(r))-F(r,u_{s,x}(r)+f(r))}_H\leq 2\left(1+\frac{\abs{\zeta}}{n}\right)\left(2M_R+\abs{\zeta}R_H\right),\qquad r\in[s,T],
\]
which is bounded uniformly with respect to $n$ and $r$. { Fix $t\in[s,T]$.} The first term of \eqref{unmenou} converges to $0$ { by \eqref{yosidaE}}. Moreover, by \eqref{yosidaE}, \eqref{contrazione-n}, \eqref{contrazione}, \eqref{limunicF} and the Dominated Convergence Theorem, the second term of \eqref{unmenou} converges to $0$. { Lastly, fix $r\in[s,T]$ and set $w:=u_{s,x}(r)+f(r)\in E\subseteq{\rm D}_F$. By \eqref{G2} the sequence $\{G_n(r,w)\}_{n\in\N}$ is bounded in $H$; let $g\in H$ be the weak limit of one of its subsequences $\{G_{n_j}(r,w)\}_{j\in\N}$. By the dissipativity of $G(r,\cdot)$ in $H$, for every $v\in{\rm D}_F$ we have $\ps{G_{n_j}(r,w)-G(r,v)}{J_{n_j}(r,w)-v}_H\leq 0$ and, letting $j\rightarrow+\infty$, by \eqref{G1} we get $\ps{g-G(r,v)}{w-v}_H\leq 0$. Choosing $v\in{\rm D}_F$ such that $v-G(r,v)=w-g$, which exists by the $m$-dissipativity of $G(r,\cdot)$, we obtain $\norm{w-v}_H^2\leq 0$, so that $v=w$ and $g=G(r,w)$. Hence $G_n(r,w)\rightharpoonup G(r,w)$ weakly in $H$ and, since $\norm{G_n(r,w)}_H\leq \norm{G(r,w)}_H$ by \eqref{G2}, the convergence is strong. Therefore, by \eqref{G1}, $F_n(r,w)=G_n(r,w)+\zeta J_n(r,w)\rightarrow G(r,w)+\zeta w=F(r,w)$ in $H$, and by the above bound, \eqref{contrazione-n} and the Dominated Convergence Theorem the third term of \eqref{unmenou} converges to $0$,} and \eqref{convdiag} follows.
By the Lipschitz continuity of $F_n$, Theorem \ref{cont_case_thm} and Hypotheses \ref{EU2}(\ref{EU2.2}-b) the functions $u_{s,x}^n, v_{s,x}^n$ belong to $C^1([s,T];E)$ and
\begin{align}
    &\frac{d}{dt}u_{s,x}^n(t)=A_n(t)u_{s,x}^n(t)+F_n(t,u_{s,x}(t)+f(t))\label{dun}\\
     &\frac{d}{dt}v_{s,x}^n(t)=A_n(t)v_{s,x}^n(t)+F_n(t,v_{s,x}(t)+f(t))\label{dvn},
\end{align}
where $A_n(t)$ are the Yosida approximations of $A(t)$. Setting $z^n=u^n_{s,x}-v^n_{s,x}$ for every $n\in\N$, we have $z^n(s)=0$ and, by \eqref{dun} and \eqref{dvn}, for every $(s,t)\in\Delta$
\begin{align*}
\norm{z^n(t)}_H^2&=2\int_s^t\ps{A_n(r)z^n(r)}{z^n(r)}_Hdr\\
&+2\int_s^t\ps{F_n(r,u_{s,x}(r)+f(r))-F_n(r,v_{s,x}(r)+f(r))}{z^n(r)}_{H} dr,
\end{align*}
so by the dissipativity of $A_n(r)$ in $H$ we get
\begin{align*}
\norm{z^n(t)}_H^2\leq 2\int_s^t\ps{F_n(r,u_{s,x}(r)+f(r))-F_n(r,v_{s,x}(r)+f(r))}{z^n(r)}_{H} dr.
\end{align*}
{ Writing $z^n(r)=u_{s,x}(r)-v_{s,x}(r)+\left[z^n(r)-u_{s,x}(r)+v_{s,x}(r)\right]$,} by \eqref{supery} (with $V=H$ and $k=h=n$), { \eqref{vy1},} \eqref{limunic} and \eqref{limunicF} there exists $C=C(\zeta,R_H,M_R)>0$ such that
\begin{align*}
\norm{z^n(t)}_H^2\leq{}& 2\abs{\zeta}\int_s^t\norm{u_{s,x}(r)-v_{s,x}(r)}^2_Hdr+\frac{C}{n}(T-s)\\
&+{ C\int_s^t\norm{z^n(r)-u_{s,x}(r)+v_{s,x}(r)}_Hdr},\qquad (s,t)\in\Delta.
\end{align*}
{ By \eqref{convdiag}, the same with $v_{s,x}$ in place of $u_{s,x}$, and the Dominated Convergence Theorem the last integral vanishes as $n\rightarrow+\infty$, since by \eqref{contrazione-n}, \eqref{vy1}, \eqref{limunic} and \eqref{limunicF} the functions $z^n$ are bounded in $H$ uniformly with respect to $n$ and $r$.} Letting $n\longrightarrow +\infty$, by \eqref{convdiag} we obtain
\begin{align*}
\norm{u_{s,x}(t)-v_{s,x}(t)}_H^2\leq 2\abs{\zeta}\int_s^t\norm{u_{s,x}(r)-v_{s,x}(r)}^2_Hdr,\qquad (s,t)\in\Delta,
\end{align*}
where $\zeta$ is given by Hypotheses \ref{EU2}(\ref{EU2.2}-c). Finally, the Gr\"onwall inequality implies the uniqueness.
\end{proof}

\begin{remark}\label{rmkE}
   We note that only under Hypothesis \ref{Hyp2} the Hilbert space $H$ is not necessary. Indeed, by \eqref{convE}, the analysis of the problem \eqref{eq} can be restricted only to the space $E$. Due to this fact, in this case, we can drop the assumptions regarding $H$ in Hypothesis \ref{EU2}. Moreover, under Hypothesis \ref{Hyp2}, it is not necessary to approximate $F$. Indeed the locally Lipschitz continuity of $F:E\rightarrow E$ is sufficient to perform a contraction method as in Proposition \ref{soldt}.
\end{remark}

\begin{remark}\label{rmkT}
Theorem \ref{proyos} still holds if we replace the interval $[0,T]$ by $[-T,T]$.
\end{remark}

\begin{remark}
If we add a further Lipschitz continuous perturbation $G:E\rightarrow E$ in \eqref{eq}, then Propositions \ref{soldt}, \ref{sold} and Theorem \ref{proyos} still hold.
\end{remark}

\section{Well-posedness of the stochastic problem \eqref{stocpbintro}}\label{sect:WS}

Let $T>0$. In this section, we investigate the well-posedness of equation \eqref{stocpbintro}. First, we will prove that for every $x\in E$ equation \eqref{stocpbintro} has a unique mild solution in the following sense.

\begin{defn}
Given $x\in E$ and $s\in [0,T)$, a mild solution of problem \eqref{stocpbintro} is an $E$-valued stochastic process $\{X_{s,x}(t)\}_{t\in[s,T]}$ that satisfies 
\begin{equation}
X_{s,x}(t)=U(t,s)x+\int_s^tU(t,r)F(r,X_{s,x}(r))dr+Z_s(t),\ \ \mathbb{P}\mbox{-a.s.}\ t\in [s,T],\nonumber
\end{equation}
where $\{Z_s(t)\}_{t\in[s,T]}$ is the stochastic convolution process defined by 
\begin{equation}\label{convoluzionestocastica}
Z_s(t)=\int_s^tU(t,r)B(r)dW(r).
\end{equation}
\end{defn}

The idea is to exploit Theorem \ref{proyos} using a change of variable. To this end, we need to assume that the stochastic convolution process $\{Z_s(t)\}_{t\in[s,T]}$ has continuous trajectories $\mathbb{P}$-a.s.

\begin{hypothesis}\label{Hyp-F}
    Assume that either Hypothesis \ref{Hyp1} or Hypothesis \ref{Hyp2} holds and that the following conditions hold.
    \begin{enumerate}[{\rm (i)}]
        \item $\{B(t)\}_{t\in[0,T]}$ is a family of linear bounded operators on $H$ and for every $x\in H$ the mapping $t\in[0,T]\longmapsto B(t)x\in H$ is Borel measurable.
        \item  For every $s\geq 0$ and $p\geq 2$ the stochastic convolution process $\{Z_s(t)\}_{t\in[s,T]}$ defined by \eqref{convoluzionestocastica} belongs to $\mathcal{C}^p([s,T];E)$ and
        \begin{equation}\label{cond-F}
        \mathbb{E}\left[\sup_{t\in [s,T]}\norm{F(t,Z_s(t))}_E^p\right]<+\infty.
        \end{equation}
        
    \end{enumerate}

\end{hypothesis}

In Subsection \ref{sec:regulariti-stoc-conv} we prove a criterion to guarantee that the stochastic convolution process is regular enough in the case where $E=C(\mathcal{O})$ where $\mathcal{O}$ is a bounded subset of $\R^d$ with $d\in\N$. Now we can prove the main result of this section.

\begin{thm}\label{solMild}
Assume that Hypothesis \ref{Hyp-F} holds. For every $s\in [0,T)$ and $x\in E$, \eqref{stocpbintro} has a unique (probabilistic strong) mild solution $\{X_{s,x}(t)\}_{t\in [s,T]}\in \mathcal{C}^p([s,T];H)\cap\mathcal{C}^p((s,T];E)$ for every $p\geq 2$. Moreover for every $(s,t)\in\Delta$ and $x,z\in E$ we have
{\small
\begin{align}
&\norm{X_{s,x}(t)}_{H}\leq e^{\zeta (t-s)}\norm{x}_{H}+3\int_s^te^{\zeta (t-r)}\left[\norm{F(r,Z_s(r))}_{H}+\abs{\zeta}
\norm{Z_s(r)}_{H}\right]dr +\norm{Z_s(t)}_{H},\;\; \mathbb{P}\mbox{-a.s.}\label{stindXX}\\
&\norm{X_{s,x}(t)}_E\leq e^{\zeta (t-s)}\norm{x}_E+3\int_s^te^{\zeta (t-r)}\left[\norm{F(r,Z_s(r))}_E+\abs{\zeta}\norm{Z_s(r)}_E\right]dr+\norm{Z_s(t)}_E,\;\;\mathbb{P}\mbox{-a.s.}\label{stindEX}	\\
&\norm{X_{s,x}(t)-X_{s,z}(t)}_{H}\leq  e^{\zeta (t-s)}\norm{x-z}_{H},\;\;\mathbb{P}\mbox{-a.s.}\label{lipXX}\\
&\norm{X_{s,x}(t)-X_{s,z}(t)}_E\leq  e^{\zeta (t-s)}\norm{x-z}_E,\;\;\mathbb{P}\mbox{-a.s.}\label{lipEX}
\end{align}
}
where $\zeta$ is the constant defined in Hypothesis \ref{EU2}(\ref{EU2.2}).
\end{thm}
\begin{proof}
Fix $x\in E$ and $s\in [0,T)$. Fix an $H$-cylindrical Wiener process $\{W(t)\}_{t\in[0,T]}$ defined on a complete filtered probability space $(\Omega,\mathcal{F},\{\mathcal{F}_t\}_{t\in[0,T]},\mathbb{P})$. By Hypothesis \ref{Hyp-F}, there exists $\Omega_0\in \mathcal{F} $ such that $\mathbb{P}(\Omega_0)=1$ and for every $\omega\in\Omega_0$ the function $t\rightarrow Z_s(t)(\omega)$ is continuous from $[s,T]$ to $E$.\\

\noindent\textbf{Existence} Let $\{Y_{s,x}(t)\}_{t\in [s,T]}$ be the process defined for every $\omega\in\Omega_0$  by 
\begin{equation*}
Y_{s,x}(\cdot)(\omega):=y_{s,x}(\cdot),
\end{equation*}
where $y_{s,x}(\cdot)$ is the unique mild solution to \eqref{eq} with $f(\cdot):=Z_s(\cdot)(\omega)$ given by Theorem \ref{proyos}. Of course the process $\{X_{s,x}(t)\}_{t\in [s,T]}$ defined by
\begin{equation*}
X_{s,x}(t):=Y_{s,x}(t)+Z_s(t),
\end{equation*}
solves the mild form of \eqref{stocpbintro} $\mathbb{P}$-a.s. and, by Theorem \ref{proyos}, it verifies \eqref{stindXX}, \eqref{stindEX}, \eqref{lipXX} and \eqref{lipEX}; moreover, by \eqref{stindXX}, \eqref{stindEX} and Hypothesis \ref{Hyp-F} we have
\[
\mathbb{E}\left[\sup_{t\in [s,T]}\norm{X_{s,x}(t)}_H^p\right]<+\infty,\qquad \mathbb{E}\left[\sup_{t\in [s+\varepsilon,T]}\norm{X_{s,x}(t)}_E^p\right]<+\infty,\qquad \varepsilon>0,\ p\geq 2.
\]
Hence, if we prove that $\{Y_{s,x}(t)\}_{t\in [s,T]}$ is adapted to $\{\mathcal{F}_t\}_{t\in[0,T]}$, then also $\{X_{s,x}(t)\}_{t\in [s,T]}$ is adapted, so it belongs to $\mathcal{C}^p([s,T];H)\cap\mathcal{C}^p((s,T];E)$ for every $p\geq 2$, and the proof of existence is concluded.
 For every $n,k\in\N$ with $k>\abs{\zeta}$, we study the following SPDE
\begin{align*}
\begin{cases}
\dfrac{d}{dt}Y^{k,n}_{s,x}(t)=A_n(t)Y^{k,n}_{s,x}(t)+F_k(t,Y^{k,n}_{s,x}(t)+Z_s(t)), & (s,t)\in\Delta,\\
Y^{k,n}_{s,x}(s)=x\in E,
\end{cases}
\end{align*}
where $F_k$ and $A_n(t)$ are the Yosida approximations of $F$ and $A(t)$ (see Section \ref{sec:Yosida}). By Proposition \ref{soldt} and the same arguments used above, there exists a unique $E$-valued continuous process $\{Y^{k,n}_{s,x}(t)\}_{t\in  [s,T]}$ such that
\[
Y^{k,n}_{s,x}(t)=x+\int_s^tA_n(r)Y^{k,n}_{s,x}(r)dr+ \int_s^tF_k(r,Y^{k,n}_{s,x}(r)+Z_s(r))dr,\qquad \mathbb{P}\mbox{-a.s.}
\]
Moreover by \eqref{convdt} and \eqref{convdX} we have
\begin{equation}\label{convnk}
\lim_{k\rightarrow +\infty}\lim_{n\rightarrow +\infty}\norm{Y^{k,n}_{s,x}(t)-Y_{s,x}(t)}_H=0,\qquad t\in[s,T],\ \mathbb{P}\mbox{-a.s.},
\end{equation}
and, since the filtration $\{\mathcal{F}_t\}_{t\in[0,T]}$ is complete, if $\{Y^{k,n}_{s,x}(t)\}_{t\in [s,T]}$ is adapted to $\{\mathcal{F}_t\}_{t\in[0,T]}$ for every $n\in\N$ and $k>\abs{\zeta}$, then also $\{Y_{s,x}(t)\}_{t\in [s,T]}$ is adapted to $\{\mathcal{F}_t\}_{t\in[0,T]}$. For every $m\in\N$ we consider the processes $\{Y^{k,n,m}_{s,x}(t)\}_{t\in [s,T]}$ defined, for every $t\in [s,T]$, by
\begin{align*}
\begin{cases}
&Y^{k,n,0}_{s,x}(t)=x,\\
&\displaystyle{Y^{k,n,m+1}_{s,x}(t)=x+\int_s^tA_n(r)Y^{k,n,m}_{s,x}(r)dr+ \int_s^tF_k(r,Y^{k,n,m}_{s,x}(r)+Z_s(r))dr},\qquad m\geq 0.
\end{cases}
\end{align*}

Since $\{Z_s(t)\}_{t\in[s,T]}$ is adapted to $\{\mathcal{F}_t\}_{t\in[0,T]}$, by construction and by induction on $m$ the process $\{Y^{k,n,m}_{s,x}(t)\}_{t\in [s,T]}$ is adapted to $\{\mathcal{F}_t\}_{t\in[0,T]}$ for every $m\in\N$. Moreover since $F_k(t,\cdot):H\longrightarrow H$ is Lipschitz continuous and $A_n(t)\in\mathscr{L}(H)$, using the classical Picard iteration scheme it is not difficult to show that
\[
\lim_{m\rightarrow +\infty}\sup_{t\in[s,T]}\norm{Y^{k,n,m}_{s,x}(t)-Y^{k,n}_{s,x}(t)}_H=0,\qquad \mathbb{P}\mbox{-a.s.},
\]
and so, by the completeness of $\{\mathcal{F}_t\}_{t\in[0,T]}$, the process $\{Y^{k,n}_{s,x}(t)\}_{t\in [s,T]}$ is adapted to $\{\mathcal{F}_t\}_{t\in[0,T]}$ and $\{Y_{s,x}(t)\}_{t\in[s,T]}$ is adapted by \eqref{convnk}.\\

\noindent\textbf{Uniqueness}  If $\{X^{(1)}_{s,x}(t)\}_{t\in [s,T]}$ and $\{X^{(2)}_{s,x}(t)\}_{t\in [s,T]}$ are two mild solutions to \eqref{stocpbintro} defined on a complete and filtered probability space $(\Omega,\mathcal{F},\{\mathcal{F}_t\}_{t\in[0,T]},\mathbb{P})$, then for every $\omega\in\Omega_0$ both functions $t\in [s,T]\longrightarrow X^{(1)}_{s,x}(\omega)(t)-Z_s(\omega)(t)$ and $t\in [s,T]\longrightarrow X^{(2)}_{s,x}(\omega)(t)-Z_s(\omega)(t)$ are mild solutions to \eqref{eq} with $f=Z_s(\omega)$, so by the result of uniqueness proven in Theorem \ref{proyos} we have
\[
\sup_{t\in [s,T]}\norm{X^{(1)}_{s,x}(t)-X^{(2)}_{s,x}(t)}_H=0,\qquad \mathbb{P}\mbox{-a.s.}
\]
\end{proof}

\begin{remark}\label{rmkTE}
   Arguments similar to those in Remarks \ref{rmkE} and \ref{rmkT} are also valid for Theorem \ref{solMild}. In particular, about Remark \ref{rmkT}, we can replace $\{W(t)\}_{t\in [0,T]}$ with  $\{\overline{W}(t)\}_{t\in [-T,T]}$ defined by
   \[
   \overline{W}(t):=\begin{cases}
   W(t)\qquad &t\in [0,T]\\
   W'(-t)\qquad &t\in [-T,0)
   \end{cases},
   \]
   where $\{W'(t)\}_{t\in [0,T]}$ is another $H$-cylindrical Wiener process independent of $\{W(t)\}_{t\in [0,T]}$.
\end{remark}

Since $E$ is densely embedded in $H$, we are able to define a process $\{X_{s,x}(t)\}_{t\in [s,T]}$ for every $x\in H$ that is a solution to \eqref{stocpbintro} in the sense of Definition \ref{gensolmild}.
\begin{thm}\label{limmild}
Assume that Hypothesis \ref{Hyp-F} holds. For every $x\in H$ and $s\in [0,T)$ there exists a unique process $\{X_{s,x}(t)\}_{t\in [s,T]}\in\mathcal{C}^p([s,T];H)$ for every $p\geq 2$ such that for every $\{x_n\}_{n\in\N}\subseteq E$ converging to $x$ we have
\begin{align}\label{CX}
&\lim_{n{\rightarrow}+\infty}\left\{\mathbb{E}\left[\sup_{t\in [s,T]}\norm{X_{s,x_n}(t)-X_{s,x}(t)}_{H}^p\right]\right\}^{1/p}=0,\qquad p\geq 2,
\end{align}
where $\{X_{s,x_n}(t)\}_{t\in[s,T]}$ is the unique mild solution of \eqref{stocpbintro} with initial datum $x_n$. In addition for every $t\in [s,T]$ and $x,z\in {H}$ we have
{\footnotesize
\begin{align}
&\norm{X_{s,x}(t)}_{H}\leq e^{\zeta (t-s)}\norm{x}_{H}+3\int_s^te^{\zeta (t-r)}\left[\norm{F(r,Z_s(r))}_{H}+\abs{\zeta}\norm{Z_s(r)}_{H}\right]dr+\norm{Z_s(t)}_H,\;\;\mathbb{P}\mbox{-a.s.}\label{SX}\\
&\norm{X_{s,x}(t)-X_{s,z}(t)}_{H}\leq  e^{\zeta (t-s)}\norm{x-z}_{H},\;\;\mathbb{P}\mbox{-a.s.}\label{LX},
\end{align}
}
where $\zeta$ is the constant defined in Hypothesis \ref{EU2}(\ref{EU2.2}).
\end{thm}

\begin{proof}
\noindent Since $E$ is dense in $H$, for every $x\in{H}$ there exists a sequence $\{x_n\}_{n\in \N}\subseteq E$ such that
\[
\lim_{n{\rightarrow}+\infty} \norm{x_n-x}_{H}=0.
\]
By \eqref{lipXX}, for every $s\in [0,T)$, $p\geq 2$ and $n,m\in\N$ we have
\[
\left\{\mathbb{E}\left[\sup_{t\in [s,T]}\norm{X_{s,x_{n}}(t)-X_{s,x_{m}}(t)}_{H}^p\right]\right\}^{1/p}\leq e^{\abs{\zeta}(T-s)}\norm{x_n-x_m}_{H},
\]
hence
\[
\lim_{n,m{\to}+\infty}\left\{\mathbb{E}\left[\sup_{t\in [s,T]}\norm{X_{s,x_{n}}(t)-X_{s,x_{m}}(t)}_{H}^p\right]\right\}^{1/p}=0,\qquad p\geq 2.
\]
Since $\mathcal{C}^p([s,T];H)$ is complete, we denote by $\{X_{s,x}(t)\}_{t\in [s,T]}$ the limit of $\{X_{s,x_n}(t)\}_{t\in [s,T]}$ in $\mathcal{C}^p([s,T];H)$. We underline that by \eqref{lipXX} the limit $\{X_{s,x}(t)\}_{t\in [s,T]}$ does not depend on the sequence $\{x_n\}_{n\in \N}$, so it is the unique process in $\mathcal{C}^p([s,T];H)$ that verifies \eqref{CX}. Finally, \eqref{stindXX} and \eqref{lipXX} yield \eqref{SX} and \eqref{LX}.
\end{proof}

\begin{defn}\label{gensolmild}
For every $x\in{H}$ we call generalized mild solution to \eqref{stocpbintro} the limit $\{X_{s,x}(t)\}_{t\in [s,T]}$ given by Theorem \ref{limmild}. 
\end{defn}

Using the previous strong well-posedness results, we can define the transition Markov evolution operator associated with \eqref{stocpbintro} both in $B_b(H)$ and $B_b(E)$.\\
Assume that Hypothesis \ref{Hyp-F} holds. Let $s\in [0,T)$, $x\in H$ and let $\{X_{s,x}(t)\}_{t\in[s,T]}$ be the generalized mild solution to \eqref{stocpbintro} given by Theorems \ref{solMild} and \ref{limmild}. We define the families of operators $\{P_{s,t}\}_{(s,t)\in\overline{\Delta}}$ and $\{P_{s,t}^E\}_{(s,t)\in\overline{\Delta}}$ given by
\begin{align*}
    & P_{s,t}\varphi(x):=\mathbb{E}\left[\varphi\left(X_{s,x}(t)\right)\right],\qquad (s,t)\in\overline{\Delta},\, x\in H,\, \varphi\in B_b(H)\\
    & P^E_{s,t}\varphi(x):=\mathbb{E}\left[\varphi\left(X_{s,x}(t)\right)\right],\qquad (s,t)\in\overline{\Delta},\, x\in E,\, \varphi\in B_b(E)
\end{align*}

\begin{proposition}\label{transizione}
Assume that Hypothesis \ref{Hyp-F} holds. Then for every $x\in H$ and $0\leq s\leq r\leq t\leq T$ we have
\begin{align}\label{eq-flusso}
X_{s,x}(t)=X_{r,X_{s,x}(r)}(t),\qquad \mathbb{P}\mbox{-a.s.}
\end{align}
Moreover $\{P_{s,t}\}_{(s,t)\in\overline\Delta}$ and $\{P^E_{s,t}\}_{(s,t)\in\overline\Delta}$ are families of positive contractions on $B_b(H)$ and $B_b(E)$ respectively, they map $C_b(H)$ into itself and $C_b(E)$ into itself, and they verify the evolution property
\begin{align}\label{Chapman}
P_{s,r}P_{r,t}\varphi=P_{s,t}\varphi,\qquad 0\leq s\leq r\leq t\leq T.
\end{align}
Finally, for every $\varphi\in B_b(H)$ and $x\in E$ we have
\begin{align}\label{PPE}
P_{s,t}\varphi(x)=P^E_{s,t}\left(\varphi_{|E}\right)(x),\ \ \ (s,t)\in\overline{\Delta}.
\end{align}
\end{proposition}
\begin{proof}
We start by proving \eqref{eq-flusso}. Let $0\leq s\leq r\leq t\leq T$. Since $U(t,r)\in\mathscr{L}(H)$, by \eqref{convoluzionestocastica}, the additivity of the stochastic integral and \eqref{flow} we have
\begin{align}\label{flusso-Z}
Z_s(t)=\int_s^rU(t,q)B(q)dW(q)+\int_r^tU(t,q)B(q)dW(q)=U(t,r)Z_s(r)+Z_r(t),\qquad \mathbb{P}\mbox{-a.s.},
\end{align}
where we used that $U(t,q)=U(t,r)U(r,q)$ for every $q\in[s,r]$. Let $x\in E$. For every $\omega\in\Omega_0$ and $t\in[r,T]$, by the mild form of \eqref{stocpbintro}, \eqref{flow} and \eqref{flusso-Z} we have
\begin{align*}
X_{s,x}(\omega)(t)&=U(t,s)x+\int_s^tU(t,q)F(q,X_{s,x}(\omega)(q))dq+Z_s(\omega)(t)\\
&=U(t,r)\left[U(r,s)x+\int_s^rU(r,q)F(q,X_{s,x}(\omega)(q))dq+Z_s(\omega)(r)\right]\\
&+\int_r^tU(t,q)F(q,X_{s,x}(\omega)(q))dq+Z_r(\omega)(t)\\
&=U(t,r)X_{s,x}(\omega)(r)+\int_r^tU(t,q)F(q,X_{s,x}(\omega)(q))dq+Z_r(\omega)(t).
\end{align*}
Hence the function $t\in[r,T]\longrightarrow X_{s,x}(\omega)(t)-Z_r(\omega)(t)$ is a mild solution to \eqref{eq} on $[r,T]$ with $f=Z_r(\omega)$ and initial datum $X_{s,x}(\omega)(r)$, so \eqref{eq-flusso} follows by the uniqueness result of Theorem \ref{proyos}. If $x\in H$, we consider a sequence $\{x_n\}_{n\in\N}\subseteq E$ converging to $x$ in $H$; by \eqref{LX} we have
\[
\norm{X_{r,X_{s,x_n}(r)}(t)-X_{r,X_{s,x}(r)}(t)}_H\leq e^{\zeta(t-r)}\norm{X_{s,x_n}(r)-X_{s,x}(r)}_H,\qquad \mathbb{P}\mbox{-a.s.},
\]
so \eqref{eq-flusso} follows by \eqref{CX} letting $n\rightarrow+\infty$.

By \eqref{eq-flusso} and the same arguments of \cite[Proposition 9.14 and Corollary 9.15]{dap_zab_2014} we obtain \eqref{Chapman}. The positivity and the contractivity of $P_{s,t}$ and $P^E_{s,t}$ are immediate consequences of their definition, while the fact that they map $C_b(H)$ into itself and $C_b(E)$ into itself follows by \eqref{LX}, \eqref{lipEX} and the dominated convergence theorem. Finally, if $\varphi\in B_b(H)$ and $x\in E$, then $X_{s,x}(t)\in E$ $\mathbb{P}$-a.s. by Theorem \ref{solMild}, so $\varphi(X_{s,x}(t))=\varphi_{|E}(X_{s,x}(t))$ $\mathbb{P}$-a.s. and \eqref{PPE} follows.
\end{proof}

\subsection{Regularity of stochastic convolution}\label{sec:regulariti-stoc-conv}
Let $T>0$ and let $\mathcal{O}$ be a bounded open subset of $\R^d$ with at least a Lipschitz boundary. Let $\{W(t)\}_{t\in [0,T]}$ be an $L^2(\mathcal{O})$-cylindrical Wiener process defined by
\begin{equation}\label{W-cil}
    W(t):=\sum_{k\in\N}\beta_k(t)e_k,
\end{equation}
where $\{e_k\}_{k\in\N}$ is an orthonormal basis of $L^2(\mathcal{O})$ and $\beta_1,\beta_2,\ldots$ are real independent Brownian motions.

 {  In this subsection, we prove in Proposition \ref{Tempo-spazio} a sufficient condition ensuring the space-time continuity of the stochastic convolution process. We emphasize that analogous results, even in more general settings, are already available in the literature for the non-autonomous case; see \cite{pro_ver_2014,seid_1993,seid_2003,van_ver_2020,ver_zim_2008,ver_2010}. Here, we re-prove such a result in our setting. Moreover, we stress that, concerning the stochastic convolution process, the main novelty of this paper is Proposition \ref{prop:reg-convu}, where we provide explicit conditions on $\{A(t)\}_{t\geq 0}$ and $\{B(t)\}_{t\geq 0}$ ensuring that the assumptions of Proposition \ref{Tempo-spazio} are satisfied and, consequently, that the stochastic convolution process is space-time continuous. }

To do so we perform a non-autonomous version of the method presented in \cite[Section 2.2.2]{dap_2004}.  For every $p\geq 1$ and $\varepsilon> 0$, we denote by $W^{\varepsilon,p}(\mathcal{O})$ the classical fractional Sobolev space with respect to the Lebesgue measure. 

\begin{hypothesis}\label{hyp: Uts}
Assume that the following conditions hold.
\begin{enumerate}[{\rm (i)}]
\item $\{U(t,s)\}_{(s,t)\in\overline{\Delta}}$ is a strongly continuous evolution operator on $H$.
\item For all $p>2$, $\{U(t,s)\}_{(s,t)\in\overline{\Delta}}$ extends to a strongly continuous evolution operator on $L^p(\mathcal O)$.
\item There exists $\eta\geq{ 2}$ such that for all $\varepsilon\in(0,\eta/2)$ and $p\geq 2$ there exists $C_{\varepsilon, p}>0$ such that 
\begin{equation}\label{normasobF}
\norm{U(t,s)f}_{W^{\varepsilon,p}(\mathcal O)}\leq \frac{C_{\varepsilon,p}}{(t-s)^{\frac{\varepsilon}{\eta}}}\norm{f}_{L^{p}(\mathcal O)},\ \ (s,t)\in\Delta,\ p\geq 2.
\end{equation}
\item $\{B(t)\}_{t\in[0,T]}$ is a family of linear bounded operators on $L^2(\mathcal{O})$ and for every $f\in L^2(\mathcal{O})$ the mapping $t\in[0,T]\longmapsto B(t)f\in L^2(\mathcal{O})$ is Borel measurable.
\end{enumerate}
\end{hypothesis}

We start with the following lemma that is a generalization of \cite[Lemma 2.12]{dap_2004}.

\begin{lemma}\label{beppe_convoluzione}
Assume that Hypothesis \ref{hyp: Uts} holds. Let $T>0$, $s<T$, $\alpha\in \left(0,1\right)$ and $p>\max\{\frac{2}{\alpha},\frac{2d}{\eta\alpha}\}$. Let $f\in L^{p}([s,T]\times\mathcal O)$ and we set
 \begin{align}
 F_s(t,\xi)=\int_s^tU(t,r)(t-r)^{\alpha-1}f(r,\xi)\,dr,\ \ t\in[s,T],\ \xi\in\mathcal O.
 \end{align}
Then $F_s\in C([s,T]\times \mathcal O)$ and there exists $C=C(p,T,\alpha,\eta)>0$ such that

\begin{align}\label{stimasupF}
\sup_{(t,\xi)\in [s,T]\times\mathcal O} \abs{F_s(t,\xi)}^{p}\leq C\norm{f}^p_{L^{p}([s,T]\times\mathcal O)}.
\end{align}
\end{lemma}

\begin{proof}
Let $\varepsilon=\frac{1}{2}\alpha \eta$. Since $\varepsilon\in (0,\eta/2)$, by \eqref{normasobF} we have
\begin{align*}
\norm{F_s(t,\cdot)}_{W^{\varepsilon,p}(\mathcal O)}&\leq\int_s^t(t-r)^{\alpha-1}\norm{U(t,r)f(r,\cdot)}_{W^{\varepsilon,p}(\mathcal O)}\,dr\\
&\leq C_{\varepsilon, p} \int_s^t \frac{\norm{f(r,\cdot)}_{L^{p}(\mathcal O)}}{(t-r)^{\frac{\varepsilon}{\eta}+1-\alpha}}\,dr\leq C_{\varepsilon, p} \int_s^t \frac{\norm{f(r,\cdot)}_{L^{p}(\mathcal O)}}{(t-r)^{1-\frac{\alpha}{2}}}\,dr\\
&=C_{\varepsilon, p} \int_s^t\frac{1}{(t-r)^{1-\frac{\alpha}{2}}}\left(\int_\mathcal{O}\abs{f(r,\xi)}^pd\xi\right)^\frac{1}{p}dr.
\end{align*}
By the H\"older inequality, we obtain
\begin{align*}
\norm{F_s(t,\cdot)}_{W^{\varepsilon,p}(\mathcal O)}&\leq C_{\varepsilon, p} \left(\int_s^t\frac{1}{(t-r)^{(1-\frac{\alpha}{2})\frac{p}{p-1}}}dr\right)^{\frac{p-1}{p}}\left(\int_{[s,T]\times\mathcal{O}}\abs{f(r,\xi)}^pd\xi dr\right)^{\frac{1}{p}}.
\end{align*}
Since $p>\frac{2}{\alpha}$ then $(1-\frac{\alpha}{2})p(p-1)^{-1}<1$, there exists a constant $C=C(p,T,\alpha,\eta)>0$  such that
\begin{align}
\norm{F_s(t,\cdot)}_{W^{\varepsilon,p}(\mathcal O)}\leq C\norm{f}_{L^{p}([s,T]\times\mathcal O)}.
\end{align} 
Since $p>\frac{2d}{\alpha\eta}$ and $\varepsilon=\frac{1}{2}\alpha \eta$ then $\varepsilon p>d$, and so \eqref{stimasupF} is a consequence of the Sobolev embedding.
It remains to prove that $F_s\in C([s,T]\times\mathcal O)$. Let $s\leq t_1<t_2\leq T$. By \eqref{normasobF} we have
\begin{align}\label{pezzettone}
\norm{F_s(t_2,\cdot)-F_s(t_1,\cdot)}_{W^{\varepsilon,p}(\mathcal O)}&\leq\int_{t_1}^{t_2}(t_2-r)^{\alpha-1}\norm{U(t_2,r)f(r,\cdot)}_{W^{\varepsilon,p}(\mathcal O)}dr\nonumber \\
&+\int_s^{t_1}\norm{\left[U(t_2,r)-U(t_1,r)\right](t_1-r)^{\alpha-1}f(r,\cdot)}_{W^{\varepsilon,p}(\mathcal O)}dr\nonumber\\
&+\int_s^{t_1}\left[(t_1-r)^{\alpha-1}-(t_2-r)^{\alpha-1}\right]\norm{U(t_2,r)f(r,\cdot)}_{W^{\varepsilon,p}(\mathcal O)}dr.
\end{align}
\begin{align*}
\int_{t_1}^{t_2}(t_2-r)^{\alpha-1}\norm{U(t_2,r)f(r,\cdot)}_{W^{\varepsilon,p}(\mathcal O)}dr&\leq C\norm{f}_{L^p([t_1,t_2]\times\mathcal O)} {  (t_2-t_1)^{\frac{\alpha}{2}-\frac{1}{p}}},\\
\int_s^{t_1}\left[(t_1-r)^{\alpha-1}-(t_2-r)^{\alpha-1}\right]\norm{U(t_2,r)f(r,\cdot)}_{W^{\varepsilon,p}(\mathcal O)}{  dr}&\leq C\norm{f}_{L^p([s,T]\times\mathcal O)}{  (t_2-t_1)^{\frac{\alpha}{2}-\frac{1}{p}}},
\end{align*}
and they converge to $0$ as $t_2-t_1\rightarrow 0$. Since there exists $C_1>0$ such that
\begin{align*}
\norm{\left[U(t_2,r)-U(t_1,r)\right](t_1-r)^{\alpha-1}f(r,\cdot)}_{W^{\varepsilon,p}(\mathcal O)}\leq \frac{C_1}{(t_1-r)^{{  1-\frac{\alpha}{2}}}} \norm{f{  (r,\cdot)}}_{{  L^p(\mathcal O)}},
\end{align*}
{  which is integrable on $(s,t_1)$ by the H\"older inequality, and the left-hand side converges to $0$ as $t_2-t_1\rightarrow 0$ for every $r\in(s,t_1)$, by Hypothesis \ref{hyp: Uts}(ii), \eqref{normasobF} and interpolation,} the second term at the right-hand side of \eqref{pezzettone} converges to $0$ by the Dominated Convergence Theorem. Hence $t\in[s,T]\longrightarrow F_s(t,\cdot)\in W^{\varepsilon,p}(\mathcal O)$ is continuous and, since $\varepsilon p>d$, the Sobolev embedding of $W^{\varepsilon,p}(\mathcal O)$ into $C(\overline{\mathcal O})$ yields $F_s\in C([s,T]\times\mathcal O)$.
\end{proof}

Now we can show a sufficient condition for the space-time regularity of the stochastic convolution process.

\begin{proposition}\label{Tempo-spazio}
Assume that Hypothesis \ref{hyp: Uts} holds true and let $s\in[0,T)$. For $\alpha\in(0,\frac12)$, $\sigma\in(s,T]$ and $\xi\in\mathcal{O}$ we set
\begin{align}\label{def-lambda}
\lambda_{s,\sigma,\xi}:=\sum_{k=1}^\infty\int_s^\sigma(\sigma-r)^{-2\alpha}\left[U(\sigma,r)B(r) e_k\right]^2(\xi)\,dr\in[0,+\infty],
\end{align}
where $\{e_k\}_{k\in\N}$ is an orthonormal basis of $L^2(\mathcal{O})$. Assume that there exist $\alpha\in (0,\frac{1}{2})$ and $p_0>\max\{\frac{2}{\alpha},\frac{2d}{\eta\alpha}\}$ such that
\begin{align}\label{C-HS-int}
\lambda_{s,\cdot,\cdot}\in L^{r}\left((s,T)\times\mathcal{O}\right)\qquad\mbox{for every } r\geq \frac{p_0}{2}.
\end{align}
Then the stochastic convolution process $\{Z_s(t)\}_{t\in [s,T]}$ has a version which belongs to $\mathcal{C}^p([s,T],C(\overline{\mathcal{O}}))$ for every $p\geq 2$.
\end{proposition}

\begin{proof}
Fix $s\in [0,T)$ and let $\alpha$, $p_0$ be as in the statement. By the factorization formula \cite{dap_kwa_zab_1987,seid_1993} we have
\begin{align}
Z_s(t)=\int_s^tU(t,r)B(r)\,dW(r)=\frac{\sin(\pi \alpha)}{\pi}\int_s^tU(t,\sigma)(t-\sigma)^{\alpha-1}Y(\sigma)\,d\sigma,
\end{align}
where, by \eqref{W-cil},
\begin{align}\label{esp-Y}
Y(\sigma)(\xi):=\int_s^\sigma U(\sigma,r)(\sigma-r)^{-\alpha}\,B(r)dW(r)(\xi)=\sum_{k\in\N}\int_s^\sigma (\sigma-r)^{-\alpha}\left[U(\sigma,r)B(r)e_k\right](\xi)d\beta_k(r).
\end{align}
 By \eqref{C-HS-int}, $\lambda_{s,\sigma,\xi}<+\infty$ for almost every $(\sigma,\xi)\in(s,T)\times\mathcal{O}$. For such $(\sigma,\xi)$ the series in \eqref{esp-Y} is a series of independent real centered Gaussian random variables with variances given by $$q_k=\int_s^\sigma(\sigma-r)^{-2\alpha}[U(\sigma,r)B(r)e_k]^2(\xi)dr.$$ Moreover $\displaystyle{\sum_{k=1}^{\infty}q_k}$ converges to $\lambda_{s,\sigma,\xi}$ and the series in \eqref{esp-Y} converges in $L^2(\Omega)$ and almost surely. Hence, $Y(\sigma)(\xi)$ is a real Gaussian random variable with mean $0$ and variance $\lambda_{s,\sigma,\xi}$. In particular, for every $p'\geq 2$ there exists $c_{p'}>0$, depending only on $p'$, such that
\[
\mathbb{E}\left[\abs{Y(\sigma)(\xi)}^{p'}\right]=c_{p'}\,\lambda_{s,\sigma,\xi}^{p'/2}.
\]
 Fix $p\geq2$ and put $p':=\max\{p,p_0\}$. Then $p'>\max\{\frac{2}{\alpha},\frac{2d}{\eta\alpha}\}$ and $p'/2\geq p_0/2$. Taking a jointly measurable version of $Y$, by Tonelli's theorem and \eqref{C-HS-int},
\begin{align}\label{jensen_int}
\mathbb{E}\left[\int_s^T\int_\mathcal{O}\abs{Y(\sigma)(\xi)}^{p'}d\xi d\sigma\right]=c_{p'}\int_s^T\int_{\mathcal{O}}\lambda_{s,\sigma,\xi}^{p'/2}\,d\xi\, d\sigma<+\infty .
\end{align}
By \eqref{jensen_int}, $Y(\omega)\in L^{p'}([s,T]\times\mathcal{O})$ for $\mathbb{P}$-a.e. $\omega$. Hence, by Lemma \ref{beppe_convoluzione} (applied with $p'$ and $f=Y(\omega)$), $Z_s(\omega)=\frac{\sin(\pi\alpha)}{\pi}F_s(\omega)\in C([s,T]\times\overline{\mathcal{O}})$ and, by \eqref{stimasupF},
\[
\mathbb{E}\left[\sup_{(t,\xi)\in[s,T]\times\mathcal{O}}\abs{Z_s(t)(\xi)}^{p'}\right]\leq C\,\mathbb{E}\left[\norm{Y}^{p'}_{L^{p'}([s,T]\times\mathcal{O})}\right]<+\infty.
\]
Since $p\leq p'$ and $\mathbb{P}$ is a probability measure, Jensen's inequality gives
\[
\mathbb{E}\left[\sup_{(t,\xi)\in[s,T]\times\mathcal{O}}\abs{Z_s(t)(\xi)}^{p}\right]\leq\left(\mathbb{E}\left[\sup_{(t,\xi)\in[s,T]\times\mathcal{O}}\abs{Z_s(t)(\xi)}^{p'}\right]\right)^{p/p'}<+\infty,
\]
so that $Z_s\in\mathcal{C}^p([s,T],C(\overline{\mathcal{O}}))$ for every $p\geq2$.
Finally, $\{Z_s(t)\}_{t\in[s,T]}$ is adapted to $\{\mathcal{F}_t\}_{t\in[0,T]}$, since it is defined by \eqref{convoluzionestocastica} as a stochastic integral.
\end{proof}

\section{Example}\label{sect:exa}
Let $d\in\N$. In this section we present a large class of operators $\{A(t)\}_{t\in[0,T]}$, $\{B(t)\}_{t\in[0,T]}$ and functions $F$ that verify all the assumptions of this paper in the case $H=L^2(\mathcal{O})$, where $\mathcal{O}$ is a bounded subset of $\R^d$ with regular boundary.

The main results of this section are summarized in the following theorem

\begin{thm}\label{esempio_teorema}
Let $T>0$ and let $d\in\N$. Assume the following conditions hold.
\begin{enumerate}[\rm(i)]
\item  $\mathcal O$ is a bounded open subset of $\R^d$ with $C^2$ boundary.

\item For every $t\in[0,T]$, we consider the operator $\mathcal A(t)$ in divergence form given by 
\begin{equation}\label{operatore-div}
\mathcal{A}(t)\varphi(\xi)=\sum_{i,j=1}^d D_i\left(a_{ij}(t,\xi)D_{j}\varphi(\xi)\right)+a_0(t,\xi)\varphi(\xi),\quad t\in [0,T],\ \xi\in \mathcal{O},
\end{equation}
and the boundary operator $\mathcal{C}(t)$ given by
\begin{align}\label{boundary2}
\mathcal{C}(t)u=
\begin{cases}
\begin{aligned}
&u \ \ \quad \quad\quad \quad \quad\quad\quad \quad \quad \quad\quad \quad \quad \quad \mbox{(Dirichlet)},  \\[1ex]
&\displaystyle{\sum_{i,j=1}^da_{ij}(t,\xi)D_i u\,\nu_j}\quad \quad\quad\quad \quad \quad\, \ \ \mbox{(Neumann)},
\end{aligned}
\end{cases}
\end{align}
 where $\nu=(\nu_1,...,\nu_d)$ is the unit outer normal vector at the boundary of $\mathcal O$.
For every $t\in[0,T]$, the operator $A(t)$ is the realization in $L^2(\mathcal O)$ of $\mathcal A(t)$ with boundary operator $\mathcal C(t)$. We assume that
 \begin{itemize}
\item there exists $\rho\in(\frac{1}{2},1)$ such that for every $i,j\in\{1,...,d\}$   $$a_{ij}\in C_b^{\rho,2}\left(\R\times\overline{\mathcal{O}}\right),\ a_0\in C_b^{\rho,0}\left(\R\times\overline{\mathcal{O}}\right);$$ 
\item $a_{ij}=a_{ji}$ for every $i,j\in\{1,\ldots,d\}$ and there exists $\alpha_0>0$ such that 
\begin{align}\label{ellipticity}
\sum_{i,j=1}^da_{ij}(t,\xi)v_iv_j\geq \alpha_0\abs{v}^2,\ \ t\in[0,T],\ \xi\in\mathcal{O},\  v\in\R^d;
\end{align}
\item it holds
\begin{align}\label{segni}
\sup_{(t,\xi)\in[0,T]\times\overline{\mathcal{O}}}a_0(t,\xi)\leq 0.
\end{align}
\end{itemize}
\item There exists $q\geq 2$, $q>d$ such that $B(t)\in\mathscr{L}(L^2(\mathcal{O});L^q(\mathcal{O}))$ for every $t\in [0,T]$ and
\begin{equation}\label{thm: condB}
\sup_{t\in[0,T]}\norm{B(t)}_{\mathscr L(L^2(\mathcal{O});L^q(\mathcal{O}))}<+\infty,\qquad t\in [0,T].
\end{equation}

\item $F:[0,T]\times E \longrightarrow H$ is the Nemytskii operator defined by the formula
\begin{equation}\label{def:FF}
F(t,x)(\xi)=b(t,\xi,x(\xi)),\qquad t\in[0,T],\ x\in C(\overline{\mathcal{O}}),\ \xi\in \overline{\mathcal{O}},
\end{equation}
where $b:[0,T]\times\overline{\mathcal O}\times\R\longrightarrow\R$ is a polynomial function given by 
\begin{align}\label{def:b}
b(t,\xi,s)=-C_{2m+1}(t,\xi)s^{2m+1}+\sum_{k=0}^{2m}C_{k}(t,\xi)s^{k},\ \ \ t\in [0,T],\ \xi\in\overline{\mathcal{O}},\ s\in\R,
\end{align}
$m\in\N$ and $C_0,...,C_{2m+1}:[0,T]\times\overline{\mathcal{O}}\longrightarrow\R$ are continuous and bounded functions and there exists a constant $c>0$ such that $\displaystyle{\inf_{(t,\xi)\in [0,T]\times\overline{\mathcal{O}}}C_{2m+1}(t,\xi)>c}$.
\item We choose $H:=L^2(\mathcal O)$ and $E=C(\overline{\mathcal O})$ or $E=L^{2(2m+1)}(\mathcal O)$.   If $E=L^{2(2m+1)}(\mathcal{O})$, then we need $d\leq 3$.
\end{enumerate}
Then all the assumptions of Theorems \ref{proyos}, \ref{solMild} and \ref{limmild} are satisfied.
\end{thm}

\begin{remark}
    If $E=C(\overline{\mathcal{O}})$ then we can consider a more general function $F$ given by
 \[
F(t,x)(\xi):=b(t,\xi,x(\xi))+g\left(t,\max_{y \in B(\xi_0,\abs{\xi})}|x(y)|\right),\qquad t\in[0,T],\ x\in E,\ \xi\in \overline{\mathcal O},
\] 
where $b$ is given by \eqref{def:b}, $\xi_0\in \mathcal{O}$, $B(\xi_0,\abs{\xi})=\{y\in\overline{\mathcal{O}}\,:\, \abs{y-\xi_0}\leq \abs{\xi} \}$ and $g(t,\cdot)$ is a Lipschitz continuous function. In \cite{cer_dap_fla_2013} it is shown that $F(t,\cdot)$ satisfies Hypothesis \ref{Hyp2} in the special case $b\equiv 0$ and $d=1$. Of course $F$ is not well defined on $H$, so in such a case the generalized mild solution cannot be defined. See also Remarks \ref{rmkE} and \ref{rmkTE}.
\end{remark}

\begin{remark}\label{rmk:shift}
The sign condition \eqref{segni} is not restrictive. Indeed, if $a_0\in C_b^{\rho,0}\left(\R\times\overline{\mathcal{O}}\right)$ has arbitrary sign, we set
\[
\omega:=\sup_{(t,\xi)\in[0,T]\times\overline{\mathcal{O}}}a_0(t,\xi)^+<+\infty
\]
and we replace $A(t)$ and $F$ by $\widetilde{A}(t)=A(t)-\omega I$ and $\widetilde{F}=F+\omega I$, respectively. The family $\{\widetilde{A}(t)\}_{t\in[0,T]}$ is the realization in $L^2(\mathcal{O})$ of the operator obtained from \eqref{operatore-div} by replacing $a_0$ with $a_0-\omega$, which verifies \eqref{segni}, while $\widetilde{F}$ is still of the form \eqref{def:b} with $C_1$ replaced by $C_1+\omega$. Hence Theorem \ref{esempio_teorema} applies to $\widetilde{A}$ and $\widetilde{F}$, and the mild solutions of the associated problems coincide with those of the original ones.
\end{remark}

\begin{remark}\label{rmk:robin}
For the sake of simplicity we have confined ourselves to Dirichlet and Neumann boundary conditions. Robin boundary conditions
\[
\mathcal{C}(t)u=\sum_{i,j=1}^da_{ij}(t,\xi)D_i u\,\nu_j+b_0(t,\xi)u,
\]
with $b_0\in C_b^{\rho,1}\left(\R\times\overline{\mathcal{O}}\right)$, can be treated by the same arguments, provided that $b_0\geq 0$ on $[0,T]\times\partial\mathcal{O}$. We stress that, unlike \eqref{segni}, such a condition on $b_0$ cannot be removed by the argument of Remark \ref{rmk:shift}, since the shift $\widetilde{A}(t)=A(t)-\omega I$ does not affect the boundary term.
\end{remark}

\begin{remark}\label{RegO}
The regularity assumptions on the boundary of $\mathcal{O}$ and on the coefficients yield smoothing properties of $U(t,s)$ similar to analytic semigroups in the autonomous case (see e.g. \cite{acq_1988}). {  To the best of our knowledge, without} such regularity assumptions these results are not known. Instead, in special cases the regularity conditions on the boundary of $\mathcal{O}$ can be relaxed. For example, if for every $t\in [0,T]$ the operators $\mathcal A(t)$ and $\mathcal{C}(t)$ have the following forms
\begin{align}
    &\mathcal A(t)u=\sum_{i=1}^d a_i(t) D^2_{ii}u,\qquad a_i\in C{  ^{\rho}}([0,T]){  ,\ \min_{t\in[0,T]}a_i(t)>0},\\
    &\mathcal{C}(t)u=u,
\end{align}
{  with $\rho\in(\frac{1}{2},1)$,} it is sufficient that $\mathcal{O}$ has a Lipschitz boundary.
\end{remark}

\begin{rmk}\label{prop:carlo}
     Let $\{A(t)\}_{t\in[0,T]}$ be defined in (ii) of Theorem \ref{esempio_teorema} and  assume that 
\begin{equation}\label{stima-C}
\sup_{t\in[0,T]}\norm{B(t)}_{\mathscr L(L^2(\mathcal{O});W^{2\gamma,2}(\mathcal{O}))}<+\infty,
\end{equation} 
for some $\gamma\geq 0$. By the classical Sobolev embedding if $\gamma>\frac{d}{4}-\frac{1}{2}$ then there exists $q\geq 2$, $q>d$ such that $W^{2\gamma,2}(\mathcal{O})$ is continuously embedded in $L^q(\mathcal{O})$ and so that \eqref{thm: condB} holds.
\end{rmk}

\subsection{Families of operators \texorpdfstring{$\{A(t)\}_{t\in[0,T]}$}{TEXT} and \texorpdfstring{$\{B(t)\}_{t\in[0,T]}$}{TEXT}}
Let $\{A(t)\}_{t\in[0,T]}$ be defined in (ii) of Theorem \ref{esempio_teorema}. In \cite[Theorem 6.3]{acq_1988} and \cite[Examples 2.8 and 2.9]{sch_2004} it is shown that the family $\{A(t)\}_{t\in[0,T]}$ is associated to an evolution operator $\{U(t,s)\}_{(s,t)\in\overline{\Delta}}$ and they satisfy Hypothesis \ref{EU2} both with $E=C(\overline{\mathcal{O}})$ and $E=L^p(\mathcal{O})$ with $p\geq 2$.

In \cite[Theorem 6.3]{acq_1988} and \cite[Remark 6.4(b)]{acq_1988} it is shown that the family $\{A(t)\}_{t\in[0,T]}$ satisfies Hypothesis \ref{ATipotesi} in $E=C(\overline{\mathcal{O}})$ and in $E=L^p(\mathcal{O})$ for every $p\in(1,+\infty)$, respectively. Indeed, expanding the derivatives in \eqref{operatore-div} we have
\[
\mathcal{A}(t)\varphi=\sum_{i,j=1}^d a_{ij}(t,\cdot)D^2_{ij}\varphi+\sum_{j=1}^d\Big(\sum_{i=1}^dD_ia_{ij}(t,\cdot)\Big)D_j\varphi+a_0(t,\cdot)\varphi,
\]
and, since $a_{ij}\in C_b^{\rho,2}(\R\times\overline{\mathcal{O}})$, the coefficients $\sum_iD_ia_{ij}$ belong to $C^{\rho,0}_b(\R\times\overline{\mathcal{O}})$, so that $\mathcal{A}(t)$ and $\mathcal{C}(t)$ verify the assumptions of \cite[(6.1)--(6.6)]{acq_1988}. Consequently, by Theorem \ref{thm23paolo} there exists a unique evolution operator $\{U(t,s)\}_{(s,t)\in\overline{\Delta}}$ associated to $\{A(t)\}_{t\in[0,T]}$, and by \cite[Theorem 2.3]{acq_1988} and \cite[Proposition 4.5]{acq_ter_1987} the approximation property \eqref{yosidaE} holds in $\mathscr{L}(E)$ for both choices of $E$. We stress that \cite{acq_1988} provides the existence of the evolution operator and its regularity properties, while the remaining conditions in Hypothesis \ref{EU2} are verified below.

We now check that $A(t)$ is $m$-dissipative in $H=L^2(\mathcal{O})$ and that $A_E(t)$ is $m$-dissipative in $E$, for every $t\in[0,T]$.

\begin{itemize}
\item\textbf{Dissipativity in $H=L^2(\mathcal{O})$.} Let $t\in[0,T]$ and $u\in {\rm D}(A(t))$. Integrating by parts and using $\mathcal{C}(t)u=0$ we obtain
\begin{align}\label{diss-L2}
\ps{A(t)u}{u}_{H}=-\int_{\mathcal{O}}\sum_{i,j=1}^da_{ij}(t,\xi)D_ju(\xi)D_iu(\xi)\,d\xi+\int_{\mathcal{O}}a_0(t,\xi)u(\xi)^2\,d\xi\leq 0,
\end{align}
where the boundary term vanishes since $u=0$ on $\partial\mathcal{O}$ in the Dirichlet case and $\sum_{i,j}a_{ij}D_iu\,\nu_j=0$ on $\partial\mathcal{O}$ in the Neumann case, and where we used \eqref{ellipticity} and \eqref{segni}. Hence $A(t)$ is dissipative in $H$.

\item\textbf{Dissipativity in $E=L^{2(2m+1)}(\mathcal{O})$.} Set $q:=2(2m+1)$ and let $u\in {\rm D}(A_E(t))$. Since $z^*:=\norm{u}_E^{1-q}\abs{u}^{q-2}u\in\partial\norm{u}_E$, integrating by parts we obtain
\begin{align*}
\dscal{A_E(t)u}{z^*}{E}=\norm{u}_E^{1-q}\Bigg[&-(q-1)\int_{\mathcal{O}}\sum_{i,j=1}^da_{ij}(t,\xi)D_ju(\xi)D_iu(\xi)\abs{u(\xi)}^{q-2}d\xi\\
&+\int_{\mathcal{O}}a_0(t,\xi)\abs{u(\xi)}^{q}d\xi\Bigg]\leq 0,
\end{align*}
again by \eqref{ellipticity} and \eqref{segni}, the boundary term vanishing as above. Hence $A_E(t)$ is dissipative in $E$.

\item\textbf{Dissipativity in $E=C(\overline{\mathcal{O}})$.} Let $u\in {\rm D}(A_E(t))$ and let $\xi_0\in\overline{\mathcal{O}}$ be such that $\abs{u(\xi_0)}=\norm{u}_E$. If $\xi_0\in\mathcal{O}$, then $\xi_0$ is an interior extremum of $u$, so that $Du(\xi_0)=0$ and, by \eqref{ellipticity}, ${\rm sgn}(u(\xi_0))\sum_{i,j}a_{ij}(t,\xi_0)D^2_{ij}u(\xi_0)\leq 0$; together with \eqref{segni} this yields ${\rm sgn}(u(\xi_0))\mathcal{A}(t)u(\xi_0)\leq 0$. If $\xi_0\in\partial\mathcal{O}$, in the Dirichlet case $u(\xi_0)=0$ and hence $u\equiv 0$, while in the Neumann case the same conclusion follows from the Hopf boundary lemma, since $\sum_{i,j}a_{ij}(t,\xi_0)D_iu(\xi_0)\nu_j=0$. Since $z^*:={\rm sgn}(u(\xi_0))\delta_{\xi_0}\in\partial\norm{u}_E$, where $\delta_{\xi_0}$ is the Dirac measure at $\xi_0$, we obtain $\dscal{A_E(t)u}{z^*}{E}\leq 0$ and $A_E(t)$ is dissipative in $E$.

\item\textbf{$m$-dissipativity.} By Hypothesis \ref{ATipotesi}\textbf{(AT1)}, which holds in $H$ and in $E$ by the above, we have $(0,+\infty)\subseteq\rho(A(t))\cap\rho(A_E(t))$ for every $t\in[0,T]$. Hence the range of $\lambda I-A(t)$ is the whole of $H$ and the range of $\lambda I-A_E(t)$ is the whole of $E$ for every $\lambda>0$, so that $A(t)$ and $A_E(t)$ are $m$-dissipative and Hypothesis \ref{EU2}\eqref{EU2.3} holds. In particular \eqref{contrazione-n} and \eqref{contrazione} are satisfied.
\end{itemize}

Let $s\in [0,T)$ and $u_0\in L^2(\mathcal{O})$. We stress that the function $u_{s,u_0}:=U(\cdot,s)u_0$ is the unique classical solution of the following Cauchy problem
\begin{align}\label{parabolico2}
\begin{cases}
\dfrac{du}{dt}(t,\xi)=\mathcal{A}(t)u(t,\cdot)(\xi),\ \ &(t,\xi)\in(s,T)\times\mathcal{O},\\
\mathcal{C}(t)u(t,\cdot)(\xi)=0,\ \ &(t,\xi)\in(s,T)\times\partial\mathcal{O},\\
u(s,\cdot)=u_0(\cdot).
\end{cases}
\end{align}
In particular, for every $t\in [0,T]$, we have
\begin{align}
{\rm D}(A(t))&=\left\{u\in W^{2,2}(\mathcal{O}):\ \mathcal{C}(t)u=0\right\}.
\end{align}

We now prove that the stochastic convolution process $\{Z_s(t)\}_{t\in [s,T]}$ defined by \eqref{convoluzionestocastica} verifies Hypothesis \ref{Hyp-F}(ii).

\begin{proposition}\label{prop:reg-convu}
 Let $\{A(t)\}_{t\in[0,T]}$ be defined in (ii) of Theorem \ref{esempio_teorema}. Assume that there exists $q\geq 2$, $q>d$ such that $B(t)\in\mathscr{L}(L^2(\mathcal O),L^q(\mathcal O))$  and
\begin{equation}\label{stima-B}
C:=\sup_{t\in[0,T]}\norm{B(t)}_{\mathscr L(L^2(\mathcal{O});L^q(\mathcal{O}))}<+\infty.
\end{equation}  
Then for every $s\in [0,T)$ the stochastic convolution process $\{Z_s(t)\}_{t\in [s,T]}$ belongs to $\mathcal{C}^p([s,T],C(\overline{\mathcal{O}}))$, for every $p\geq 2$.
\end{proposition}
\begin{proof}
We want to apply Proposition \ref{Tempo-spazio}. First of all, we prove that Hypothesis \ref{hyp: Uts} holds. Under hypothesis (iii) of Theorem \ref{esempio_teorema} for all $t\in[0,T]$ we can consider the realization $A_p(t)$ in $L^p(\mathcal O)$ of $\mathcal{A}(t)$ with boundary conditions $\mathcal C(t)$. Since Hypothesis \ref{ATipotesi} are satisfied in $L^p(\mathcal O)$ for all $p\in[1,+\infty)$ (see \cite[Remark 2.8 and Remark 2.9]{sch_2004}), we can consider the evolution operator $\{U_p(t,s)\}_{(s,t)\in\overline{\Delta}}$ associated to the family $\{A_p(t)\}_{t\in[0,T]}$. By \cite[Theorem 3.3]{acq_1988}, for all $x\in L^p(\mathcal O)$ and $s\in [0,T)$, the Cauchy problem
\begin{align} \label{probLp}
\begin{cases}
u'(t)=A_p(t)u(t)\ \ t\in (s,T], \\
u(s)=x,
\end{cases}
\end{align} 
has a unique strong solution given by $u_{s,x}(t):=U_p(t,s)x$, for all $t\in [s,T]$ . Hence if $p>q$ then $U_p(t,s)x=U_q(t,s)x$ for all $x\in L^p(\mathcal O).$ For this reason, we omit the subindex $p$ and we simply write $U(t,s)$.

Moreover, by Theorem \ref{thm23paolo}(iii) and (v) there exist $C_p,K_p>0$ such that
\begin{align}
&\norm{U(t,s)}_{\mathscr{L}(L^p(\mathcal O))}\leq C_p,\ \ (s,t)\in\overline{\Delta},\label{Uts-Lp}\\
&\norm{A(t)U(t,s)}_{\mathscr{L}(L^p(\mathcal O))}\leq \frac{K_p}{t-s},\ \ (s,t)\in\Delta.\label{Uts-DA}
\end{align}
For every $t\in[0,T]$ and $\gamma\in(0,1)$, let $\left(L^p(\mathcal{O}),{\rm D}(A(t))\right)_{\gamma,{ p}}$ be the real interpolation space between $L^p(\mathcal{O})$ and ${\rm D}(A(t))$ endowed with the graph norm. By \eqref{Uts-Lp}, \eqref{Uts-DA} and { the interpolation inequality $\norm{y}_{\left(L^p(\mathcal{O}),{\rm D}(A(t))\right)_{\gamma,p}}\leq c_{\gamma,p}\norm{y}_{L^p(\mathcal{O})}^{1-\gamma}\norm{y}_{{\rm D}(A(t))}^{\gamma}$, $y\in{\rm D}(A(t))$ (see \cite[Chapter 1]{lun_1995})}, for every $\gamma\in(0,1)$ there exists $C_{p,\gamma}>0$, independent of $(s,t)$, such that
\begin{align}\label{Uts-interp}
\norm{U(t,s)}_{\mathscr{L}(L^p(\mathcal{O});\left(L^p(\mathcal{O}),{\rm D}(A(t))\right)_{\gamma,{ p}})}\leq \frac{C_{p,\gamma}}{(t-s)^\gamma},\ \ (s,t)\in\Delta.
\end{align}
Recalling that ${\rm D}(A(t))=\{u\in W^{2,p}(\mathcal{O}):\mathcal{C}(t)u=0\}$, for every $\gamma\in(0,1)$ with $2\gamma\notin\N$ and $2\gamma\neq \frac{1}{p}$ (Dirichlet case) or $2\gamma\neq 1+\frac{1}{p}$ (Neumann case) we have
\begin{align}\label{guidetti}
\left(L^p(\mathcal{O}),{\rm D}(A(t))\right)_{\gamma,{ p}}\hookrightarrow W^{2\gamma,p}(\mathcal{O}),
\end{align}
with embedding constant independent of $t\in[0,T]$. Hence, choosing $\varepsilon\in(0,\eta/2)$ with $\eta:=2$ and $\gamma:=\varepsilon/\eta$, by \eqref{Uts-interp} and \eqref{guidetti} there exists $C_{\varepsilon,p}>0$ such that
\begin{align}\label{eq:sob-sob}
\norm{U(t,s)f}_{W^{\varepsilon,p}(\mathcal O)}\leq \frac{C_{\varepsilon,p}}{(t-s)^{\varepsilon/\eta}}\norm{f}_{L^p(\mathcal O)},\ \ (s,t)\in\Delta,\ f\in L^p(\mathcal{O}),
\end{align}
that is \eqref{normasobF} holds with $\eta=2$, and so Hypothesis \ref{hyp: Uts} holds. It remains to prove {  \eqref{C-HS-int}}. { Since $q>d$, there exist $\alpha\in(0,\frac{1}{2})$ and $\varepsilon\in(\frac{d}{q},1)$ such that $2\alpha+\varepsilon<1$. By \eqref{eq:sob-sob} with $p=q$ and the Sobolev embedding $W^{\varepsilon,q}(\mathcal{O})\hookrightarrow C(\overline{\mathcal{O}})$, which holds since $\varepsilon q>d$, there exists $C_\varepsilon>0$ such that
\[
\norm{U(t,r)}_{\mathscr{L}(L^q(\mathcal{O});C(\overline{\mathcal{O}}))}\leq \frac{C_\varepsilon}{(t-r)^{\frac{\varepsilon}{2}}},\qquad r<t.
\]
Hence, by \eqref{stima-B}, for every $r<t$ and $\xi\in\overline{\mathcal{O}}$ the map $h\longmapsto[U(t,r)B(r)h](\xi)$ is a linear bounded functional on $L^2(\mathcal{O})$ with norm at most $CC_\varepsilon(t-r)^{-\frac{\varepsilon}{2}}$, and so
\[
\sum_{k=1}^{+\infty}[U(t,r)B(r)e_k]^2(\xi)\leq \frac{C^2C_\varepsilon^2}{(t-r)^{\varepsilon}}.
\]
Therefore} the function $\lambda_{s,\cdot,\cdot}$ defined in \eqref{def-lambda} is bounded on $(s,T)\times\mathcal{O}$ and \eqref{C-HS-int} holds true
for every $p_0>\max\{\frac{2}{\alpha},\frac{2d}{\eta\alpha}\}$. Hence, the statement follows by Proposition \ref{Tempo-spazio}.
\end{proof}

\subsection{Examples for \texorpdfstring{$F$}{TEXT}}

Let $F$ be defined by \eqref{def:FF} of Theorem \ref{esempio_teorema}. First, we note that $F$ verifies Hypothesis \ref{EU2}(\ref{EU2.2}), see \cite[Sect. 6.1.1]{cer_2001} for case $E=C(\overline{\mathcal{O}})$ and see \cite[Chap. 4]{dap_2004} for case $E=L^{2(2m+1)}(\mathcal O)$. Now we discuss Hypotheses \ref{Hyp1} and \ref{Hyp2}.

\begin{itemize}
\item If $m=0$ in \eqref{def:b} then we can take $H=E=L^2(\mathcal{O})$ and the calculations are trivial.
    \item {  If $E=L^{2(2m+1)}(\mathcal O)$, then $E$ is reflexive and $F$ verifies Hypothesis \ref{Hyp1}(2). Indeed, since
\begin{align}\label{stima-punt-b}
\abs{b(t,\xi,s_1)-b(t,\xi,s_2)}\leq c_m\Big(\max_{0\leq k\leq 2m+1}\norm{C_k}_\infty\Big)\abs{s_1-s_2}\left(1+\abs{s_1}^{2m}+\abs{s_2}^{2m}\right),
\end{align}
for every $t\in[0,T]$, $\xi\in\overline{\mathcal{O}}$ and $s_1,s_2\in\R$, where $c_m>0$ only depends on $m$, setting $q:=2(2m+1)$ and $\displaystyle{b:=\frac{2m+1}{m}}$ we have
\[
\frac{1}{2}+\frac{1}{q}+\frac{1}{b}=1,\qquad 2mb=q.
\]
Hence, by \eqref{stima-punt-b} and the H\"older inequality, for every $t\in[0,T]$, $h\in E$ and $x,y\in E$ we obtain
\begin{align*}
\abs{\ps{F(t,x)-F(t,y)}{h}_H}&=\int_\mathcal O \abs{b(t,\xi,x(\xi))-b(t,\xi,y(\xi))}\abs{h(\xi)}d\xi\\
&\leq C\int_\mathcal O \abs{x(\xi)-y(\xi)}\abs{h(\xi)}\left(1+\abs{x(\xi)}^{2m}+\abs{y(\xi)}^{2m}\right)d\xi\\
&\leq C\norm{x-y}_{L^2(\mathcal{O})}\norm{h}_{L^{q}(\mathcal{O})}\norm{1+\abs{x}^{2m}+\abs{y}^{2m}}_{L^{b}(\mathcal{O})}\\
&\leq C\left(\abs{\mathcal{O}}^{\frac{m}{2m+1}}+\norm{x}_{E}^{2m}+\norm{y}_{E}^{2m}\right)\norm{x-y}_{H}\norm{h}_E,
\end{align*}
where $C:=c_m\max_{0\leq k\leq 2m+1}\norm{C_k}_\infty$ and where we used that $\norm{\abs{x}^{2m}}_{L^b(\mathcal{O})}=\norm{x}^{2m}_{L^{2mb}(\mathcal{O})}=\norm{x}_E^{2m}$. Hence \eqref{FdualEH} holds with $L_R:=C\left(\abs{\mathcal{O}}^{\frac{m}{2m+1}}+2R^{2m}\right)$ for every $R>0$. Moreover 
\[
\norm{F(t,x)}_E\leq M\left(1+\norm{x}_{L^{2(2m+1)^2}(\mathcal{O})}^{2m+1}\right),\qquad x\in L^{2(2m+1)^2}(\mathcal{O}),\; t\in [0,T],
\]
for some constant $M>0$, that is, $L^{2(2m+1)^2}(\mathcal{O})\subseteq {\rm D}_{F_E}$. Furthermore, by standard Sobolev embeddings and \eqref{eq:sob-sob}, if $d\leq 3$ then \eqref{smoothing} holds with $\theta=\gamma\in(\frac{3}{4},1)$.
}
\item If  $E=C(\overline{\mathcal{O}})$, $F$ verifies Hypothesis \ref{Hyp2}, see  \cite[Sect. 6.1.1]{cer_2001}. Moreover,
\[
\norm{F(t,x)}_E\leq M\left(1+\norm{x}_E^{2m+1}\right),\qquad x\in C(\overline{\mathcal{O}}),\; t\in [0,T].
\]
\end{itemize}
Finally, since $\mathcal O$ is bounded, by Proposition \ref{prop:reg-convu} the process $Z_s$ has trajectories in $C([s,T];C(\overline{\mathcal{O}}))$ $\mathbb{P}$-a.s., and $C(\overline{\mathcal{O}})\subseteq L^{2(2m+1)^2}(\mathcal{O})$. Hence condition \eqref{cond-F} holds in both cases.

\section{An application to the multiplicative case}\label{sect:mult}

In this section we show that the deterministic results of Section \ref{sect: WD} can also be exploited to study a class of stochastic problems driven by a multiplicative noise, in the case $E=C(\overline{\mathcal{O}})$. In this section we present the non-autonomous counterpart of the well-posedness result in \cite{cer_2003}.

Assume Hypothesis \ref{Hyp2} holds, so that Theorem \ref{proyos} applies. We replace the family $\{B(t)\}_{t\in[0,T]}$ by a family $\{B(t,x)\}_{t\in[0,T],\,x\in E}$ and we consider the problem
\begin{align}\label{stocpb-mult}
\begin{cases}
dX(t)=\left[A(t)X(t)+F(t,X(t))\right]dt+B(t,X(t))dW(t),\qquad t\in (s,T],\\
X(s)=x\in E,
\end{cases}
\end{align}
where $\{W(t)\}_{t\in[0,T]}$ is the $L^2(\mathcal{O})$-cylindrical Wiener process defined in \eqref{W-cil} and the family $\{B(t,x)\}_{t\in[0,T],\,x\in E}$ satisfies the following assumptions.

\begin{hypothesis}\label{Hyp-B-mult}
Assume Hypothesis \ref{Hyp2} holds with $H=L^2(\mathcal{O})$ and $E=C(\overline{\mathcal{O}})$ where $\mathcal{O}$ is an open bounded subset of $\R^d$ with $d\in\N$. Assume that the following conditions hold.
\begin{enumerate}[{\rm (i)}]
\item $\{U(t,s)\}_{(s,t)\in\overline{\Delta}}$ is a strongly continuous evolution operator on $H$.
\item For all $p>2$, $\{U(t,s)\}_{(s,t)\in\overline{\Delta}}$ extends to a strongly continuous evolution operator on $L^p(\mathcal O)$.
\item There exists $\eta\geq{ 2}$ such that for all $\varepsilon\in(0,\eta/2)$ and $p\geq 2$ there exists $C_{\varepsilon, p}>0$ such that 
\begin{equation}\label{normasobFM}
\norm{U(t,s)f}_{W^{\varepsilon,p}(\mathcal O)}\leq \frac{C_{\varepsilon,p}}{(t-s)^{\frac{\varepsilon}{\eta}}}\norm{f}_{L^{p}(\mathcal O)},\ \ (s,t)\in\Delta,\ p\geq 2.
\end{equation}
\item The map $(t,x)\in[0,T]\times E\longmapsto B(t,x)h\in L^2(\mathcal{O})$ is measurable for every $h\in L^2(\mathcal{O})$. Moreover there exist $q\geq 2$ with $q>d$ and $L>0$ such that for every $t\in[0,T]$ and $x,y\in E$ we have $B(t,x)\in\mathscr{L}(L^2(\mathcal{O});L^q(\mathcal{O}))$ and
\begin{align}
&\sup_{t\in[0,T]}\norm{B(t,0)}_{\mathscr{L}(L^2(\mathcal{O});L^q(\mathcal{O}))}<+\infty,\label{B-0}\\
&\norm{B(t,x)-B(t,y)}_{\mathscr{L}(L^2(\mathcal{O});L^q(\mathcal{O}))}\leq L\norm{x-y}_E.\label{B-lip}
\end{align}
\end{enumerate}

\end{hypothesis}

By \eqref{B-0} and \eqref{B-lip}, setting $\Lambda:=\max\{\sup_{t\in[0,T]}\norm{B(t,0)}_{\mathscr{L}(L^2(\mathcal{O});L^q(\mathcal{O}))},L\}$, we have
\begin{equation}\label{B-lin}
\norm{B(t,x)}_{\mathscr{L}(L^2(\mathcal{O});L^q(\mathcal{O}))}\leq \Lambda\left(1+\norm{x}_E\right),\qquad t\in[0,T],\ x\in E.
\end{equation}

\begin{remark}
   We notice that the first three conditions of Hypothesis~\ref{Hyp-B-mult} are precisely the same as in Hypothesis~\ref{hyp: Uts}. Consequently, the family of operators $\{A(t)\}_{t\in{  [0,T]}}$ considered in Section~\ref{sect:exa} satisfies {  conditions (i)--(iii) of} Hypothesis \ref{Hyp-B-mult}.
\end{remark}

The linear growth \eqref{B-lin} is not enough to obtain global solutions if the nonlinearity $F$ is only $\zeta$-dissipative. We need one of the following two additional conditions.

\begin{hypothesis}\label{Hyp-F-mult}
There exist $\nu\geq 1$ and $M>0$ such that
\begin{equation}\label{F-growth}
\norm{F(t,z)}_E\leq M\left(1+\norm{z}_E^{\nu}\right),\qquad t\in[0,T],\ z\in E,
\end{equation}
and one of the following conditions holds.
\begin{enumerate}[{\rm (a)}]
\item There exists $\Lambda_0>0$ such that
\begin{equation}\label{B-sublin}
\norm{B(t,x)}_{\mathscr{L}(L^2(\mathcal{O});L^q(\mathcal{O}))}\leq \Lambda_0\left(1+\norm{x}_E^{1/\nu}\right),\qquad t\in[0,T],\ x\in E.
\end{equation}
\item There exist $a>0$ and $\beta\geq 0$ such that for every $t\in[0,T]$ and $x,h\in E$ there exists $h^*\in\partial\norm{h}_E$ with
\begin{equation}\label{F-SD}
\dscal{F(t,x+h)-F(t,x)}{h^*}{E}\leq -a\norm{h}_E^{\nu}+\beta\left(1+\norm{x}_E^{\nu}\right).
\end{equation}
\end{enumerate}
\end{hypothesis}

\begin{remark}\label{rmk:SD-example}
Condition (a) corresponds to the sublinear diffusion coefficients case of \cite[Theorem 5.3]{cer_2003} and {  covers} the case of a bounded $B$. Instead under condition (b), we can consider the case where $B$ has a linear growth ({  not necessarily} sublinear). Assumption \eqref{F-SD} is the abstract counterpart of \cite[(5.19)]{cer_2003} and it is satisfied by the Nemytskii operator $F$ given by \eqref{def:FF} \nuovo{ for every $m\in\N$, with $\nu=2m+1$, since $[(x+h)^{2m+1}-x^{2m+1}]\,{\rm sign}(h)\geq 2^{-2m}\abs{h}^{2m+1}$ for all $x,h\in\R$ and the lower order terms can be absorbed by the Young inequality}.
\end{remark}

\begin{example}\label{ex:mult}
Let $d\in\N$, $\mathcal{O}$ be an open bounded subset of $\R^d$ with smooth enough boundary, let $\Delta$ be the realization of the Laplacian with homogeneous Dirichlet boundary conditions in $H$, and let $\sigma:[0,T]\times \mathcal{O} \times\R\longrightarrow\R$ be a measurable function such that {  $\sup_{(t,\xi)\in[0,T]\times\mathcal{O}}\abs{\sigma(t,\xi,0)}<+\infty$ and}
\[
\abs{\sigma(t,\xi,s_1)-\sigma(t,\xi,s_2)}\leq L_\sigma\abs{s_1-s_2},\qquad t\in[0,T],\ \xi\in \mathcal{O},\ s_1,s_2\in\R.
\]
We set
\[
\left[B(t,x)h\right](\xi):=\sigma(t,\xi,x(\xi))[{  (-\Delta)^{-\gamma}} h](\xi),\qquad t\in[0,T],\ x\in E,\ h\in L^2(\mathcal{O}),\ \xi\in \mathcal{O},
\]
where $\gamma\in [0,+\infty)\cap (d/4-1/2,+\infty)$. Since $(-\Delta)^{-\gamma}\in\mathscr{L}(L^2(\mathcal{O});W^{2\gamma,2}(\mathcal{O}))$, there exist $q\geq 2$ with $q>d$ and $C_\gamma>0$ such that $\norm{{  (-\Delta)^{-\gamma}}h}_{L^q(\mathcal{O})}\leq C_\gamma\norm{h}_{L^2(\mathcal{O})}$ for every $h\in L^2(\mathcal{O})$.
Then for every $t\in{  [0,T]}$  we have
\[
\norm{B(t,0)}_{\mathscr{L}(L^2(\mathcal{O});L^q(\mathcal{O}))}\leq C_\gamma\sup_{(t,\xi)}\abs{\sigma(t,\xi,0)},
\]
and, for every $h\in L^2(\mathcal{O})$,
\begin{align*}
\norm{\left[B(t,x)-B(t,y)\right]h}_{L^q(\mathcal{O})}&\leq \sup_{\xi\in \mathcal{O}}\abs{\sigma(t,\xi,x(\xi))-\sigma(t,\xi,y(\xi))}\,\norm{{  (-\Delta)^{-\gamma}}h}_{L^q(\mathcal{O})}\\
&\leq L_\sigma C_\gamma\norm{x-y}_E\norm{h}_{L^2(\mathcal{O})}.
\end{align*}
Hence {  condition (iv) of Hypothesis \ref{Hyp-B-mult}} holds. If $F$ is defined by \eqref{def:FF} with \nuovo{$m\in\N$}, then Hypothesis \ref{Hyp-F-mult}(b) holds, so that $\sigma$ may grow linearly, e.g. $\sigma(t,\xi,s)=s$. If in addition $\abs{\sigma(t,\xi,s)}\leq c(1+\abs{s}^{1/(2m+1)})$, then we can consider a general function $F$ that verifies {  Hypothesis} \ref{Hyp2} {  and \eqref{F-growth} with $\nu=2m+1$}.
\end{example}

\begin{remark}\label{rmk:conf-sanz}
The structure of Example \ref{ex:mult} is similar to the one considered in
\cite{san_vui_2002,san_vui_2003}, where the noise term is of the form
$h(u(x,t))W(x,dt)$, with $h:\R\longrightarrow\R$ Lipschitz continuous, and
$\{W(t)\}_{t\in[0,T]}$ is the $L^2(\mathcal{O})$-valued Brownian motion
\[
W(\cdot,t)=\sum_{k=1}^{+\infty}\lambda_k^{1/2}f_k(\cdot)\beta_k(t).
\]
Here $\{\beta_k\}_{k\in\N}$ is a sequence of independent real standard Brownian
motions, $\{f_k\}_{k\in\N}$ is an orthonormal basis of $L^2(\mathcal{O})$
consisting of eigenfunctions of the covariance operator $Q$ of $W(1)$, and
$\{\lambda_k\}_{k\in\N}$ are the corresponding eigenvalues, i.e.
$Qf_k=\lambda_kf_k$ for every $k\in\N$. In
\cite[hypothesis \nuovo{(C$^d$)}]{san_vui_2003} it is assumed that there exists $s>d$
such that $f_k\in L^s(\mathcal{O})$ for every $k\in\N$ and
\[
\sum_{k\in\N}\lambda_k\norm{f_k}^2_{L^s(\mathcal{O})}<+\infty.
\]
The main difference with respect to \cite{san_vui_2002,san_vui_2003} is that here we
can consider the case where $\{W(t)\}_{t\in[0,T]}$ is merely
$L^2(\mathcal{O})$-cylindrical, see \eqref{W-cil}.
\end{remark}

We first prove an a priori estimate for the deterministic problem \eqref{eq} under the strong dissipativity condition \eqref{F-SD}.

\begin{lemma}\label{lemma:SD}
Assume that Hypothesis \ref{Hyp2}, \eqref{F-growth} and \eqref{F-SD} hold. Let $s\in[0,T)$, $x\in E$, $f\in C([s,T];E)$ and let $y$ be the unique mild solution to \eqref{eq} given by Theorem \ref{proyos}. Then
\begin{equation}\label{stiSD}
\norm{y(t)}_E\leq \norm{x}_E+K_0\Big(1+\sup_{r\in[s,t]}\norm{f(r)}_E\Big),\qquad t\in(s,T],
\end{equation}
where $K_0:=\left((\beta+M)/a\right)^{1/\nu}$.
\end{lemma}

\begin{proof}
For every $n\in\N$ and $k>\max\{\abs{\zeta},1\}$ let $y^{k,n}\in C^1([s,T];E)$ be the strict solution to \eqref{eqdt} given by Proposition \ref{soldt}. By \eqref{stidEth} and the continuity of $f$, there exists $R_1>0$, independent of $k$ and $n$, such that $\norm{y^{k,n}(r)+f(r)}_E\leq R_1$ for every $r\in[s,T]$. By \eqref{F-growth} and \eqref{cji} (with $V=E$ and ${\rm D}_F$ replaced by $E$), for every $r\in[0,T]$ and $z\in E$ with $\norm{z}_E\leq R_1$ we have
\[
\norm{J_k(r,z)-z}_E\leq \frac{M(1+R_1^\nu)+\abs{\zeta}R_1}{k}\leq M(1+R_1^\nu)+\abs{\zeta}R_1,
\]
where $J_k$ is given by \eqref{eq_YO1}. Hence, by Hypothesis \ref{Hyp2}, there exists $R_2>0$ such that, for every $n\in\N$, $k>\max\{\abs{\zeta},1\}$ and  $r\in{  [s,T]}$, we obtain
\begin{align}
\norm{F_k(r,y^{k,n}(r)+f(r))-F(r,y^{k,n}(r)+f(r))}_E&=\norm{F\left(r,J_k({  r,}y^{k,n}(r)+f(r))\right)-F(r,y^{k,n}(r)+f(r))}_E\notag\\
&\leq \frac{R_2}{k}{  .}\label{errSD}
\end{align}
{  Fix $t\in(s,T]$.} Hence, by \eqref{Ldiff-}, \eqref{eqdt}, \eqref{F-SD}, \eqref{F-growth} and \eqref{errSD}, for every $n\in\N$, $k>\max\{\abs{\zeta},1\}$ and  $r\in{  (s,t]}$, we get
\begin{align*}
\frac{d^-\norm{y^{k,n}}_E}{dr}(r)&\leq \dscal{F(r,y^{k,n}(r)+f(r))-F(r,f(r))}{h^*}{E}+\dscal{F(r,f(r))}{h^*}{E}+\frac{R_2}{k}\\
&\leq -a\norm{y^{k,n}(r)}_E^{\nu}+(\beta+M)\left(1+\sup_{r\in[s,t]}\norm{f(r)}_E\right)^\nu+\frac{{  R_2}}{k},
\end{align*}
{  where $h^*\in\partial\norm{y^{k,n}(r)}_E$ is given by \eqref{F-SD} applied with $f(r)$ and $y^{k,n}(r)$ in place of $x$ and $h$,} and so{  , since $\beta+M=aK_0^\nu$, by the convexity of the $\nu$-th power and} \eqref{varofcost} we conclude
\[
\norm{y^{k,n}(t)}_E\leq \norm{x}_E+K_0\Big(1+\sup_{r\in[s,t]}\norm{f(r)}_E\Big)+\frac{R_2(t-s)}{k},
\]
where $K_0:=\left((\beta+M)/a\right)^{1/\nu}$.
Letting first $n\rightarrow+\infty$ and then $k\rightarrow+\infty$, by \eqref{convdt} and \eqref{convE} we obtain \eqref{stiSD}.
\end{proof}

Before studying the well-posedness in the mild sense of \eqref{stocpb-mult}, we need to define the stochastic convolution process in the multiplicative case. For every $s\in[0,T)$, $p\geq 2$ and $X\in\mathcal{C}^p({ (s,T]};E)$ we set
\begin{equation}\label{conv-mult}
Z^X_s(t):=\int_s^tU(t,r)B(r,X(r))\,dW(r),\qquad t\in [s,T].
\end{equation}

We stress that, since $X$ is adapted and has continuous trajectories in $E$ { on $(s,T]$}, the process $r\longmapsto B(r,X(r))$ is predictable, and by \eqref{B-lin} and the computations below \eqref{conv-mult} is meaningful. The following result is the multiplicative counterpart of Propositions \ref{Tempo-spazio} and \ref{prop:reg-convu}.

\begin{proposition}\label{prop:conv-mult}
Assume that Hypothesis \ref{Hyp-B-mult} holds, let $\alpha\in\left(0,\frac{1}{2}\left(1-\frac{d}{q}\right)\right)$ and let $p_\alpha:=\max\{\frac{2}{\alpha},\frac{2d}{\eta\alpha}\}$. Then for every $s\in [0,T)$, $p>p_\alpha$ and $X\in\mathcal{C}^p({ (s,T]};E)$ the process $Z^X_s$ defined by \eqref{conv-mult} admits a version belonging to $\mathcal{C}^p([s,T];E)$. Moreover there exists $C_p>0$, independent of $s$, such that for every $X,X_1,X_2\in\mathcal{C}^p({ (s,T]};E)$ and $t\in[s,T]$ we have
\begin{align}
&\mathbb{E}\left[\sup_{r\in[s,t]}\norm{Z^X_s(r)}_E^p\right]\leq C_p\int_s^t\mathbb{E}\left[\norm{B(r,X(r))}^p_{\mathscr{L}(L^2(\mathcal{O});L^q(\mathcal{O}))}\right]dr,\label{conv-mult-i}\\
&\mathbb{E}\left[\sup_{r\in[s,t]}\norm{Z^{X_1}_s(r)-Z^{X_2}_s(r)}_E^p\right]\leq C_pL^p\int_s^t\mathbb{E}\left[\norm{X_1(r)-X_2(r)}_E^p\right]dr.\label{conv-mult-ii}
\end{align}
\end{proposition}

\begin{proof}
We set $\Phi(r):=B(r,X(r))${ , we fix $\varepsilon\in\left(\frac{d}{q},1-2\alpha\right)$} and { we set} $\vartheta:=2\alpha+{ \varepsilon}<1$. Let $Y$ be defined as in \eqref{esp-Y} with $\Phi$ instead of $B$. { Since $\eta\geq 2$ and $\varepsilon<1\leq\frac{\eta}{2}$, by \eqref{normasobFM} with $p=q$ and the Sobolev embedding $W^{\varepsilon,q}(\mathcal{O})\hookrightarrow C(\overline{\mathcal{O}})$, which holds since $\varepsilon q>d$, there exists $C_\varepsilon>0$ such that
\[
\norm{U(\sigma,r)}_{\mathscr{L}(L^q(\mathcal{O});C(\overline{\mathcal{O}}))}\leq \frac{C_\varepsilon}{(\sigma-r)^{\frac{\varepsilon}{2}}},\qquad r<\sigma.
\]
Hence, for every $r<\sigma$ and $\xi\in\overline{\mathcal{O}}$, the map $h\longmapsto[U(\sigma,r)\Phi(r)h](\xi)$ is a linear bounded functional on $L^2(\mathcal{O})$ with norm at most $C_\varepsilon(\sigma-r)^{-\frac{\varepsilon}{2}}\norm{\Phi(r)}_{\mathscr{L}(L^2(\mathcal{O});L^q(\mathcal{O}))}$, and so}
\begin{align}\label{C-HS-mult}
\sum_{k=1}^{+\infty}[U(\sigma,r)\Phi(r)e_k]^2(\xi)\leq \frac{{ C_\varepsilon^2}\norm{\Phi(r)}^2_{\mathscr{L}(L^2(\mathcal{O});L^q(\mathcal{O}))}}{(\sigma-r)^{{ \varepsilon}}},\qquad \mathbb{P}\mbox{-a.s.}
\end{align}
Hence, applying the Burkholder-Davis-Gundy inequality to the real valued martingale $r\longmapsto \int_s^{r}(\sigma-u)^{-\alpha}\left[U(\sigma,u)\Phi(u)dW(u)\right](\xi)$, $r\in[s,\sigma]$, for every $\sigma\in(s,T]$ and $\xi\in\mathcal{O}$ we get
\[
\mathbb{E}\left[\abs{Y(\sigma)(\xi)}^p\right]\leq c_p{ C_\varepsilon^{p}}\,\mathbb{E}\left[\left(\int_s^\sigma(\sigma-r)^{-\vartheta}\norm{\Phi(r)}^2_{\mathscr{L}(L^2(\mathcal{O});L^q(\mathcal{O}))}dr\right)^{p/2}\right].
\]

Since $p\geq 2$ and $\int_s^\sigma(\sigma-r)^{-\vartheta}dr\leq (1-\vartheta)^{-1}T^{1-\vartheta}$, by the Jensen inequality with respect to the finite measure $(\sigma-r)^{-\vartheta}dr$ on $[s,\sigma]$ we obtain
\[
\mathbb{E}\left[\abs{Y(\sigma)(\xi)}^p\right]\leq C\int_s^\sigma(\sigma-r)^{-\vartheta}\,\mathbb{E}\left[\norm{\Phi(r)}^p_{\mathscr{L}(L^2(\mathcal{O});L^q(\mathcal{O}))}\right]dr,
\]
with $C=C(p,q,d,\alpha,T)$. Integrating over $(\sigma,\xi)\in[s,t]\times\mathcal{O}$ and using the Fubini theorem we get
\begin{equation}\label{Y-int}
\mathbb{E}\left[\norm{Y}^p_{L^p([s,t]\times\mathcal{O})}\right]\leq \frac{C\abs{\mathcal{O}}T^{1-\vartheta}}{1-\vartheta}\int_s^t\mathbb{E}\left[\norm{\Phi(r)}^p_{\mathscr{L}(L^2(\mathcal{O});L^q(\mathcal{O}))}\right]dr.
\end{equation}
By the factorization formula and Lemma \ref{beppe_convoluzione} applied on $[s,t]$ we have
\[
\sup_{r\in[s,t]}\norm{Z^X_s(r)}_E^p\leq C(p,T,\alpha,\eta)\norm{Y}^p_{L^p([s,t]\times\mathcal{O})},\qquad \mathbb{P}\mbox{-a.s.}
\]
Together with \eqref{Y-int} this yields \eqref{conv-mult-i}, and the continuity of the trajectories follows by Lemma \ref{beppe_convoluzione}.\\
Let $X_1,X_2\in\mathcal{C}^p({ (s,T]};E)$ and set $\Delta\Phi(r):=B(r,X_1(r))-B(r,X_2(r))$. By the linearity of the stochastic integral, $Z^{X_1}_s-Z^{X_2}_s$ is given by \eqref{conv-mult} with $B(r,X(r))$ replaced by $\Delta\Phi(r)$. Since the previous computations only involve the norm of the integrand in $\mathscr{L}(L^2(\mathcal{O});L^q(\mathcal{O}))$, they still hold with $\Phi$ replaced by $\Delta\Phi$, and \eqref{conv-mult-ii} follows by \eqref{B-lip}.
\end{proof}

\begin{rmk}
    We note that, by \eqref{B-lin} and \eqref{conv-mult-i}, we have
\begin{equation}\label{conv-mult-i-bis}
\norm{Z^X_s}^p_{\mathcal{C}^p([s,T];E)}\leq 2^{p-1}C_p\Lambda^p(T-s)\left(1+\norm{X}^p_{\mathcal{C}^p({ (s,T]};E)}\right).
\end{equation}
\end{rmk}

Now we can study the mild formulation of \eqref{stocpb-mult} in the following sense.
\begin{defn}\label{solmild-mult}
Given $x\in E$ and $s\in [0,T)$, a mild solution of problem \eqref{stocpb-mult} is a process $\{X_{s,x}(t)\}_{t\in[s,T]}$ belonging to ${ \mathcal{C}^p([s,T];H)\cap\mathcal{C}^p((s,T];E)}$ for every $p\geq 2$ and such that
\begin{equation}\label{mild-mult}
X_{s,x}(t)=U(t,s)x+\int_s^tU(t,r)F(r,X_{s,x}(r))dr+Z^{X_{s,x}}_s(t),\qquad \mathbb{P}\mbox{-a.s.},\ t\in [s,T].
\end{equation}
\end{defn}

\begin{remark}\label{rmk-locality}
By \eqref{conv-mult}, if $\tau$ is a stopping time with values in $[s,T]$ and $X_1,X_2\in\mathcal{C}^p({ (s,T]};E)$ verify $X_1(t)=X_2(t)$ for every $t\in[s,\tau]$ $\mathbb{P}$-a.s., then $Z^{X_1}_s(t)=Z^{X_2}_s(t)$ for every $t\in[s,\tau]$ $\mathbb{P}$-a.s.
\end{remark}

\begin{thm}\label{solMild-mult}
Assume that Hypotheses \ref{Hyp-B-mult} and \ref{Hyp-F-mult} hold. Then for every $s\in [0,T)$ and $x\in E$ problem \eqref{stocpb-mult} has a unique mild solution $\{X_{s,x}(t)\}_{t\in [s,T]}$ in the sense of Definition \ref{solmild-mult}, and for every $p\geq 2$ there exists $C=C(p,T)>0$, independent of $s$ and $x$, such that
\begin{equation}\label{momenti-mult}
\mathbb{E}\left[\sup_{t\in[s,T]}\norm{X_{s,x}(t)}_E^p\right]\leq C\left(1+\norm{x}_E^p\right).
\end{equation}
Moreover, for every $t\in [s,T]$ we have
\begin{equation}\label{stiE-mult}
\norm{X_{s,x}(t)}_E\leq e^{\zeta (t-s)}\norm{x}_E+3\int_s^te^{\zeta (t-r)}\left[\norm{F(r,Z^{X_{s,x}}_s(r))}_E+\abs{\zeta}\norm{Z^{X_{s,x}}_s(r)}_E\right]dr+\norm{Z^{X_{s,x}}_s(t)}_E,
\end{equation}
$\mathbb{P}$-a.s., where $\zeta$ is the constant given by Hypothesis \ref{EU2}(\ref{EU2.2}-c).
\end{thm}

Before starting the proof of Theorem \ref{solMild-mult} we prove the following standard lemma necessary for the localisation argument that we {  are} going to use.

\begin{lemma}\label{lemma:loc-uniq}
Assume that Hypothesis \ref{Hyp-B-mult} holds. Let $s\in[0,T)$, $x\in E$, $p>p_\alpha$, where $p_\alpha$ is given by Proposition \ref{prop:conv-mult}, and let $G:[0,T]\times E\longrightarrow E$ be a continuous function such that, for some ${  L_G}>0$,
\[
\norm{G(t,z_1)-G(t,z_2)}_E\leq {  L_G}\norm{z_1-z_2}_E,\qquad t\in[0,T],\ z_1,z_2\in E.
\]
Let $\tau$ be a stopping time with values in $[s,T]$ and let $X_1,X_2\in\mathcal{C}^p({ (s,T]};E)$ be such that, for $i=1,2$ and every $t\in[s,T]$,
\begin{equation}\label{mild-G}
X_i(t)=U(t,s)x+\int_s^tU(t,r)G(r,X_i(r))dr+Z^{X_i}_s(t),\qquad \mathbb{P}\mbox{-a.s. on }\{t\leq\tau\}.
\end{equation}
Then $X_1(t)=X_2(t)$ for every $t\in[s,\tau]$, $\mathbb{P}$-a.s.
\end{lemma}

\begin{proof}
The statement follows directly by the {  Lipschitz continuity} of $G$ and \eqref{conv-mult-ii}.
\end{proof}

\begin{proof}[Proof of Theorem \ref{solMild-mult}]
Fix $s\in [0,T)$, $x\in E$ and $p>p_\alpha$, where $p_\alpha$ is given by Proposition \ref{prop:conv-mult}. It is enough to prove the statement for such $p$, since for $2\leq p'\leq p$ we have ${ \mathcal{C}^{p}([s,T];H)\cap\mathcal{C}^{p}((s,T];E)}\subseteq{ \mathcal{C}^{p'}([s,T];H)\cap\mathcal{C}^{p'}((s,T];E)}$ and \eqref{momenti-mult} for $p'$ follows from the one for $p$ by the H\"older inequality. \\
We proceed by a standard localization argument, so first we prove the well-posedness of the equation with $F$ replaced by {  a} suitable truncated version and we prove an a priori estimate independent of the truncation. For every $R>0$ let $\theta_R:E\longrightarrow E$ be defined by
\[
\theta_R(z):=\begin{cases} z, &\norm{z}_E\leq R,\\[1ex] \dfrac{R}{\norm{z}_E}\,z, &\norm{z}_E> R,\end{cases},\qquad F^R(t,z):=F(t,\theta_R(z)),\qquad t\in{  [0,T]},\ z\in E{  .}
\]
 By Hypothesis \ref{Hyp2} the function $F^R:[0,T]\times E\longrightarrow E$ is bounded and there exists $\Theta_R>0$ such that
\begin{equation}\label{lip-FR}
\norm{F^R(t,z_1)-F^R(t,z_2)}_E\leq \Theta_R\norm{z_1-z_2}_E,\qquad t\in[0,T],\ z_1,z_2\in E.
\end{equation}
 For every $\sigma\in [s,T)$ and every $E$-valued $\mathcal{F}_\sigma$-measurable random variable $y$ with $\mathbb{E}\left[\norm{y}_E^p\right]<+\infty$, we consider the Volterra operator $\Gamma^\sigma_R:\mathcal{C}^p({ (\sigma,T]};E)\longrightarrow\mathcal{C}^p({ (\sigma,T]};E)$ defined by
\[
\Gamma^\sigma_R(X)(t):=U(t,\sigma)y+\int_\sigma^tU(t,r)F^R(r,X(r))dr+Z^X_\sigma(t),\qquad t\in[\sigma,T].
\]
By \eqref{contrazione}, the boundedness of $F^R$ and \eqref{conv-mult-i-bis}, the map $\Gamma^\sigma_R$ is well defined. Moreover, by \eqref{contrazione}, \eqref{lip-FR} and \eqref{conv-mult-ii}, for every $X_1,X_2\in\mathcal{C}^p({ (\sigma,T]};E)$ we have
\[
\norm{\Gamma^\sigma_R(X_1)-\Gamma^\sigma_R(X_2)}_{\mathcal{C}^p({ (\sigma,T]};E)}\leq \left[\Theta_R(T-\sigma)+C_p^{1/p}L(T-\sigma)^{1/p}\right]\norm{X_1-X_2}_{\mathcal{C}^p({ (\sigma,T]};E)}.
\]
Hence there exists $T_0=T_0(R,p)>0$, such that $\Gamma^\sigma_R$ is a contraction on $\mathcal{C}^p({ (\sigma,T]};E)$ {  whenever $T-\sigma\leq T_0$}. Splitting $[s,T]$ into finitely many subintervals of length at most {  $T_0$} and arguing as in the proof of Proposition \ref{soldt}, we obtain a unique process $\{X^R_{s,x}(t)\}_{t\in [s,T]}\in\mathcal{C}^p({ (s,T]};E)$ such that
\begin{equation}\label{mild-R}
X^R_{s,x}(t)=U(t,s)x+\int_s^tU(t,r)F^R(r,X^R_{s,x}(r))dr+Z^{X^R_{s,x}}_s(t),\qquad \mathbb{P}\mbox{-a.s.},\ t\in [s,T].
\end{equation}

 Now we can prove an estimate for $\{X^R_{s,x}(t)\}_{t\in [s,T]}$ independent of $R$. Let $R>\norm{x}_E$ and {  consider} the following stopping time
\[
\tau_R:=\inf\left\{t\in [s,T]:\ \norm{X^R_{s,x}(t)}_E> R\right\}\wedge T,\qquad \widehat{X}_R(t):=X^R_{s,x}(t\wedge\tau_R),\quad t\in[s,T].
\]
Fix $\omega\in\Omega_0$ with $\tau_R(\omega)>s$, {  where $\Omega_0\in\mathcal{F}$, with $\mathbb{P}(\Omega_0)=1$, is such that $t\in[s,T]\longmapsto Z^{\widehat{X}_R}_s(\omega)(t)\in E$ is continuous for every $\omega\in\Omega_0$ (see Proposition \ref{prop:conv-mult})}. Since $\theta_R(z)=z$ whenever $\norm{z}_E\leq R$, by \eqref{mild-R} the function
\[
t\in[s,\tau_R(\omega)]\longmapsto X^R_{s,x}(\omega)(t)-Z^{\widehat{X}_R}_s(\omega)(t)\in E
\]
is the mild solution to \eqref{eq} on $[s,\tau_R(\omega)]$ with $f:=Z^{\widehat{X}_R}_s(\omega)$ and initial datum $x$. We distinguish the two cases of Hypothesis \ref{Hyp-F-mult}.

\noindent\textbf{Case (a).} By \eqref{stiE} (with $\tau_R(\omega)$ in place of $T$) we get
\begin{align}\label{sti-tau}
\norm{X^R_{s,x}(t)}_E\leq e^{\zeta (t-s)}\norm{x}_E&+3\int_s^{t}e^{\zeta (t-r)}\left[\norm{F(r,Z^{X^R_{s,x}}_s(r))}_E+\abs{\zeta}\norm{Z^{X^R_{s,x}}_s(r)}_E\right]dr\notag\\
&+\norm{Z^{X^R_{s,x}}_s(t)}_E,\qquad t\in[s,\tau_R],\qquad  \mathbb{P}\mbox{-a.s.}
\end{align}
By \eqref{F-growth}, there exists $K_1=K_1(\zeta,M,T)>0$ such that
\[
\sup_{r\in[s,t\wedge\tau_R]}\norm{X^R_{s,x}(r)}_E\leq K_1\left(1+\norm{x}_E+\sup_{r\in[s,t]}\norm{Z^{\widehat{X}_R}_s(r)}_E^\nu\right),\qquad t\in[s,T],\ \mathbb{P}\mbox{-a.s.}
\]
Since $\nu p>p_\alpha$, by \eqref{conv-mult-i} with $\nu p$ in place of $p$ and \eqref{B-sublin}, for every $t\in[s,T]$ we have
\begin{align*}
\mathbb{E}\left[\sup_{r\in[s,t]}\norm{Z^{\widehat{X}_R}_s(r)}_E^{\nu p}\right]&\leq C_{\nu p}\Lambda_0^{\nu p}\int_s^t\mathbb{E}\left[\left(1+\norm{\widehat{X}_R(r)}_E^{1/\nu}\right)^{\nu p}\right]dr\\
&\leq 2^{\nu p-1}C_{\nu p}\Lambda_0^{\nu p}\int_s^t\left(1+\mathbb{E}\left[\sup_{h\in[s,r]}\norm{\widehat{X}_R(h)}_E^p\right]\right)dr.
\end{align*}
Hence there exists $C_1=C_1(p,\nu,\zeta,M,\Lambda_0,T)>0$, independent of $R$, such that
\begin{equation}\label{gronwall-R}
\mathbb{E}\left[\sup_{h\in[s,t]}\norm{\widehat{X}_R(h)}_E^p\right]\leq C_1\left(1+\norm{x}_E^p\right)+C_1\int_s^t\mathbb{E}\left[\sup_{h\in[s,r]}\norm{\widehat{X}_R(h)}_E^p\right]dr,\qquad t\in[s,T].
\end{equation}

\noindent\textbf{Case (b).} By Lemma \ref{lemma:SD} (with $\tau_R(\omega)$ in place of $T$) we get
\begin{equation}\label{sti-tau-SD}
\norm{X^R_{s,x}(t)}_E\leq \norm{x}_E+K_0\Big(1+\sup_{r\in[s,t]}\norm{Z^{X^R_{s,x}}_s(r)}_E\Big)+\norm{Z^{X^R_{s,x}}_s(t)}_E,\qquad t\in[s,\tau_R],\ \mathbb{P}\mbox{-a.s.},
\end{equation}
so that
\[
\sup_{r\in[s,t\wedge\tau_R]}\norm{X^R_{s,x}(r)}_E\leq \norm{x}_E+K_0+(K_0+1)\sup_{r\in[s,t]}\norm{Z^{\widehat{X}_R}_s(r)}_E,\qquad t\in[s,T],\ \mathbb{P}\mbox{-a.s.}
\]
By \eqref{conv-mult-i} and \eqref{B-lin}, for every $t\in[s,T]$ we have
\[
\mathbb{E}\left[\sup_{r\in[s,t]}\norm{Z^{\widehat{X}_R}_s(r)}_E^{p}\right]\leq 2^{p-1}C_p\Lambda^p\int_s^t\left(1+\mathbb{E}\left[\sup_{h\in[s,r]}\norm{\widehat{X}_R(h)}_E^p\right]\right)dr,
\]
and so \eqref{gronwall-R} holds with a constant $C_1=C_1(p,K_0,\Lambda,T)>0$ independent of $R$.\\

\noindent In both cases (a) and (b) of Hypothesis \ref{Hyp-F-mult}, by \eqref{gronwall-R} and the Gr\"onwall inequality there exists $C_2=C_2(p,T)>0$, independent of $R$, $s$ and $x$, such that
\begin{equation}\label{sti-unif-R}
\mathbb{E}\left[\sup_{t\in [s,\tau_R]}\norm{X^R_{s,x}(t)}_E^p\right]\leq C_2\left(1+\norm{x}_E^p\right),\qquad R>\norm{x}_E.
\end{equation}

Now we need to prove that the stopping time $\tau_R$ does not explode. By the definition of $\tau_R$, the Chebyshev inequality and \eqref{sti-unif-R} we get
\[
\mathbb{P}\left(\tau_R<T\right)\leq\frac{1}{R^p}\,\mathbb{E}\left[\sup_{t\in [s,\tau_R]}\norm{X^R_{s,x}(t)}_E^p\right]\leq\frac{C_2(1+\norm{x}_E^p)}{R^p}\xrightarrow[R\rightarrow+\infty]{}0.
\]
Moreover, if $\norm{x}_E<R<R'$ then $F^{R}=F^{R'}$ on $\overline{B_E(0,R)}$. Since $\norm{X^R_{s,x}(t)}_E\leq R$ for every $t\in[s,\tau_R]$, by \eqref{mild-R} both $X^{R}_{s,x}$ and $X^{R'}_{s,x}$ verify \eqref{mild-G} with $G=F^{R'}$ and $\tau=\tau_R\wedge\tau_{R'}$, so that by Lemma \ref{lemma:loc-uniq} we obtain $X^{R}_{s,x}(t)=X^{R'}_{s,x}(t)$ for every $t\in[s,\tau_{R}\wedge\tau_{R'}]$, $\mathbb{P}$-a.s., and consequently $\tau_{R}\leq\tau_{R'}$ $\mathbb{P}$-a.s.
Therefore $\{\tau_R\}_{R>\norm{x}_E}$ is non-decreasing and $\tau_R\uparrow T$ $\mathbb{P}$-a.s. as $R\rightarrow+\infty$. Setting
\[
X_{s,x}(t):=X^R_{s,x}(t),\qquad t\in[s,\tau_R],
\]
we obtain a well defined continuous adapted process $\{X_{s,x}(t)\}_{t\in [s,T]}$ which verifies \eqref{mild-mult}, by Remark \ref{rmk-locality}. Since $\mathbf{1}_{\{\tau_R=T\}}\uparrow 1$ $\mathbb{P}$-a.s., by \eqref{sti-unif-R} and the Fatou lemma we have
\[
\mathbb{E}\left[\sup_{t\in[s,T]}\norm{X_{s,x}(t)}_E^p\right]\leq C_2\left(1+\norm{x}_E^p\right),
\]
which gives \eqref{momenti-mult}; in particular $X_{s,x}\in{ \mathcal{C}^p([s,T];H)\cap\mathcal{C}^p((s,T];E)}$ { (the continuity in $H$ at $t=s$ follows by \eqref{mild-mult})}. Letting $R\rightarrow+\infty$ in \eqref{sti-tau} we obtain \eqref{stiE-mult}.

Finally let $\{X^{(1)}_{s,x}(t)\}_{t\in [s,T]}$ and $\{X^{(2)}_{s,x}(t)\}_{t\in [s,T]}$ be two mild solutions in the sense of Definition \ref{solmild-mult}, and for every $R>\norm{x}_E$ let
\[
\sigma_R:=\inf\left\{t\in [s,T]:\ \norm{X^{(1)}_{s,x}(t)}_E> R\ \mbox{or}\ \norm{X^{(2)}_{s,x}(t)}_E> R\right\}\wedge T.
\]
Since both processes belong to $\mathcal{C}^p({ (s,T]};E)$, arguing as above we get $\sigma_R\uparrow T$ $\mathbb{P}$-a.s. as $R\rightarrow+\infty$. Since $F^R=F$ on $\overline{B_E(0,R)}$, both processes verify \eqref{mild-G} with $G=F^R$ and $\tau=\sigma_R$, hence, by Lemma \ref{lemma:loc-uniq}, we obtain $X^{(1)}_{s,x}(t)=X^{(2)}_{s,x}(t)$ for every $t\in[s,\sigma_R]$, $\mathbb{P}$-a.s. Letting $R\rightarrow+\infty$ the uniqueness follows.
\end{proof}

\subsubsection*{Acknowledgments}
The first author is a member of GNAMPA-INdAM.


\end{document}